\documentclass[11pt]{amsart}

\usepackage{amsmath,amssymb,amsthm,mathtools,mathrsfs}
\usepackage[letterpaper,margin=1in]{geometry}
\usepackage{microtype}
\usepackage[colorlinks=true,linkcolor=blue,citecolor=blue,urlcolor=blue]{hyperref}

\allowdisplaybreaks
\numberwithin{equation}{section}
\newtheorem{theorem}{Theorem}[section]
\newtheorem{lemma}[theorem]{Lemma}
\newtheorem{corollary}[theorem]{Corollary}
\newtheorem{proposition}[theorem]{Proposition}
\theoremstyle{definition}
\newtheorem{definition}{Definition}
\theoremstyle{remark}
\newtheorem{remark}[theorem]{Remark}

\newcommand{\C}{\mathbb C}
\newcommand{\Hh}{\mathbb H}
\newcommand{\N}{\mathbb N}
\newcommand{\Q}{\mathbb Q}
\newcommand{\R}{\mathbb R}
\newcommand{\Z}{\mathbb Z}
\newcommand{\Eodd}{\mathcal E^{\mathrm o}}
\newcommand{\Eeven}{\mathcal E^{\mathrm e}}
\newcommand{\Cone}{\mathcal C}
\newcommand{\ve}{\varepsilon}
\newcommand{\1}{\mathbf 1}
\newcommand{\sgnstd}{\operatorname{sgn}}

\title[Nahm sums and indefinite theta vectors]{Indefinite theta vectors and generalized Nahm sums of arbitrary rank}
\author[Dandan Chen]{Dandan Chen}
\address{Department of Mathematics, Shanghai University, People's Republic of China}
\address{Newtouch Center for Mathematics of Shanghai University, Shanghai, People's Republic of China}
\email{mathcdd@shu.edu.cn}
\author[Siyu Yin]{Siyu Yin$^{*}$}
\address{Department of Mathematics, East China Normal University, People's Republic of China}
\email{siyuyin0113@126.com, 52300155017@stu.ecnu.edu.cn}
\thanks{$^{*}$Corresponding author.}
\subjclass[2010]{11P84, 11F03, 11F27}
\keywords{Generalized Nahm sums, Bailey pairs, indefinite theta functions,
vector-valued modular transformations}
\date{}
\hypersetup{pdftitle={Indefinite theta vectors and generalized Nahm sums of arbitrary rank},pdfauthor={Dandan Chen, Siyu Yin}}

\begin{document}
\begin{abstract}
We construct indefinite theta representations for a family of signed
generalized Nahm sums of arbitrary rank \(r\geq2\). After explicit
normalization and boundary correction, one signed branch gives the
even components of a \((2r-1)\)-component holomorphic theta vector
associated with a lattice of signature \((1,1)\). Euler-type multiple
sums provide the odd components, yielding formulas uniform in both
the rank and the component index. Using Zwegers' theory, we determine
the weight-one modular transformation laws of the real-analytic
completion, with an explicit finite sine matrix for the
\(S\)-transformation. We decompose the nonholomorphic correction into
a finite sum of products of unary theta functions and one-dimensional error-function
series, and compute the antiholomorphic derivative of the completion.
Together with the completed transformation laws, these formulas give
the inhomogeneous transformation law of the holomorphic vector. We
also evaluate a linear combination of two signed branches as a difference of Jacobi product quotients. This identity and further
evaluations for two families of symmetrizable matrices and their duals,
combined with known results, give multiple-sum representations for
every component of four associated product vectors. These representations
are finite combinations of generalized Nahm sums with the prescribed
matrices and symmetrizers; in the dual families, half-period translates are absorbed into
signs and constant phase factors.
\end{abstract}
\maketitle

\section{Introduction and main results}

Rogers--Ramanujan identities relate multiple $q$-hypergeometric
series to infinite products and modular functions. Nahm sums provide
a natural setting for studying such identities: their quadratic
exponents arise in fermionic character formulas, while their
modularity is related to dilogarithm identities and the Bloch group
\cite{Nahm1994,Nahm1995,Nahm2007,Zagier2007}. In a vector-valued
setting, an explicit product evaluation raises a further question.
Can all components of the associated vector be represented by
multiple sums without changing the prescribed quadratic matrix and
symmetrizer?

Throughout, $\N:=\{0,1,2,\ldots\}$,
$\Hh:=\{\tau\in\C:\operatorname{Im}(\tau)>0\}$, and
$q=e^{2\pi i\tau}$, where $\tau=x+iy\in\Hh$.
We interpret $q^\alpha$ as $e^{2\pi i\alpha\tau}$ for real
$\alpha$. We use the standard notation
\[
 (a;q)_n:=\prod_{k=0}^{n-1}(1-aq^k),\qquad
 (a_1,\ldots,a_t;q)_n:=\prod_{k=1}^t(a_k;q)_n
\]
where $n\in\mathbb{N}\cup \{\infty\}$ and $|q|\textless 1$. Empty sums and products are $0$ and $1$, respectively. If
$\mathbf d:=(d_1,\ldots,d_r)^T\in\Z_{>0}^r$ and
$\mathbf D:=\operatorname{diag}(d_1,\ldots,d_r)$, the generalized
Nahm sum introduced by Mizuno \cite[Eq.~(5)]{Mizuno2025} is
\begin{equation}
 \widetilde f_{\mathbf A,\mathbf b,c,\mathbf d}(q)
 :=\sum_{\mathbf n\in\N^r}
 \frac{q^{\frac12\mathbf n^T\mathbf A\mathbf D\mathbf n+
                 \mathbf n^T\mathbf b+c}}
 {\prod_{a=1}^r(q^{d_a};q^{d_a})_{n_a}},
 \label{eq:generalized-nahm}
\end{equation}
where $\mathbf A\in\Q^{r\times r}$,
$\mathbf b\in\Q^r$, $c\in\Q$, and $\mathbf A\mathbf D$ is
symmetric positive definite. The ordinary Nahm sum
$f_{\mathbf A,\mathbf b,c}$ is the case
$\mathbf d=(1,\ldots,1)^T$. The data are called a modular triple
in the ordinary case, or a modular quadruple in the generalized
case, when the corresponding series is a modular function. A signed
sum has an additional factor $\prod_a\varepsilon_a^{n_a}$, with
$\varepsilon_a\in\{1,-1\}$, in its summand.

Zagier \cite[Section~3A]{Zagier2007} determined the seven rank-one
modular triples. In higher rank, Vlasenko and Zwegers
\cite{VlasenkoZwegers2011} gave counterexamples to the proposed
equivalence in Nahm's conjecture. Explicit evaluations include the
rank-two and rank-three identities of Cherednik and Feigin
\cite{CherednikFeigin2013}, Wang
\cite{WangRankTwoZagier,WangRankThreeZagier}, and Cao, Rosengren, and
Wang \cite{CaoRosengrenWang2024}. For symmetrizable matrices, Mizuno
\cite{Mizuno2025} proposed modularity and vector-valued
transformation conjectures; product evaluations and transformation
formulas for these low-rank families were established by B.~Wang
and L.~Wang
\cite{WangWangRankThreeI,WangWangRankThreeII,WangWangRankTwo}.
More recently, Du, Ji, Shen, and Xu \cite{DuJiShenXu2026} obtained
three families of modular generalized Nahm sums of arbitrary rank
$r\geq2$, with indices $(2,\ldots,2,1)$ and $(1,\ldots,1,2)$,
and studied the associated vector-valued transformations.

These evaluations do not provide multiple-sum representations for
every component of the four product vectors considered in
\cite{DuJiShenXu2026}; the missing ranges are specified in their
Problem~1.13. Related signed sums lead to a second issue. Their
Hecke-type double sums involve an indefinite quadratic form, so a
holomorphic $q$-series must be distinguished from its real-analytic
completion. Moreover, changing from a summation cone to a shifted
lattice can introduce boundary terms. Thus a component formula alone
does not determine either the full theta vector or its modular
transformation law.

We address these questions by constructing a $(2r-1)$-component
theta vector, representing all of its components by multiple sums,
and evaluating the signed combinations needed for the remaining
product components. The first two theorems give the even and odd
components of the same holomorphic vector. The third determines the
modular transformations of its completion. The product cancellation
in the first theorem also enters the fourth theorem, which recovers
all components of the four prescribed product vectors.

For the product formulas, define
\begin{align*}
 j(z;q)&:=(z;q)_\infty(q/z;q)_\infty(q;q)_\infty,
 &J_{a,m}&:=j(q^a;q^m),\\
 J_m&:=(q^m;q^m)_\infty,
 &\zeta_m&:=e^{2\pi i/m}.
\end{align*}
Here $m>0$ and $a\in\R$; when $a,m$ are integers, let
$\widetilde J_{a,m}:=j((-q)^a;(-q)^m)$.
The sign function has its usual values $1,0,-1$ on positive,
zero, and negative arguments, respectively.

\begin{definition}
\label{def:fhat}
Let $r\geq2$ be an integer and $\tau\in\Hh$. For
$1\leq\ell\leq2r-1$, define
\begin{equation}
 \widehat f_\ell^{(r)}(\tau)
 :=\sum_{\substack{u\in\Z+(-1)^{\ell+1}\ell/(4r)\\
                   v\in\Z-1/6}}
 \bigl(\sgnstd(u-v)+\sgnstd(u+v)\bigr)
 q^{2ru^2-3v^2/2}e^{-\pi iv}.
 \label{eq:fhat-scalar}
\end{equation}
The associated vector is
$\widehat{\boldsymbol f}^{(r)}
 :=(\widehat f_1^{(r)},\ldots,\widehat f_{2r-1}^{(r)})^T$.
The circumflex does not denote a real-analytic completion.
\end{definition}

\begin{theorem}
\label{thm:nahm-fbar-hat}
Let $r\geq2$ be an integer. The series
$\widehat f_\ell^{(r)}(\tau)$, $1\leq\ell\leq2r-1$, converge
absolutely and locally uniformly on $\Hh$ and are holomorphic.
For $1\leq j\leq r-1$ and
$\mathbf n=(n_1,\ldots,n_r)^T\in\N^r$, define
$N_a=n_{a+1}+\cdots+n_r$ for $1\leq a\leq r-1$. Then
\begin{align}
 \widehat f_{2j}^{(r)}(\tau)
 ={}&2\zeta_{12}^{-5}
 q^{\frac{(2r-j)^2}{2r}-\frac{25}{24}}J_{1,4}
 \notag\\*[-2pt]
 &\quad\times\sum_{\mathbf n\in\N^r}
 \frac{(-1)^{n_r}
 q^{n_1^2+2n_1n_r+2n_1+2\sum_{a=1}^{r-1}N_a^2+n_r^2
       +2\sum_{a=1}^{j-1}N_a+4\sum_{a=j}^{r-1}N_a}}
 {\bigl(\prod_{a=1}^{r-1}(q^2;q^2)_{n_a}\bigr)(q^4;q^4)_{n_r}}
 \notag\\*
 &+\zeta_{12}^{-5}q^{\frac{j^2}{2r}-\frac1{24}}
 \begin{cases}
 0,&3j<r,\\
 J_{2r-2j-1,\,4r-3},&3j=r,\\
 2J_{2r-2j-1,\,4r-3},&3j>r.
 \end{cases}
 \label{eq:even-explicit}
\end{align}
The two signed branches satisfy the product identity
\begin{align}
 &\sum_{\mathbf n\in\N^r}
 \frac{(-1)^{n_r}(1-q^{2n_1+2n_r+1})
 q^{n_1^2+2n_1n_r+2\sum_{a=1}^{r-1}N_a^2+n_r^2
       +2\sum_{a=1}^{j-1}N_a+4\sum_{a=j}^{r-1}N_a-2n_r}}
 {\bigl(\prod_{a=1}^{r-1}(q^2;q^2)_{n_a}\bigr)(q^4;q^4)_{n_r}}
 \notag\\
 &\hspace{18mm}=
 \frac{q^{2j-2r+2}}{J_{1,4}}
 \bigl(J_{2r-2j-1,\,4r-3}-J_{2r-2j-2,\,4r-3}\bigr).
 \label{eq:sub-eta-Bdual}
\end{align}
\end{theorem}

The first identity identifies one signed branch with an even theta
component, including its boundary contribution. Expanding
$1-q^{2n_1+2n_r+1}$ in the second identity expresses the other
branch as $q$ times the first plus the displayed product quotient.
Thus both branches have the same indefinite theta part. Their
difference cancels that part and gives the product evaluation used
for the $B^\vee$-family below. The coefficient at $3j=r$ is half
the coefficient at $3j>r$, because the shifted lattice then meets a
cone boundary and $\sgnstd(0)=0$.

The even indices account for $r-1$ components. The next theorem
supplies the remaining $r$ components, using Euler-type sums with
$2r-2$ summation variables rather than the signed rank-$r$ sums.

\begin{theorem}
\label{thm:euler}
For integers $r\geq2$ and $1\leq j\leq r$, the odd components satisfy
\begin{align}
 \widehat f_{2j-1}^{(r)}(\tau)
 ={}&-2\zeta_{12}^{-5}
 q^{\frac{(4r-2j+1)^2}{8r}-\frac{25}{24}}(q;q)_\infty
 \notag\\[-2pt]
 &\quad\times
 \sum_{N_1\geq\cdots\geq N_{2r-2}\geq0}
 \frac{q^{\sum_{a=1}^{2r-2}N_a^2+
          \sum_{a=1}^{2j-2}N_a+2\sum_{a=2j-1}^{2r-2}N_a}}
 {\prod_{a=1}^{2r-3}(q;q)_{N_a-N_{a+1}}}
 \notag\\
 &-2\zeta_{12}^{-5}q^{\frac{(2j-1)^2}{8r}-\frac1{24}}
 \begin{cases}
 0,&3j\leq r+1,\\
 J_{2r-2j,\,4r-3},&3j\geq r+2.
 \end{cases}
 \label{eq:oddmain}
\end{align}
In particular, the boundary term at $j=r$ vanishes, since
$J_{0,4r-3}=0$.
\end{theorem}

Theorems~\ref{thm:nahm-fbar-hat} and~\ref{thm:euler} therefore give
multiple-sum formulas for all $2r-1$ components of
$\widehat{\boldsymbol f}^{(r)}$. The modular transformations of this
holomorphic vector require a nonholomorphic correction. Replacing
the two sign functions by the error-function kernels of Zwegers
\cite{Zwegers2002} gives the following closed system.

\begin{theorem}
\label{thm:modular}
Let $r\geq2$ be an integer, $\tau:=x+iy\in\Hh$, and
$E(t):=2\int_0^t e^{-\pi w^2}\,dw$. For
$1\leq\ell\leq2r-1$, define
\begin{align}
 H_\ell^{(r)}(\tau):={}&
 \sum_{\substack{u\in\Z+(-1)^{\ell+1}\ell/(4r)\\v\in\Z-1/6}}
 \Biggl[
 E\!\left(\sqrt{\frac{24ry}{4r-3}}(u-v)\right)
 \notag\\[-2pt]
 &\hspace{37mm}+
 E\!\left(\sqrt{\frac{24ry}{4r-3}}(u+v)\right)
 \Biggr]q^{2ru^2-3v^2/2}e^{-\pi iv}.
 \label{eq:H-scalar}
\end{align}
The vector
$\boldsymbol H^{(r)}:=(H_1^{(r)},\ldots,H_{2r-1}^{(r)})^T$
is real analytic and has the weight-one transformation laws
\begin{align}
 H_\ell^{(r)}(\tau+1)
 &=e^{2\pi i(\ell^2/(8r)-1/24)}H_\ell^{(r)}(\tau),
 \label{eq:HT-components}\\
 H_\ell^{(r)}(-1/\tau)
 &=-\frac{i\tau}{\sqrt r}
 \sum_{k=1}^{2r-1}(-1)^{\ell+k}
 \sin\!\left(\frac{\pi\ell k}{2r}\right)H_k^{(r)}(\tau).
 \label{eq:HS-components}
\end{align}
\end{theorem}

The transformation law holds for every $r\geq2$, without a parity
restriction on the rank. The proof uses the decomposition
$\widehat{\boldsymbol f}^{(r)}=\boldsymbol H^{(r)}+
\boldsymbol G^{(r)}$ and expresses each correction term as a finite
sum of products of a unary theta series and a one-dimensional
error-function series. This gives both the
antiholomorphic derivative of the completion and the inhomogeneous
transformation law of the holomorphic vector. The boundary terms
in the first two theorems remain separate from this nonholomorphic
correction.

We now return to the component-representation problem. The product
vectors below are those of Du, Ji, Shen, and Xu
\cite{DuJiShenXu2026}. The signed cancellation in
Theorem~\ref{thm:nahm-fbar-hat}, together with further Bailey-pair
evaluations for the transpose families, supplies the missing
components. The prescribed matrices are specified by a
quadratic-form identity valid at every rank $r\geq2$.

Let $r\geq2$ be an integer and set
\[
 \mathbf d:=(2,\ldots,2,1)^T,\quad
 \mathbf d^\vee:=(1,\ldots,1,2)^T,\quad
 D:=\operatorname{diag}(\mathbf d),\quad
 D^\vee:=\operatorname{diag}(\mathbf d^\vee).
\]
Let $B$ be the unique matrix such that $BD$ is symmetric and,
for every $\mathbf x=(x_1,\ldots,x_r)^T\in\R^r$,
\begin{equation}
 \frac12\mathbf x^TBD\mathbf x
 =\left(x_1+\frac{x_r}{2}\right)^2
  +\frac12\sum_{a=1}^{r-1}
     \left(2\sum_{b=a+1}^{r-1}x_b+x_r\right)^2.
 \label{eq:BD-quadratic}
\end{equation}
Set $A:=B-\frac12E_{rr}$, $A^\vee:=A^T$, and $B^\vee:=B^T$,
where $E_{rr}$ has a single nonzero entry, equal to $1$, in
position $(r,r)$.

Define, for $1\leq j\leq2r-1$,
\begin{align}
 g_j(\tau)
 &:=q^{\frac{(4r-4j+1)^2}{32r-8}}
 \left(q^{\frac{2j-r-1}{2}}
       \frac{J_{8r-4j,\,16r-4}}{J_{1,4}}
       +\frac{\widetilde J_{2j-1,\,4r-1}}{J_{2,4}}\right),
 \label{eq:g-product}\\
 g_j^\vee(2\tau)
 &:=\frac12q^{\frac{2j^2-2j-2r+1}{16r-4}}
 \left(\frac{J_{2r-j,\,4r-1}}{J_{1,4}}
       +(-1)^{j(j-1)/2}
        \frac{\widetilde J_{2r-j,\,4r-1}}{\widetilde J_{1,4}}\right),
 \label{eq:gdual-product}
\end{align}
and, for $1\leq j\leq2r-2$,
\begin{align}
 h_j(\tau)
 &:=q^{\frac{(4j-4r+1)^2}{32r-24}}
 \left(q^{\frac{4j-2r-1}{4}}
       \frac{J_{4(2r-j-1),\,4(4r-3)}}{J_{1,4}}
       +\frac{\widetilde J_{2(2r-j-1),\,4r-3}}{J_{2,4}}\right),
 \label{eq:h-product}\\
 h_j^\vee(2\tau)
 &:=\frac12q^{\frac{j^2-j-r+1}{8r-6}}
 \left(\frac{J_{2r+j-2,\,4r-3}}{J_{1,4}}
       +\frac{\widetilde J_{2r+j-2,\,4r-3}}{\widetilde J_{1,4}}\right).
 \label{eq:hdual-product}
\end{align}

\begin{theorem}
\label{thm:four-families}
For every integer $r\geq2$, the functions
$g_j(\tau)$, $g_j^\vee(2\tau)$, $h_j(\tau)$, and
$h_j^\vee(2\tau)$, in their full index ranges, are finite
$\C$-linear combinations of rank-$r$ generalized Nahm sums with
respective matrix, index vector, and nome
\[
 (A,\mathbf d,q),\qquad
 (A^\vee,\mathbf d^\vee,q^2),\qquad
 (B,\mathbf d,q),\qquad
 (B^\vee,\mathbf d^\vee,q^2).
\]
The sums in the $A$- and $B$-families are unsigned. For
$A^\vee$ and $B^\vee$, sign factors
$\prod_{a=1}^r\varepsilon_a^{n_a}$, with
$\varepsilon_a\in\{1,-1\}$, are allowed.
\end{theorem}

In the transpose families, half-period translates are absorbed into
summand signs and constant phase factors, without changing the matrix,
index vector, or rank. Theorem~\ref{thm:four-families} supplies the
component representations requested in all four parts of
\cite[Problem~1.13]{DuJiShenXu2026} under these conventions. It does
not assert that each dual component is a single untwisted generalized
Nahm sum.
For the $B^\vee$-family, the signed product evaluation used in the
construction is exactly the cancellation identity of the first
theorem.

The proofs use two related applications of Bailey pairs. For the
theta vector, parameter-raising and parameter-lowering
transformations produce Hecke-type double sums; comparing their
summation cones with the shifted lattice gives the exact boundary
terms. A Bailey pair with constant nonzero $\beta$-terms, together
with an iterated transformation relative to $1$, gives the odd
components. Zwegers' transformation formula then reduces the
$S$-transformation to a finite sine transform. For the product
vectors, Euler summation and separation of the parity of the last
variable reduce the multiple sums to bilateral theta series.
Their product evaluations recover the components through explicit
linear combinations, while quadratic congruences control the dual
half-period translates.

Together, the four theorems give a uniform construction in the rank
and the component index. They distinguish the multiple-sum
normalization, the holomorphic cone-boundary correction, and the
nonholomorphic modular completion, and they retain the prescribed
quadratic data in the product-component representations.
Section~\ref{sec:preliminaries} develops the Bailey transformations
and Hecke-type sums. Section~\ref{sec:theta-proofs} proves the
theta-vector results, determines the correction terms, and records
the low-rank specializations. Section~\ref{sec:proof-four-families}
proves the product evaluations and the component-recovery formulas.

\section{Bailey transformations and Hecke-type sums}
\label{sec:preliminaries}

We recall the Bailey transformations and establish the Hecke-type
identities used in the component formulas.

\subsection{Bailey-pair transformations}

Bailey pairs have been used to obtain indefinite quadratic sums by
Lovejoy \cite{Lovejoy2014,Lovejoy2022} and families of mock theta
functions by Lovejoy and Osburn \cite{LovejoyOsburn2013}. We recall
the transformations needed below.

Throughout this subsection, \(0<|q|<1\),
\(a\in\C\setminus\{0\}\), and parameters are chosen so that the
displayed denominators are nonzero. Formulas at removable singularities
are interpreted by continuity. Two sequences
\((\alpha_n(a;q),\beta_n(a;q))_{n\geq0}\)
form a Bailey pair relative to \((a,q)\) if
\begin{equation}
 \beta_n(a;q)=
 \sum_{k=0}^n\frac{\alpha_k(a;q)}
 {(q;q)_{n-k}(a q;q)_{n+k}}.
 \label{eq:bailey-general}
\end{equation}
When the base is understood, we say that the pair is relative to
\(a\).  The inverse relation
\cite[Lemma~3]{Andrews1979} gives \(\alpha_0=\beta_0\) and, for
\(n\geq1\),
\begin{equation}
 \alpha_n(a;q)
 =(1-a q^{2n})
 \sum_{k=0}^n
 \frac{(a q;q)_{n+k-1}(-1)^{n-k}
 q^{\binom{n-k}{2}}}{(q;q)_{n-k}}\,
 \beta_k(a;q).
 \label{eq:bailey-inverse}
\end{equation}
If \(\sum_{n\geq0}|\alpha_n(a;q)|<\infty\), dominated convergence
in \eqref{eq:bailey-general} gives
\begin{equation}
 \lim_{n\to\infty}\beta_n(a;q)
 =\frac1{(q,a q;q)_\infty}
 \sum_{k=0}^\infty\alpha_k(a;q).
 \label{eq:bailey-limit}
\end{equation}

\begin{lemma}[Bailey's lemma]
\label{lem:bailey-lemma}
If \((\alpha_n,\beta_n)\) is a Bailey pair relative to
\((a,q)\), then so is \((\alpha'_n,\beta'_n)\), where
\(\rho_1,\rho_2\in\C\setminus\{0\}\) are auxiliary parameters and
\begin{align*}
 \alpha'_n(a;q)
 &=\frac{(\rho_1,\rho_2;q)_n
 (a q/(\rho_1\rho_2))^n}
 {(a q/\rho_1,a q/\rho_2;q)_n}
 \alpha_n(a;q),\\
 \beta'_n(a;q)
 &=\sum_{k=0}^n
 \frac{(\rho_1,\rho_2;q)_k
 (a q/(\rho_1\rho_2);q)_{n-k}
 (a q/(\rho_1\rho_2))^k}
 {(a q/\rho_1,a q/\rho_2;q)_n(q;q)_{n-k}}
 \beta_k(a;q).
\end{align*}
In particular, letting \(\rho_1,\rho_2\to\infty\) gives
\cite[Eq.~(S1)]{BressoudIsmailStanton2000}
\begin{align}
 \alpha'_n(a;q)
 &=a^nq^{n^2}\alpha_n(a;q),
 &\beta'_n(a;q)
 &=\sum_{k=0}^n\frac{a^kq^{k^2}}{(q;q)_{n-k}}
 \beta_k(a;q).
 \label{eq:bailey-chain}
\end{align}
\end{lemma}

The next transformation is the composition of
\cite[Eqs.~(2.4)--(2.5)]{Lovejoy2004}, with \(b\to\infty\), and
\cite[Eq.~(L1)]{BressoudIsmailStanton2000}. The first changes the
relative parameter from \(a\) to \(aq\), and the second returns
it to \(a\). Its iterates are needed for the odd components, so we
include a direct proof.

\begin{lemma}
\label{lem:bailey-B}
If \((\alpha_n,\beta_n)\) is a Bailey pair relative to
\((a,q)\), then so is \((\alpha'_n,\beta'_n)\), where
\begin{align}
 \alpha'_n(a;q)
 &=a^nq^{n(n-1)}
 \left(\alpha_n(a;q)
 +(1-a q^{2n})\sum_{k=0}^{n-1}
 \alpha_k(a;q)\right),\notag\\*
 \beta'_n(a;q)
 &=\sum_{k=0}^n\frac{a^kq^{k(k-1)}}{(q;q)_{n-k}}
 \beta_k(a;q).
 \label{eq:bailey-B-general}
\end{align}
At \(n=0\), the inner sum is empty, so
\(\alpha'_0=\alpha_0\).
\end{lemma}

\begin{proof}
It suffices to establish the coefficient identity
\begin{align}
 &\sum_{k=h}^n
 \frac{a^kq^{k(k-1)}}
 {(q;q)_{n-k}(q;q)_{k-h}(a q;q)_{h+k}}
 \notag\\
 &\quad=
 \frac{a^hq^{h(h-1)}}
 {(q;q)_{n-h}(a q;q)_{n+h}}
 +\sum_{m=h+1}^n
 \frac{a^mq^{m(m-1)}(1-a q^{2m})}
 {(q;q)_{n-m}(a q;q)_{n+m}},
 \label{eq:bailey-B-coefficient}
\end{align}
valid for \(0\leq h\leq n\).  Denote its left-hand side by
\(\mathcal L_{n,h}\).  Replacing \(k\) by \(h+k\) and
applying the terminating transformation in
\cite[p.~19, Exercise~1.18(i)]{GasperRahman2004} gives
\begin{align*}
 \mathcal L_{n,h}
 &=\frac{a^hq^{h(h-1)}}
 {(q;q)_{n-h}(a q;q)_{2h}}
 \sum_{k=0}^{n-h}
 \frac{(q^{h-n};q)_k(-1)^kq^{\binom k2}
 a^kq^{(n+h)k}}
 {(a q^{2h+1};q)_k(q;q)_k}\\
 &=\frac{a^hq^{h(h-1)}}
 {(q;q)_{n-h}(a q;q)_{2h}}
 \frac{(a q^{n+h};q)_\infty}
 {(a q^{2h+1};q)_\infty}
 \sum_{k=0}^{1}
 \frac{(q^{-1};q)_k(-1)^kq^{\binom k2}
 a^kq^{(2h+1)k}}
 {(a q^{n+h};q)_k(q;q)_k}\\
 &=\frac{a^hq^{h(h-1)}
 (1-a q^{n+h}+a q^{2h})}
 {(q;q)_{n-h}(a q;q)_{n+h}}.
\end{align*}
The intermediate expression is interpreted by continuity when a
removable singularity occurs. To compare it with the right-hand side
of \eqref{eq:bailey-B-coefficient}, set
\[
 L_m:=\frac{a^mq^{m(m-1)}}
 {(q;q)_{n-m}(a q;q)_{n+m}},
 \qquad
 V_m:=a q^{2m}(1-q^{n-m})L_m
 \quad(h\leq m\leq n).
\]
For \(h+1\leq m\leq n\), direct cancellation gives
\(V_{m-1}=(1-a q^{n+m})L_m\), and hence
\begin{equation*}
 V_{m-1}-V_m=(1-a q^{2m})L_m.
\end{equation*}
Since \(V_n=0\), summation yields
\[
 L_h+\sum_{m=h+1}^n(1-a q^{2m})L_m
 =L_h+V_h
 =(1-a q^{n+h}+a q^{2h})L_h,
\]
which proves \eqref{eq:bailey-B-coefficient}.  Substituting
\eqref{eq:bailey-general} into the formula for \(\beta'_n\),
interchanging the finite sums, and applying
\eqref{eq:bailey-B-coefficient} now gives
\begin{align*}
 \beta'_n
 &=\sum_{h=0}^n\alpha_h
 \sum_{k=h}^n
 \frac{a^kq^{k(k-1)}}
 {(q;q)_{n-k}(q;q)_{k-h}(a q;q)_{k+h}}\\
 &=\sum_{m=0}^n
 \frac{a^mq^{m(m-1)}}
 {(q;q)_{n-m}(a q;q)_{n+m}}
 \left(\alpha_m+(1-a q^{2m})
 \sum_{h=0}^{m-1}\alpha_h\right)\\
 &=\sum_{m=0}^n\frac{\alpha'_m}
 {(q;q)_{n-m}(a q;q)_{n+m}}.
\end{align*}
This is the Bailey-pair relation for \((\alpha'_n,\beta'_n)\).
\end{proof}

We also require transformations that lower and raise the relative
parameter from \(a\) to \(a/q\) and \(aq\), respectively.

\begin{lemma}[{\cite[Lemma~3.1]{Lovejoy2022}}]
\label{lem:bailey-lower}
If \((\alpha_n,\beta_n)\) is a Bailey pair relative to
\((a,q)\), then the following sequences form a Bailey
pair relative to \((a/q,q)\):
\begin{align}
 \alpha'_n(a/q;q)
 &=(1-a)
 \left(\frac{q^n\alpha_n(a;q)}{1-a q^{2n}}
 -\frac{q^{n-1}\alpha_{n-1}(a;q)}
 {1-a q^{2n-2}}\right),\notag\\
 \beta'_n(a/q;q)&=q^n\beta_n(a;q).
 \label{eq:bailey-lower}
\end{align}
The second term in the formula for \(\alpha'_n\) is omitted when
\(n=0\).
\end{lemma}

\begin{lemma}[{\cite[Eqs.~(2.4)--(2.5)]{Lovejoy2004}}]
\label{lem:bailey-raise}
If \((\alpha_n,\beta_n)\) is a Bailey pair relative to
\((a,q)\), then the following sequences form a Bailey
pair relative to \((a q,q)\):
\begin{align*}
 \alpha'_n(a q;q)
 &=\frac{(1-a q^{2n+1})(a q/b;q)_n
 (-b)^nq^{\binom n2}}
 {(1-a q)(bq;q)_n}
 \sum_{k=0}^n\frac{(b;q)_k(-b)^{-k}q^{-\binom k2}}
 {(a q/b;q)_k}\,\alpha_k(a;q),\\
 \beta'_n(a q;q)
 &=\frac{(b;q)_n}{(bq;q)_n}\,\beta_n(a;q).
\end{align*}
Here \(b\in\C\setminus\{0\}\) is an auxiliary parameter.
The specialization \(b\to0\) gives
\begin{align}
 \alpha'_n(a q;q)
 &=\frac{1-a q^{2n+1}}{1-a q}
 a^nq^{n^2}
 \sum_{k=0}^na^{-k}q^{-k^2}\alpha_k(a;q),
 \notag\\
 \beta'_n(a q;q)&=\beta_n(a;q),
 \label{eq:bailey-raise}
\end{align}
whereas the limit \(b\to\infty\) gives
\begin{align}
 \alpha'_n(a q;q)
 &=\frac{1-a q^{2n+1}}{1-a q}\,q^{-n}
 \sum_{k=0}^n\alpha_k(a;q),
 &\beta'_n(a q;q)&=q^{-n}\beta_n(a;q).
 \label{eq:bailey-raise-infty}
\end{align}
\end{lemma}

The proofs use the following pairs from Groups C and G of Slater's
list \cite[p.~469]{Slater1951} and the indicated variants.
The symbols \(\alpha_n(a;q)\) and \(\beta_n(a;q)\) are local to each
pair and will be reused in subsequent constructions.
For the starred pairs and the pairs
\(\mathrm{G2}\) and \(\mathrm{G5}\), we use the normalizations in
\cite[Subsections~2.2--2.3]{DuJiShenXu2026}.

In all these pairs, \(\alpha_0(a;q)=1\).  The formulas for
\(\alpha_{2n}\) in \(\mathrm{C1}\) and \(\mathrm{C5}\), and for
\(\alpha_n\) in \(\mathrm{G1}\), apply when \(n\geq1\); all other
formulas apply when \(n\geq0\).
\begin{align*}
 &\alpha_{2n}(1;q)=(-1)^nq^{3n^2}(q^n+q^{-n}),\quad
  \alpha_{2n+1}(1;q)=0,
  \\*
 &\beta_n(1;q)=\frac1{(q;q)_n(q;q^2)_n};\tag{C1}\label{eq:pair-C1}\\[4pt]
 &\alpha_{2n}(q;q)=(-1)^nq^{3n^2+3n},\quad
  \alpha_{2n+1}(q;q)=(-1)^{n+1}q^{3n^2+3n},
\\*
 &\beta_n(q;q)=\frac{q^n}{(q;q)_n(q^3;q^2)_n};  \tag{C4}\label{eq:pair-C4}\\[4pt]
 &\alpha_{2n}(1;q)=(-1)^nq^{n^2}(q^n+q^{-n}),\quad
  \alpha_{2n+1}(1;q)=0,
  \\*
 &\beta_n(1;q)=\frac{q^{\binom n2}}{(q;q)_n(q;q^2)_n};\tag{C5}\label{eq:pair-C5}\\[4pt]
 &\alpha_n(1;q)=(-1)^nq^{\frac12n^2+\frac12\binom n2}(1+q^{n/2}),\quad
  \beta_n(1;q)=\frac1{(-q^{1/2};q)_n(q^2;q^2)_n};
  \tag{G1}\label{eq:pair-G1}\\[4pt]
 &\alpha_n(q;q)=(-1)^nq^{\frac32\binom{n+1}2}
  \frac{q^{-n}-q^{n+1}}{1-q},\quad
  \beta_n(q;q)=\frac1{(-q^{1/2};q)_n(q^2;q^2)_n};
  \tag{G1*}\label{eq:pair-G1-star}\\[4pt]
 &\alpha_n(q;q)=(-1)^nq^{\frac32\binom{n+1}2}
  \frac{q^{-n/2}-q^{(n+1)/2}}{1-q^{1/2}},\quad
  \beta_n(q;q)=\frac1{(-q^{3/2};q)_n(q^2;q^2)_n};
  \tag{G2}\label{eq:pair-G2}\\[4pt]
 &\alpha_n(q;q)=(-1)^nq^{n^2/4-3n/4}\frac{1-q^{2n+1}}{1-q},\quad
  \beta_n(q;q)=\frac{(-1)^nq^{n^2/2-n}}
  {(-q^{1/2};q)_n(q^2;q^2)_n};
  \tag{G4**}\label{eq:pair-G4-starstar}\\[4pt]
 &\alpha_n(q;q)=(-1)^nq^{\frac12\binom n2}
  \frac{1-q^{n+1/2}}{1-q^{1/2}},\quad
  \beta_n(q;q)=\frac{(-1)^nq^{n^2/2}}
  {(-q^{3/2};q)_n(q^2;q^2)_n}.
  \tag{G5}\label{eq:pair-G5}
\end{align*}

\subsection{The odd Euler sums}

For $r\geq2$ and $1\leq j\leq r$, abbreviate the multiple sum
in Theorem~\ref{thm:euler} by
\begin{equation}
 \Eodd_{r,j}(q):=
 \sum_{N_1\geq\cdots\geq N_{2r-2}\geq0}
 \frac{q^{\sum_{a=1}^{2r-2}N_a^2+
           \sum_{a=1}^{2j-2}N_a+2\sum_{a=2j-1}^{2r-2}N_a}}
 {\prod_{a=1}^{2r-3}(q;q)_{N_a-N_{a+1}}}.
 \label{eq:Eodd}
\end{equation}

For \(n\in\Z_{\geq0}\), define the finite theta sum
\[
 T_n(q):=\sum_{t=-n}^{n-1}(-1)^tq^{-(3t^2+t)/2}.
\]
In particular, \(T_0(q)=0\). The seed pair below will be used to
represent \(\Eodd_{r,j}\). It is related to the pair with constant
\(\beta\)-sequence in \cite[Lemma~10]{Lovejoy2002}: invert
\eqref{eq:bailey-raise} with \(a=1\), and subtract the resulting
pair relative to \(1\) from the unit pair, whose \(\beta\)-sequence
is \(\delta_{n,0}\). Here \(\delta\) is the Kronecker delta.
The following direct proof does not require this derivation.

\begin{lemma}
\label{lem:odd-seed}
The following sequences form a Bailey pair relative to \((1,q)\):
\begin{align}
 \alpha_0(1;q)&=\beta_0(1;q)=0,
 &\beta_n(1;q)&=-1\quad(n\geq1),
 \label{eq:seedbeta}\\
 \alpha_n(1;q)&=q^{2n^2-n}(1-q^{2n})T_n(q)
 &&(n\geq1).
 \label{eq:seedalpha}
\end{align}
\end{lemma}

\begin{proof}
Substituting the stated \(\beta\)-sequence into \eqref{eq:bailey-inverse}
with \(a=1\) reduces the verification of the Bailey-pair relation to
the finite pentagonal identity
\begin{equation}
 -\sum_{k=0}^{n-1}(-1)^kq^{\binom{k}{2}}
 \frac{(q;q)_{2n-k-1}}{(q;q)_k}
 =q^{2n^2-n}T_n(q).
 \label{eq:finitepentagonal}
\end{equation}
We prove this identity by showing that its two sides satisfy the
same first-order recurrence. Define
\begin{align*}
 U_{n,k}&:=(-1)^kq^{\binom{k}{2}}
 \frac{(q;q)_{2n-k-1}}{(q;q)_k},
 &F_n&:=-\sum_{k=0}^{n-1}U_{n,k},\\
 C_{n,k}&:=
 \frac{(q^{2n}-q^k)(q^k-1)
 (q^{2n+k+1}+q^{2n+1}-q^k)}{q^{2k}},
 &G_n&:=q^{2n^2-n}T_n(q).
\end{align*}
For \(n\geq1\) and \(0\leq k<n\), direct simplification gives
\begin{equation}
 U_{n+1,k}-q^{4n+1}U_{n,k}
 =C_{n,k+1}U_{n,k+1}-C_{n,k}U_{n,k}.
 \label{eq:finitecertificate}
\end{equation}
Summing \eqref{eq:finitecertificate} over \(0\leq k<n\), using
\(C_{n,0}=0\), and retaining the additional endpoint
\(U_{n+1,n}\) gives
\begin{equation*}
 F_{n+1}=q^{4n+1}F_n
 +(-1)^{n+1}q^{\binom{n+1}{2}}(1-q^{2n+1}),
 \qquad F_1=q-1.
\end{equation*}
Separating the two new endpoint terms in \(T_{n+1}\) shows that
\(G_n\) satisfies the same recurrence and initial value.  Thus
\(F_n=G_n\), which proves \eqref{eq:finitepentagonal} and hence the
Bailey-pair relation for \eqref{eq:seedbeta}--\eqref{eq:seedalpha}
relative to \(1\).
\end{proof}

The following identity relates the finite theta sums \(T_n(q)\)
to \(\Eodd_{r,j}(q)\). For \(1\leq j\leq r\), let
\[
 S_{r,j}(q):=\sum_{n\geq1}
 \left(q^{2rn^2-(2j-1)n}-q^{2rn^2+(2j-1)n}\right)T_n(q).
\]

\begin{lemma}
\label{lem:oddBailey}
For \(r\geq2\) and \(1\leq j\leq r\),
\begin{equation}
 S_{r,j}(q)=-q^{2(r-j)}(q;q)_\infty\Eodd_{r,j}(q).
 \label{eq:SoddE}
\end{equation}
\end{lemma}

\begin{proof}
Specializing Lemmas~\ref{lem:bailey-lemma} and~\ref{lem:bailey-B}
to \(a=1\) gives the two transformations
\begin{align}
 \mathsf S_1:\quad
 \alpha'_n&=q^{n^2}\alpha_n,
 &\beta'_n&=\sum_{k=0}^n
 \frac{q^{k^2}}{(q;q)_{n-k}}\beta_k,
 \notag\\
 \mathsf B:\quad
 \alpha'_n&=q^{n(n-1)}
 \left(\alpha_n+(1-q^{2n})\sum_{k=0}^{n-1}\alpha_k\right),
 &\beta'_n&=\sum_{k=0}^n
 \frac{q^{k(k-1)}}{(q;q)_{n-k}}\beta_k
 \label{eq:Bbailey}
\end{align}
which preserve Bailey pairs relative to \(1\). For \(t\in\N\),
let \(\mathsf B^t\) denote the
\(t\)-fold iterate of the transformation of the \(\alpha\)-sequence in
\eqref{eq:Bbailey}, with \(\mathsf B^0\) the identity.  For a rapidly
decreasing sequence \(\alpha=(\alpha_n)_{n\geq0}\), the sum of its
\(t\)-th iterate is
\begin{equation}
 \sum_{n\geq0}(\mathsf B^t\alpha)_n
 =\sum_{n\geq0}q^{tn^2}
 \left(\sum_{h=0}^tq^{(2h-t)n}\right)\alpha_n.
 \label{eq:Bfunctional}
\end{equation}
Here rapid decrease means
\(\lvert\alpha_n\rvert\leq C_0\lvert q\rvert^{\gamma_0n^2}\) for
constants \(C_0,\gamma_0>0\).  The initial sequence
\eqref{eq:seedalpha} satisfies this condition, since its smallest
exponent is \((n^2-n)/2\).  The condition is also preserved by every
finite sequence of the two Bailey transformations used here.  Thus all
rearrangements in \eqref{eq:Bfunctional} are absolutely convergent.
To prove \eqref{eq:Bfunctional}, set
\(w_t(m):=q^{tm^2}\sum\limits_{h=0}^tq^{(2h-t)m}\).
The identity
\begin{equation*}
 q^{n(n-1)}(1-q^{2n})w_t(n)
 =q^{(t+1)n(n-1)}-q^{(t+1)n(n+1)}
\end{equation*}
shows that the tail sum over \(n>m\) telescopes to
\(q^{(t+1)m(m+1)}\).  Adding the \(n=m\) term gives
\(w_{t+1}(m)\), proving \eqref{eq:Bfunctional} by induction.

Starting with \eqref{eq:seedbeta}--\eqref{eq:seedalpha}, apply
\(\mathsf S_1\) \(2r-2j\) times, followed by
\(\mathsf B\) \(2j-2\) times.  With \(t:=2j-2\),
equation \eqref{eq:Bfunctional} gives the sum of the final \(\alpha\)-sequence as
\begin{equation}
 \sum_{n\geq1}q^{2rn^2}T_n(q)(1-q^{2n})
 \sum_{h=0}^{t}q^{(2h-t-1)n}=S_{r,j}(q),
 \label{eq:finalalphasum}
\end{equation}
because multiplication by \(1-q^{2n}\) telescopes the finite
\(h\)-sum to its two endpoint terms.  Denote the
resulting pair by \((\alpha_n^{(2r)},\beta_n^{(2r)})\).
Its \(\beta\)-sequence is
\begin{equation}
 \beta_n^{(2r)}
 =-\sum_{n\geq N_1\geq\cdots\geq N_{2r-2}\geq1}
 \frac{q^{\sum_{a=1}^{2j-2}(N_a^2-N_a)+\sum_{a=2j-1}^{2r-2}N_a^2}}
 {(q;q)_{n-N_1}
 (q;q)_{N_1-N_2}(q;q)_{N_2-N_3}\cdots
 (q;q)_{N_{2r-3}-N_{2r-2}}}.
 \label{eq:oddfinalbeta}
\end{equation}
The Gaussian bound on the final \(\alpha\)-sequence justifies
\eqref{eq:bailey-limit}. In \eqref{eq:oddfinalbeta}, the positive
quadratic part and the uniform bounds on the reciprocal Pochhammer
factors justify passage to the limit under the nested sum by dominated
convergence. Hence \eqref{eq:finalalphasum} and
\eqref{eq:oddfinalbeta} give
\begin{align*}
 S_{r,j}(q)
 &=(q;q)_\infty^2\lim_{n\to\infty}\beta_n^{(2r)}
 \notag\\
 &=-(q;q)_\infty
 \sum_{N_1\geq\cdots\geq N_{2r-2}\geq1}
 \frac{q^{\sum_{a=1}^{2j-2}(N_a^2-N_a)+\sum_{a=2j-1}^{2r-2}N_a^2}}
 {(q;q)_{N_1-N_2}(q;q)_{N_2-N_3}\cdots
  (q;q)_{N_{2r-3}-N_{2r-2}}}.
\end{align*}
Setting \(N_a=\widetilde N_a+1\) changes the lower bounds to zero
and contributes the constant \(2r-2j\) to the exponent. The
remaining quadratic and linear terms are exactly those of
\(\Eodd_{r,j}(q)\) in \eqref{eq:Eodd}, proving \eqref{eq:SoddE}.

\end{proof}

\subsection{The signed Nahm sums}

The symmetric product for the transpose family has the
quadratic-form identity
\begin{equation}
 \frac12\mathbf n^TB^\vee D^\vee\mathbf n
 =\frac12n_1^2+n_1n_r+
  \sum_{a=1}^{r-1}(n_{a+1}+\cdots+n_r)^2+\frac12n_r^2.
 \label{eq:BD-dual}
\end{equation}
It follows directly from the quadratic-form identity
\eqref{eq:BD-quadratic}, since
$D^\vee=2D^{-1}$ and
$B^\vee D^\vee=2D^{-1}(BD)D^{-1}$.
For $1\leq j\leq r-1$, define the integer vectors
\begin{align}\label{b-+}
 \mathbf b^+_{r,j}
  &=(1,1,2,\ldots,j-1,j+1,j+3,\ldots,2r-j-1)^T,\\
 \mathbf b^-_{r,j}\label{b--}
  &=(0,1,2,\ldots,j-1,j+1,j+3,\ldots,
       2r-j-3,2r-j-2)^T
\end{align}
where each progression is empty if its initial term exceeds its
terminal term. These formulas include $r=2$ without an endpoint convention.
The two signed branches are defined by
\begin{equation}
 \mathcal N^\pm_{r,j}(q):=
 \sum_{\mathbf n=(n_1,\ldots,n_r)^T\in\N^r}
 \frac{(-1)^{n_r}
 q^{\frac12\mathbf n^TB^\vee D^\vee\mathbf n+
       \mathbf n^T\mathbf b^\pm_{r,j}}}
 {(q;q)_{n_1}\cdots(q;q)_{n_{r-1}}(q^2;q^2)_{n_r}}.
 \label{eq:Npm}
\end{equation}
The superscripts \(+\) and \(-\) refer only to the linear terms
\(\mathbf b^+_{r,j}\) and \(\mathbf b^-_{r,j}\); both branches have
the same sign factor \((-1)^{n_r}\). For \(1\leq j\leq r-1\), define
\begingroup\small
\begin{equation}
\begin{aligned}
 \Eeven_{r,j}(q)
 &:={}
 \sum_{N_1\geq\cdots\geq N_{r-1}\geq0}
 \frac{(-1)^{N_{r-1}}q^{\sum\limits_{a=1}^{r-1}N_a^2+\sum\limits_{a=1}^{j-1}N_a+2\sum\limits_{a=j}^{r-1}N_a+\frac12N_{r-1}^2}}
 {\displaystyle\prod_{a=1}^{r-2}(q;q)_{N_a-N_{a+1}}
  (-q^{3/2};q)_{N_{r-1}}(q^2;q^2)_{N_{r-1}}}.
\end{aligned}
 \label{eq:Eeven}
\end{equation}
\endgroup

The following lemma reduces the plus branch by Euler summation
and identifies the resulting nested sum with a Hecke-type double sum.

\begin{lemma}
\label{lem:evenBailey}
For \(r\geq2\) and \(1\leq j\leq r-1\),
\begin{equation*}
 \Eeven_{r,j}(q)
 =\frac{\mathcal N^+_{r,j}(q)}{(-q^{3/2};q)_\infty}
 =\frac{1+q^{1/2}}{(q;q)_\infty}\,\overline f_{r,j}(q),
 \qquad
 \mathcal N^+_{r,j}(q)
 =\frac{\overline f_{r,j}(q)}{J_{1/2,2}},
\end{equation*}
where the common Hecke-type double sum is
\begin{align}
 \overline f_{r,j}(q):={}&
 \sum_{n=0}^\infty q^{rn^2+(2r-j)n}
 \sum_{t=-n-1}^{n}(-1)^tq^{-(3t^2+5t)/4}
 \notag\\*
 &-\sum_{n=0}^\infty q^{rn^2+jn+j-r}
 \sum_{t=-n}^{n-1}(-1)^tq^{-(3t^2+5t)/4}
 \qquad(1\leq j\leq r-1).
 \label{eq:fbar}
\end{align}
\end{lemma}
\begin{proof}
For \(r\geq2\) and \(1\leq j\leq r-1\), the tail sums are
\[
 N_a:=n_{a+1}+\cdots+n_r\qquad(1\leq a\leq r-1).
\]
Then \(N_1\geq\cdots\geq N_{r-1}\geq0\), with
\(n_{a+1}=N_a-N_{a+1}\) for \(1\leq a\leq r-2\) and
\(n_r=N_{r-1}\).  The entries of \(B^\vee D^\vee\) and
\(\mathbf b^+_{r,j}\) give
\begin{align*}
 \frac12\mathbf n^TB^\vee D^\vee\mathbf n
 +\mathbf n^T\mathbf b^+_{r,j}
 ={}&\frac12n_1^2+n_1N_{r-1}+n_1
 +\sum_{a=1}^{r-1}N_a^2+\frac12N_{r-1}^2+\sum_{a=1}^{j-1}N_a+2\sum_{a=j}^{r-1}N_a.
\end{align*}
Euler's identity
\[
 \sum_{n\geq0}\frac{q^{\binom n2}z^n}{(q;q)_n}
 =(-z;q)_\infty
\]
evaluates the sum over \(n_1\) as
\((-q^{N_{r-1}+3/2};q)_\infty\).  Since
\[
 (-q^{N_{r-1}+3/2};q)_\infty
 =\frac{(-q^{3/2};q)_\infty}{(-q^{3/2};q)_{N_{r-1}}},
\]
we obtain
\begin{equation}
 \mathcal N^+_{r,j}(q)=(-q^{3/2};q)_\infty\Eeven_{r,j}(q).
 \label{eq:even-euler-reduction}
\end{equation}

To evaluate \(\Eeven_{r,j}\), apply \eqref{eq:bailey-raise} with
\(a=q\) to the pair \eqref{eq:pair-G5}.  This gives the
following pair relative to \((q^2,q)\):
\begin{align}
 \alpha_n^{(0)}(q^2;q)
 &=\frac{(1-q^{2n+2})q^{n(n+1)}}
 {(1-q^2)(1-q^{1/2})}
 \sum_{t=-n-1}^{n}(-1)^tq^{-(3t^2+5t)/4},\notag\\
 \beta_n^{(0)}(q^2;q)
 &=\frac{(-1)^nq^{n^2/2}}
 {(-q^{3/2};q)_n(q^2;q^2)_n}.
 \label{eq:even-raised-pair}
\end{align}
In particular, \(\alpha_0^{(0)}=\beta_0^{(0)}=1\).
Since \(r-j\geq1\), applying
\eqref{eq:bailey-chain}, with relative parameter \(q^2\),
\(r-j-1\) times yields
\begin{equation}
 \alpha_n^{(r-j-1)}(q^2;q)
 =\frac{(1-q^{2n+2})q^{(r-j)n^2+(2r-2j-1)n}}
 {(1-q^2)(1-q^{1/2})}
 \sum_{t=-n-1}^{n}(-1)^tq^{-(3t^2+5t)/4}.
 \label{eq:even-chained-pair}
\end{equation}
Its \(\beta\)-sequence is
\begin{equation*}
 \beta_n^{(r-j-1)}(q^2;q)
 =\sum_{n=N_j\geq N_{j+1}\geq\cdots\geq N_{r-1}\geq0}
 \frac{(-1)^{N_{r-1}}
 q^{\sum_{a=j+1}^{r-1}(N_a^2+2N_a)+N_{r-1}^2/2}}
 {\displaystyle\prod_{a=j}^{r-2}(q;q)_{N_a-N_{a+1}}
 (-q^{3/2};q)_{N_{r-1}}(q^2;q^2)_{N_{r-1}}}.
\end{equation*}
The index \(N_j=n\) is fixed. If \(j=r-1\), no chain step has
been applied, so there are no free summation indices and the formula
reduces to \(\beta_n^{(0)}(q^2;q)\).

Applying \eqref{eq:bailey-lower} with \(a=q^2\) now gives a pair
relative to \((q,q)\), with
\begin{align}
 \alpha_n^{(r-j)}(q;q)
 =\frac1{1-q^{1/2}}\Bigg(&q^{(r-j)n^2+(2r-2j)n}
 \sum_{t=-n-1}^{n}(-1)^tq^{-(3t^2+5t)/4}\notag\\
 &-q^{(r-j)(n-1)^2+(2r-2j)(n-1)}
 \sum_{t=-n}^{n-1}(-1)^tq^{-(3t^2+5t)/4}\Bigg),
 \label{eq:even-lowered-pair}
\end{align}
with \(\beta_n^{(r-j)}(q;q)=q^n\beta_n^{(r-j-1)}(q^2;q)\).
The second finite sum is empty at \(n=0\).
Finally, applying \eqref{eq:bailey-chain}, now with relative
parameter \(q\), \(j\) times gives
\begin{align}
 \alpha_n^{(r)}(q;q)
 =\frac1{1-q^{1/2}}\Bigg(&q^{rn^2+(2r-j)n}
 \sum_{t=-n-1}^{n}(-1)^tq^{-(3t^2+5t)/4}\notag\\
 &-q^{rn^2+jn+j-r}
 \sum_{t=-n}^{n-1}(-1)^tq^{-(3t^2+5t)/4}\Bigg),
 \label{eq:even-final-alpha}
\end{align}
and
\begin{equation*}
 \beta_n^{(r)}(q;q)
 =\sum_{n\geq N_1\geq\cdots\geq N_{r-1}\geq0}
 \frac{(-1)^{N_{r-1}}
 q^{\sum_{a=1}^{r-1}N_a^2+\sum_{a=1}^{j-1}N_a
 +2\sum_{a=j}^{r-1}N_a+N_{r-1}^2/2}}
 {\displaystyle(q;q)_{n-N_1}
 \prod_{a=1}^{r-2}(q;q)_{N_a-N_{a+1}}
 (-q^{3/2};q)_{N_{r-1}}(q^2;q^2)_{N_{r-1}}}.
\end{equation*}
The superscripts label the successive stages of this construction.

The sum of the final \(\alpha\)-sequence is
\(\overline f_{r,j}(q)/(1-q^{1/2})\), by \eqref{eq:fbar}.
For \(r\geq2\), the quadratic exponents give absolute convergence
of this sum and of the limiting nested \(\beta\)-sum. Applying
\eqref{eq:bailey-limit} with relative parameter \(q\) therefore gives
\begin{equation}
 \frac{\Eeven_{r,j}(q)}{(q;q)_\infty}
 =\frac{\overline f_{r,j}(q)}
 {(1-q^{1/2})(q;q)_\infty(q^2;q)_\infty},
 \label{eq:even-final-limit}
\end{equation}
or equivalently,
\[
 \Eeven_{r,j}(q)
 =\frac{1+q^{1/2}}{(q;q)_\infty}\,\overline f_{r,j}(q).
\]
Combining this with \eqref{eq:even-euler-reduction} and
\[
 \frac{(q;q)_\infty}
 {(1+q^{1/2})(-q^{3/2};q)_\infty}=J_{1/2,2}
\]
proves the lemma.
\end{proof}

\section{The theta vector and its completion}
\label{sec:theta-proofs}

We now compare the Hecke-type sums with the shifted lattice,
prove the first two theorems, and then determine the modular
transformations and the nonholomorphic correction. For these
arguments it is convenient to use the invariant lattice notation.

The quadratic and bilinear forms are, in the notation of
\cite[Section~2.2]{Zwegers2002},
\begin{equation*}
 A_r:=\begin{pmatrix}4r&0\\0&-3\end{pmatrix},\qquad
 Q_r(\nu):=\frac12\nu^TA_r\nu,\qquad
 B_r(\mu,\nu):=\mu^TA_r\nu,
\end{equation*}
where \(\mu,\nu\in\R^2\) are column vectors. With
\begin{equation*}
 \gamma_r:=\gcd(3,4r)=\gcd(3,r),
\end{equation*}
the primitive negative vectors and the elliptic characteristic are
\begin{equation*}
 c_+:=\frac1{\gamma_r}\binom{3}{4r},\qquad
 c_-:=\frac1{\gamma_r}\binom{-3}{4r},\qquad
 b:=\binom0{1/6}.
\end{equation*}
For \(1\leq\ell\leq2r-1\), let
\begin{equation}
 \ve_\ell:=(-1)^{\ell+1},\qquad
 s_\ell:=\ve_\ell\ell,
 \qquad
 a_\ell:=\binom{s_\ell/(4r)}{-1/6}.
 \label{eq:asl}
\end{equation}
Thus \((s_1,s_2,s_3,\ldots)=(1,-2,3,-4,\ldots)\).

In this notation, the series in Definition~\ref{def:fhat} has the form
\begin{equation}
 \widehat f_\ell^{(r)}(\tau)
 =\sum_{\nu\in a_\ell+\Z^2}
 \bigl\{\sgnstd(B_r(c_+,\nu))-\sgnstd(B_r(c_-,\nu))\bigr\}
 e^{2\pi iQ_r(\nu)\tau+2\pi iB_r(\nu,b)}.
 \label{eq:fdef}
\end{equation}
The sign kernel vanishes unless \(\lvert\nu_2\rvert\leq
\lvert\nu_1\rvert\) with $\nu=\binom{\nu_1}{\nu_2}$. On this support,
\[
 Q_r(\nu)=2r\nu_1^2-\frac32\nu_2^2
 \geq\frac{4r-3}{4}(\nu_1^2+\nu_2^2).
\]
The series therefore converges absolutely and locally uniformly on
\(\Hh\), so \(\widehat f_\ell^{(r)}\) is holomorphic.
The equality with the scalar expression in Definition~\ref{def:fhat}
follows from
$B_r(c_+,\nu)=12r(\nu_1-\nu_2)/\gamma_r$,
$B_r(c_-,\nu)=-12r(\nu_1+\nu_2)/\gamma_r$, and
$2\pi iB_r(\nu,b)=-\pi i\nu_2$.

\subsection{The even components}
\label{subsec:even-components}

For $1\leq j\leq r-1$, use the abbreviations
\begin{align*}
 \kappa^{\mathrm e}_{r,j}
 &=\frac{(2r-j)^2}{2r}-\frac{25}{24},&
 \lambda^{\mathrm e}_{r,j}
 &=\frac{j^2}{2r}-\frac1{24},
\end{align*}
and denote the even boundary term by
\begin{equation}
 \Theta^{\mathrm e}_{r,j}(\tau)=
 \begin{cases}
 0,&3j<r,\\
 \zeta_{12}^{-5}q^{\lambda^{\mathrm e}_{r,j}}
 J_{2r-2j-1,4r-3},&3j=r,\\
 2\zeta_{12}^{-5}q^{\lambda^{\mathrm e}_{r,j}}
 J_{2r-2j-1,4r-3},&3j>r.
 \end{cases}
 \label{eq:eventheta}
\end{equation}

\begin{proposition}
\label{prop:even-theta}
For $r\geq2$ and $1\leq j\leq r-1$,
\begin{align}
 \widehat f_{2j}^{(r)}(\tau)
 &=2\zeta_{12}^{-5}q^{\kappa^{\mathrm e}_{r,j}}
   J_{1,4}\mathcal N^+_{r,j}(q^2)
   +\Theta^{\mathrm e}_{r,j}(\tau)\notag\\
 &=2\zeta_{12}^{-5}q^{\kappa^{\mathrm e}_{r,j}}
   \frac{(q^2;q^2)_\infty}{1+q}\Eeven_{r,j}(q^2)
   +\Theta^{\mathrm e}_{r,j}(\tau).
 \label{eq:evenmain}
\end{align}
\end{proposition}

\begin{proof}
For \(1\leq j\leq r-1\), a representative of the characteristic
in \eqref{eq:fdef} is
\begin{equation*}
 a_{2j}\equiv
 \binom{(2r-j)/(2r)}{5/6}\pmod{\Z^2}
\end{equation*}
and every point of the shifted lattice is obtained by adding
\((X,Y)^T\), with \(X,Y\in\Z\), to this representative.
Substitution into \eqref{eq:fdef} gives
\begin{align}
 \widehat f_{2j}^{(r)}
 ={}&\zeta_{12}^{-5}q^{\kappa^{\mathrm e}_{r,j}}
 \sum_{X,Y\in\Z}
 \left\{
 \sgnstd\!\left(X-Y+\frac16-\frac{j}{2r}\right)
 \right.\notag\\[-2pt]
 &\left.\hspace{31mm}+
 \sgnstd\!\left(X+Y+1+\frac56-\frac{j}{2r}\right)
 \right\}
 (-1)^Yq^{2rX^2+(4r-2j)X-(3Y^2+5Y)/2}.
 \label{eq:evenraw}
\end{align}
The substitutions \((X,Y)=(n,t)\) in the first sum of
\eqref{eq:fbar} and \((X,Y)=(-n-1,t)\) in the second give
\begin{align}
 2\overline f_{r,j}(q^2)
 ={}&\sum_{X,Y\in\Z}
 \{\sgnstd(X-Y+\tfrac12)+\sgnstd(X+Y+\tfrac32)\}
 \notag\\[-2pt]
 &\hspace{25mm}\times
 (-1)^Yq^{2rX^2+(4r-2j)X-(3Y^2+5Y)/2}.
 \label{eq:evenfbarcone}
\end{align}
The second sign functions in \eqref{eq:evenraw} and
\eqref{eq:evenfbarcone} agree because \(0<5/6-j/(2r)<1\).
For the first sign functions, the comparison depends on the sign of
\(1/6-j/(2r)\). If \(3j<r\), they agree. If \(3j=r\), their
difference is \(-1\) on \(X=Y\) and zero elsewhere. If \(3j>r\),
the difference is \(-2\) on \(X=Y\) and zero elsewhere, since
\(-1/3<1/6-j/(2r)<0\). Thus, on setting
\[
 m_{r,j}:=\begin{cases}0,&3j<r,\\1,&3j=r,\\2,&3j>r,\end{cases}
\]
the difference of the first sign functions is
\(-m_{r,j}\1_{\{X=Y\}}\).
The corresponding boundary series is
\begin{align}
 \sum_{X\in\Z}(-1)^X
 q^{\{(4r-3)X^2+(8r-4j-5)X\}/2}
 &=J_{6r-2j-4,\,4r-3}
 \notag\\
 &=-q^{-(2r-2j-1)}J_{2r-2j-1,\,4r-3}.
 \label{eq:evenboundary}
\end{align}
For \(a\in\R\), the elementary product relation
\(j(q^{4r-3+a};q^{4r-3})=-q^{-a}j(q^a;q^{4r-3})\)
gives the second line.  Since
\begin{equation*}
 \kappa^{\mathrm e}_{r,j}-(2r-2j-1)
 =\lambda^{\mathrm e}_{r,j},
\end{equation*}
equations \eqref{eq:evenraw}--\eqref{eq:evenboundary} give
\begin{equation}
 \widehat f_{2j}^{(r)}
 =2\zeta_{12}^{-5}q^{\kappa^{\mathrm e}_{r,j}}
 \overline f_{r,j}(q^2)+\Theta^{\mathrm e}_{r,j}.
 \label{eq:even-theta-comparison}
\end{equation}
The boundary contribution is therefore exactly the one prescribed in
\eqref{eq:eventheta}. Lemma~\ref{lem:evenBailey}, with \(q\) replaced
by \(q^2\), now gives \eqref{eq:evenmain}.
\end{proof}

\subsection{The two signed Nahm branches}
\label{subsec:signed-branches}

\begin{proof}[Proof of Theorem~\ref{thm:nahm-fbar-hat}]
The even-component identity follows from
Proposition~\ref{prop:even-theta} after expanding the quadratic
and linear terms. To prove the product cancellation, we evaluate
the second signed branch using the same tail sums
\[
 N_a:=n_{a+1}+\cdots+n_r\qquad(1\leq a\leq r-1).
\]
Thus \(N_1\geq\cdots\geq N_{r-1}\geq0\) and
\[
 n_{a+1}=N_a-N_{a+1}\quad(1\leq a\leq r-2),
 \qquad n_r=N_{r-1}.
\]
Expanding the quadratic form and the two linear terms gives
\begin{align*}
 \frac12\mathbf n^TB^\vee D^\vee\mathbf n
 +\mathbf n^T\mathbf b^+_{r,j}
 ={}&\frac12n_1^2+n_1N_{r-1}+n_1
 +\sum_{a=1}^{r-1}N_a^2+\frac12N_{r-1}^2
 +\sum_{a=1}^{j-1}N_a+2\sum_{a=j}^{r-1}N_a,\\
 \frac12\mathbf n^TB^\vee D^\vee\mathbf n
 +\mathbf n^T\mathbf b^-_{r,j}
 ={}&\frac12n_1^2+n_1N_{r-1}
 +\sum_{a=1}^{r-1}N_a^2+\frac12N_{r-1}^2
 +\sum_{a=1}^{j-1}N_a+2\sum_{a=j}^{r-1}N_a-N_{r-1}.
\end{align*}
The first expansion is the one used in
\eqref{eq:even-euler-reduction}. The second leads to the auxiliary
series
\begin{equation*}
 \mathcal E^-_{r,j}(q):=
 \sum_{N_1\geq\cdots\geq N_{r-1}\geq0}
 \frac{(-1)^{N_{r-1}}q^{\sum\limits_{a=1}^{r-1}N_a^2+\frac12N_{r-1}^2+\sum\limits_{a=1}^{j-1}N_a+2\sum\limits_{a=j}^{r-1}N_a-N_{r-1}}}
 {\displaystyle
  \prod_{a=1}^{r-2}(q;q)_{N_a-N_{a+1}}
  (-q^{1/2};q)_{N_{r-1}}(q^2;q^2)_{N_{r-1}}}.
\end{equation*}
Euler's identity with \(z=q^{N_{r-1}+1/2}\) evaluates the
\(n_1\)-sum and gives
\[
 \mathcal N^-_{r,j}(q)=(-q^{1/2};q)_\infty
 \mathcal E^-_{r,j}(q).
\]
We will also use the product identity
\[
 \frac{(-q^{1/2};q)_\infty}{(q;q)_\infty}
 =\frac1{J_{1/2,\,2}}.
\]

To evaluate \(\mathcal E^-_{r,j}\), use the transformation sequence
from Lemma~\ref{lem:evenBailey}, but start with
\eqref{eq:pair-G4-starstar}. The parameter-raising transformation
\eqref{eq:bailey-raise} gives the pair relative to \((q^2,q)\)
\begin{equation}
\begin{aligned}
 \alpha_n^{(0)}(q^2;q)
 &=\frac{(1-q^{2n+2})q^{n(n+1)}}{(1-q^2)(1-q)}
   \sum_{t=-n-1}^{n}(-1)^tq^{-(3t^2+7t)/4},\\
 \beta_n^{(0)}(q^2;q)
 &=\frac{(-1)^nq^{n(n-2)/2}}
 {(-q^{1/2};q)_n(q^2;q^2)_n}.
\end{aligned}
 \label{eq:minus-raised-pair}
\end{equation}
Apply \eqref{eq:bailey-chain} \(r-j-1\) times relative to
\((q^2,q)\), then apply the parameter-lowering transformation
\eqref{eq:bailey-lower} with \(a=q^2\), and finally apply
\eqref{eq:bailey-chain} \(j\) times relative to \((q,q)\).
The resulting pair relative to \((q,q)\) is
\begin{align*}
 \alpha_n^{(r)}(q;q)
 =\frac1{1-q}\Bigg(&
 q^{rn^2+(2r-j)n}
 \sum_{t=-n-1}^{n}(-1)^tq^{-(3t^2+7t)/4}\\[-2pt]
 &\quad\quad\quad\quad-q^{rn^2+jn+j-r}
 \sum_{t=-n}^{n-1}(-1)^tq^{-(3t^2+7t)/4}\Bigg)
\end{align*}
and
\begin{align*}
 \beta_n^{(r)}(q;q)
 ={}&\sum_{n\geq N_1\geq\cdots\geq N_{r-1}\geq0}
 \frac{(-1)^{N_{r-1}}q^{\sum_{a=1}^{r-1}N_a^2+\frac12N_{r-1}^2+\sum_{a=1}^{j-1}N_a+2\sum_{a=j}^{r-1}N_a-N_{r-1}}}
 {\displaystyle
  (q;q)_{n-N_1}\prod_{a=1}^{r-2}(q;q)_{N_a-N_{a+1}}
  (-q^{1/2};q)_{N_{r-1}}(q^2;q^2)_{N_{r-1}}}.
\end{align*}
Define the auxiliary double sum
\begin{align*}
 \overline f^{[7]}_{r,j}(q):={}&
 \sum_{n=0}^{\infty}q^{rn^2+(2r-j)n}
 \sum_{t=-n-1}^{n}(-1)^tq^{-(3t^2+7t)/4}\\[-2pt]
 &-\sum_{n=0}^{\infty}q^{rn^2+jn+j-r}
 \sum_{t=-n}^{n-1}(-1)^tq^{-(3t^2+7t)/4}.
\end{align*}
The resulting series have Gaussian decay, so
\eqref{eq:bailey-limit} applies. Since
\((q^2;q)_\infty=(q;q)_\infty/(1-q)\), we obtain
\[
 \mathcal E^-_{r,j}(q)
 =\frac{\overline f^{[7]}_{r,j}(q)}{(q;q)_\infty},
 \qquad
 \mathcal N^-_{r,j}(q)
 =\frac{\overline f^{[7]}_{r,j}(q)}{J_{1/2,\,2}}.
\]

It remains to compare \(\overline f^{[7]}_{r,j}\) with
\(\overline f_{r,j}\). Their difference reduces to endpoint terms
in the finite inner sums. Set
\[
 a_t:=(-1)^tq^{-(3t^2+7t)/4}\qquad(t\in\Z).
\]
The elementary identity
\[
 q^{1/2}(-1)^tq^{-(3t^2+5t)/4}=a_{-t-2}
\]
implies, for every \(n\geq0\),
\begin{align*}
 \sum_{t=-n-1}^{n}(a_t-a_{-t-2})&=a_n-a_{-n-2},\\
 \sum_{t=-n}^{n-1}(a_t-a_{-t-2})&=a_{n-1}-a_{-n-1}.
\end{align*}
Substituting these endpoint terms and combining the ranges
\(k\leq0\) and \(k\geq1\) into bilateral sums, we obtain
\begin{align*}
 \overline f^{[7]}_{r,j}(q)-q^{1/2}\overline f_{r,j}(q)
 ={}&q^{j-r+1}\sum_{k\in\Z}(-1)^k
 \left(q^{(r-3/4)k^2+(1/4-j)k}
       -q^{(r-3/4)k^2-(j+1/4)k}\right)\\
 ={}&q^{j-r+1}
 \left(J_{r-j-1/2,\,2r-3/2}
       -J_{r-j-1,\,2r-3/2}\right),
\end{align*}
where the last equality follows from Jacobi's triple product identity.  Hence
\[
 \mathcal N^-_{r,j}(q)
 =\frac{q^{1/2}\overline f_{r,j}(q)
 +q^{j-r+1}
  \left(J_{r-j-1/2,\,2r-3/2}
       -J_{r-j-1,\,2r-3/2}\right)}{J_{1/2,\,2}}.
\]

The independently established cone comparison
\eqref{eq:even-theta-comparison} gives
\[
 \overline f_{r,j}(q^2)=
 \frac{\zeta_{12}^{5}}2q^{-\kappa^{\mathrm e}_{r,j}}
 \left(\widehat f_{2j}^{(r)}(\tau)
       -\Theta^{\mathrm e}_{r,j}(\tau)\right).
\]
Substituting $q^2$ for $q$ in the two branch evaluations gives
\begin{align*}
 \mathcal N^+_{r,j}(q^2)
 &=\frac{\zeta_{12}^{5}}{2J_{1,4}}
   q^{-\kappa^{\mathrm e}_{r,j}}
   \bigl(\widehat f_{2j}^{(r)}-\Theta^{\mathrm e}_{r,j}\bigr),\\
 \mathcal N^-_{r,j}(q^2)
 &=\frac{\zeta_{12}^{5}}{2J_{1,4}}
   q^{1-\kappa^{\mathrm e}_{r,j}}
   \bigl(\widehat f_{2j}^{(r)}-\Theta^{\mathrm e}_{r,j}\bigr)\\
 &\quad+\frac{q^{2j-2r+2}}{J_{1,4}}
   \bigl(J_{2r-2j-1,4r-3}-J_{2r-2j-2,4r-3}\bigr).
\end{align*}
Subtracting $q\mathcal N^+_{r,j}(q^2)$ from
\(\mathcal N^-_{r,j}(q^2)\) then gives \eqref{eq:sub-eta-Bdual}.
\end{proof}

\subsection{The odd components}
\label{subsec:odd-components}

For $1\leq j\leq r$, define
\begin{align*}
 \kappa^{\mathrm o}_{r,j}
 &=\frac{(4r-2j+1)^2}{8r}-\frac{25}{24},&
 \lambda^{\mathrm o}_{r,j}
 &=\frac{(2j-1)^2}{8r}-\frac1{24},
\end{align*}
and set
\begin{equation}
 \Theta^{\mathrm o}_{r,j}(\tau)=
 \begin{cases}
 0,&3j\leq r+1,\\
 -2\zeta_{12}^{-5}q^{\lambda^{\mathrm o}_{r,j}}
 J_{2r-2j,4r-3},&3j\geq r+2.
 \end{cases}
 \label{eq:oddtheta}
\end{equation}
The cone comparison below relates the sign-kernel series to the
Bailey sum already evaluated in Lemma~\ref{lem:oddBailey}.

The indicator $\1_{\mathcal P}$ is $1$ when $\mathcal P$ is true
and $0$ otherwise. For $n,t\in\Z$, define the cone indicator
\begin{equation}
 \epsilon(n,t):=
 \1_{\{n-t\geq0,\ n+t+1\geq0\}}
 -\1_{\{n-t<0,\ n+t+1<0\}}.
 \label{eq:coneindicator}
\end{equation}
For an integer \(L_0\), define the cone sum
\begin{equation}
 \Cone_{r,L_0}(q):=
 \sum_{n,t\in\Z}\epsilon(n,t)(-1)^t
 q^{2rn^2+L_0n-(3t^2+5t)/2}.
 \label{eq:conesum}
\end{equation}
On the support of \(\epsilon\), one has \(|t|\leq |n|+1\).
Since \(r\geq2\), the quadratic part therefore dominates the
linear terms, and \(\Cone_{r,L_0}(q)\) converges absolutely.

\begin{lemma}
\label{lem:oddcone}
For \(1\leq j\leq r\), define
\begin{equation*}
 c^{\mathrm o}_{r,j}:=
 \frac{(2j-1)^2}{8r}-\frac{25}{24},
 \qquad
 \chi_{r,j}:=\1_{\{3j\geq r+2\}}.
\end{equation*}
Then
\begingroup
\small
\begin{equation}
 \widehat f_{2j-1}^{(r)}(\tau)
 =2\zeta_{12}^{-5}q^{c^{\mathrm o}_{r,j}}
 \Bigl(q^{2r+2j-1}\Cone_{r,4r+2j-1}(q)
 -\chi_{r,j}qJ_{2r-2j,\,4r-3}\Bigr).
 \label{eq:oddconeformula}
\end{equation}
\endgroup
Moreover,
\begin{equation}
 \Cone_{r,4r+2j-1}(q)
 =q^{2-2r-2j}S_{r,j}(q).
 \label{eq:coneS}
\end{equation}
\end{lemma}

\begin{proof}[Proof of Lemma~\ref{lem:oddcone} and Theorem~\ref{thm:euler}]
For \(1\leq j\leq r\), the characteristic has the representative
\begin{equation*}
 a_{2j-1}\equiv
 \binom{(2j-1)/(4r)}{5/6}\pmod{\Z^2}.
\end{equation*}
Adding \((X,Y)^T\), where \(X,Y\in\Z\), to this representative
parametrizes the shifted lattice. The exponential factor in
\eqref{eq:fdef} becomes
\begin{equation*}
 \zeta_{12}^{-5}q^{c^{\mathrm o}_{r,j}}(-1)^Y
 q^{2rX^2+(2j-1)X-(3Y^2+5Y)/2}.
\end{equation*}
The offsets in the two sign functions are
\((2j-1)/(4r)-5/6\) and \((2j-1)/(4r)+5/6\). Comparing these
with the adjacent integers gives
\begin{equation*}
 \begin{cases}
 \sgnstd(X-Y-\tfrac12)+\sgnstd(X+Y+\tfrac12),&3j\leq r+1,\\
 \sgnstd(X-Y-\tfrac12)+\sgnstd(X+Y+\tfrac32),&3j\geq r+2.
 \end{cases}
\end{equation*}
The only possible boundary equality would be
\(3(2j-1)=2r\), which is impossible by parity. Thus no half-weight
term occurs for the odd components. The substitution \(X=n+1\),
\(Y=t\) gives, in both cases,
\begin{equation*}
 \sgnstd(n-t+\tfrac12)
 +\sgnstd(n+t+\tfrac32+\chi_{r,j})
 =2\epsilon(n,t)+2\chi_{r,j}\1_{\{t=-n-2\}}.
\end{equation*}
Here \(\epsilon\) is the cone indicator in
\eqref{eq:coneindicator}. The additional term is supported on the
line \(t=-n-2\). On this line, the substitution \(k=n+1\) and
Jacobi's triple product identity give
\begin{align*}
 &\sum_{n\in\Z}(-1)^{-n-2}
 q^{\displaystyle
 2r(n+1)^2+(2j-1)(n+1)
 -\{3(-n-2)^2+5(-n-2)\}/2}
 \notag\\
 &\hspace{35mm}=-qJ_{2r+2j-3,\,4r-3}
 =-qJ_{2r-2j,\,4r-3},
\end{align*}
where the last equality uses \(j(z;q)=j(q/z;q)\). This proves
\eqref{eq:oddconeformula}. To obtain the second assertion of
Lemma~\ref{lem:oddcone}, evaluate \eqref{eq:conesum} by setting
\(n=N-1\) in the positive cone and \(n=-N-1\) in the negative cone,
with \(N\geq1\). In both cones, the inner sum ranges over
\(-N\leq t\leq N-1\). Replacing \(t\) by \(-t-1\) in each inner
sum gives \(-qT_N(q)\), and the two cone contributions combine to
give \eqref{eq:coneS}.

Combining \eqref{eq:oddconeformula}, \eqref{eq:coneS}, and
\eqref{eq:SoddE} with the identities
\begin{align*}
 c^{\mathrm o}_{r,j}+2r-2j+1
 &=\kappa^{\mathrm o}_{r,j},
 &c^{\mathrm o}_{r,j}+1
 &=\lambda^{\mathrm o}_{r,j}
\end{align*}
gives the boundary correction \eqref{eq:oddtheta} and the
multiple-sum identity \eqref{eq:oddmain}. This proves Theorem~\ref{thm:euler}.
\end{proof}

\subsection{Modular transformations}
\label{subsec:modular-proof}

We construct the real-analytic completion by replacing the sign
functions in \eqref{eq:fdef} with Zwegers' error-function kernels
\cite{Zwegers2002}. The chosen negative vectors satisfy
\[
 Q_r(c_+)=Q_r(c_-)=-\frac{6r(4r-3)}{\gamma_r^2}<0,
 \qquad
 B_r(c_+,c_-)=-\frac{12r(4r+3)}{\gamma_r^2}<0.
\]
Thus \(c_+\) and \(c_-\) lie in the same component of the negative
cone.

\begin{definition}
For \(z\in\C\) and \(v\in\R_{\geq0}\), we use the integrals
\[
 E(z):=2\int_0^z e^{-\pi t^2}\,dt,
 \qquad
 \beta(v):=\int_v^\infty t^{-1/2}e^{-\pi t}\,dt.
\]
These are the functions in
\cite[Definition~1.6 and Lemma~1.7]{Zwegers2002}; in particular,
\[
 E(u)=\sgnstd(u)\bigl(1-\beta(u^2)\bigr)
 \qquad(u\in\R).
\]
For \(c\in\R^2\) with \(Q_r(c)<0\), define
\begin{align*}
 \rho^c(\nu;\tau)&:=
 E\!\left(\frac{B_r(c,\nu)}{\sqrt{-Q_r(c)}}\sqrt y\right),
 &
 \rho^{c_+,c_-}(\nu;\tau)&:=
 \rho^{c_+}(\nu;\tau)-\rho^{c_-}(\nu;\tau).
\end{align*}
For \(a,b\in\R^2\), define
\begin{equation}
 \vartheta_{a,b}^{c_+,c_-}(\tau):=
 \sum_{\nu\in a+\Z^2}\rho^{c_+,c_-}(\nu;\tau)
 e^{2\pi iQ_r(\nu)\tau+2\pi iB_r(\nu,b)}.
 \label{eq:indtheta}
\end{equation}
This is the indefinite theta function of
\cite[Definition~2.1]{Zwegers2002}, with the completed kernel.
Its absolute convergence follows from
\cite[Proposition~2.4]{Zwegers2002}.
\end{definition}

For the characteristics in \eqref{eq:asl}, define the completed
components and their correction terms by
\begin{align}
 H_\ell^{(r)}(\tau)
 &:=\vartheta_{a_\ell,b}^{c_+,c_-}(\tau),
 \notag\\
 G_\ell^{(r)}(\tau)
 &:=\sum_{\nu\in a_\ell+\Z^2}
 \Biggl\{
 \sgnstd(B_r(c_+,\nu))
 \beta\!\left(-\frac{B_r(c_+,\nu)^2}{Q_r(c_+)}y\right)
 \notag\\[-2pt]
 &\hspace{39mm}-
 \sgnstd(B_r(c_-,\nu))
 \beta\!\left(-\frac{B_r(c_-,\nu)^2}{Q_r(c_-)}y\right)
 \Biggr\}
 e^{2\pi iQ_r(\nu)\tau+2\pi iB_r(\nu,b)}.
 \label{eq:Gdef}
\end{align}
The corresponding vectors are
\begin{align*}
 \boldsymbol H^{(r)}
 &:=(H_1^{(r)},\ldots,H_{2r-1}^{(r)})^T,
 &
 \boldsymbol G^{(r)}
 &:=(G_1^{(r)},\ldots,G_{2r-1}^{(r)})^T.
\end{align*}
Since \(E(0)=\sgnstd(0)=0\) and \(\beta(0)=1\),
\cite[Lemma~1.7]{Zwegers2002} gives
\begin{equation}
 \widehat{\boldsymbol f}^{(r)}
 =\boldsymbol H^{(r)}+\boldsymbol G^{(r)}.
 \label{eq:fHG}
\end{equation}
Thus
\(H_\ell^{(r)}=\widehat f_\ell^{(r)}-G_\ell^{(r)}\)
is the real-analytic completion of \(\widehat f_\ell^{(r)}\).

In the scalar formula of Theorem~\ref{thm:modular}, the two
error-function arguments are precisely those in this kernel,
since
\[
 \frac{B_r(c_\pm,\nu)}{\sqrt{-Q_r(c_\pm)}}
 =\sqrt{\frac{24r}{4r-3}}
 \begin{cases}\nu_1-\nu_2,&c_+,\\-\nu_1-\nu_2,&c_-.\end{cases}
\]
For the proof, introduce the matrices
\begin{align}
 T_r&:=\zeta_{24}^{-1}\operatorname{diag}
       (\zeta_{8r}^{1^2},\ldots,\zeta_{8r}^{(2r-1)^2}),
       \label{eq:Tr}\\
 (M_r)_{\ell k}&:=(-1)^{\ell+k}
       \sin\!\left(\frac{\pi\ell k}{2r}\right),
       \qquad1\leq\ell,k\leq2r-1,
       \label{eq:Mr}\\
 C_r(\tau)&:=-\frac{i\tau}{\sqrt r}M_r.
       \label{eq:Cr}
\end{align}
Theorem~\ref{thm:modular} is equivalently
\begin{align}
 \boldsymbol H^{(r)}(\tau+1)&=T_r\boldsymbol H^{(r)}(\tau),
 \label{eq:HT}\\
 \boldsymbol H^{(r)}(-1/\tau)&=C_r(\tau)\boldsymbol H^{(r)}(\tau).
 \label{eq:HS}
\end{align}
The sine matrix satisfies $M_r^T=M_r$ and
$M_r^2=rI_{2r-1}$, where $I_{2r-1}$ is the identity matrix.
The same calculation also gives the transformations of the
holomorphic vector.

\begin{proposition}
\label{prop:holomorphic-transform}
For every $r\geq2$,
\begin{align}
 \boldsymbol G^{(r)}(\tau+1)&=T_r\boldsymbol G^{(r)}(\tau),
 \label{eq:GT}\\
 \widehat{\boldsymbol f}^{(r)}(\tau+1)
 &=T_r\widehat{\boldsymbol f}^{(r)}(\tau),
 \label{eq:fT}\\
 \widehat{\boldsymbol f}^{(r)}(-1/\tau)
 &=C_r(\tau)\widehat{\boldsymbol f}^{(r)}(\tau)
   +\boldsymbol G^{(r)}(-1/\tau)-C_r(\tau)\boldsymbol G^{(r)}(\tau).
 \label{eq:fS}
\end{align}
In particular, the modular discrepancy
\begin{equation}
 \boldsymbol{\mathcal D}_{S,r}(\tau)
 :=\boldsymbol G^{(r)}(-1/\tau)-C_r(\tau)\boldsymbol G^{(r)}(\tau)
 \label{eq:holomorphic-defect}
\end{equation}
is holomorphic on $\Hh$.
\end{proposition}

\begin{proof}[Proof of Theorem~\ref{thm:modular} and Proposition~\ref{prop:holomorphic-transform}]
For arbitrary real characteristics, we abbreviate
\(\vartheta_{a,b}^{c_+,c_-}\) to \(\vartheta_{a,b}\).
With \(b=(0,1/6)^T\), define
\begin{equation*}
 \mathfrak a_u:=\binom{u/(4r)}{-1/6},\qquad
 K_u(\tau):=\vartheta_{\mathfrak a_u,b}(\tau)
 \quad(u\in\Z).
\end{equation*}
Thus \(\mathfrak a_{s_\ell}=a_\ell\) in the notation of
\eqref{eq:asl}. Consider the two integral reflections
\[
 \mathscr D:=\operatorname{diag}(-1,1),\qquad
 \mathscr C:=\operatorname{diag}(1,-1).
\]
Both preserve the quadratic form defined by \(A_r\). Moreover,
\(\mathscr Dc_+=c_-\), \(\mathscr Dc_-=c_+\), and
\(\mathscr Db=b\), whereas
\(\mathscr Cc_+=-c_-\) and \(\mathscr Cc_-=-c_+\).  Changes of variables in
\eqref{eq:indtheta} therefore give
\begin{align}
 \rho^{c_+,c_-}(\mathscr D\nu;\tau)&=-\rho^{c_+,c_-}(\nu;\tau),
 &\vartheta_{\mathscr Da,\mathscr Db}
 &=-\vartheta_{a,b},
 \notag\\*
 \rho^{c_+,c_-}(\mathscr C\nu;\tau)&=\rho^{c_+,c_-}(\nu;\tau),
 &\vartheta_{\mathscr Ca,\mathscr Cb}
 &=\vartheta_{a,b}.
 \label{eq:Creflection}
\end{align}
Since \(\mathfrak a_{u+4r}-\mathfrak a_u=(1,0)^T\) and
\(\mathscr D\mathfrak a_u=\mathfrak a_{-u}\), it follows that
\begin{equation}
 K_{u+4r}=K_u,\qquad K_{-u}=-K_u.
 \label{eq:Ksym}
\end{equation}
The periodicity is also the characteristic-shift law
\cite[Corollary~2.9(1)]{Zwegers2002}. In particular, \(K_0=K_{2r}=0\), and
\begin{equation*}
 H_\ell^{(r)}=K_{s_\ell}=\ve_\ell K_\ell.
\end{equation*}

For \(m,n,u\in\Z\), direct calculation gives
\begin{align*}
 Q_r\!\left(\mathfrak a_u+\binom mn\right)-Q_r(\mathfrak a_u)
 &=um+2rm^2+\frac{n-3n^2}{2}\in\Z,
 \\
 Q_r(\mathfrak a_u)&=\frac{u^2}{8r}-\frac1{24}.
\end{align*}
The kernel \(\rho^{c_+,c_-}\) depends on \(\tau\) only through
\(y\), so
\begin{equation}
 K_u(\tau+1)=\zeta_{24}^{-1}\zeta_{8r}^{u^2}K_u(\tau).
 \label{eq:KuT}
\end{equation}
This agrees with the specialization of
\cite[Corollary~2.9(4)]{Zwegers2002}.
The specialization \(u=s_\ell\) in \eqref{eq:KuT} yields \eqref{eq:HT}.
The factors involving \(\beta\) in \eqref{eq:Gdef} also depend only
on \(y\), so the same congruence proves \eqref{eq:GT}. Combining
\eqref{eq:GT} with \eqref{eq:fHG} proves \eqref{eq:fT}.

For the \(S\)-transformation, we evaluate the discriminant-group
sum. The negative vectors lie in the same component of the negative
cone, as checked above, so the transformation formula applies
to the characteristics under consideration. By
\cite[Corollary~2.9(5)]{Zwegers2002},
\begin{equation}
 \vartheta_{a,b}(-1/\tau)=
 \frac{\tau}{\sqrt{12r}}e^{2\pi iB_r(a,b)}
 \sum_{\gamma\in A_r^{-1}\Z^2/\Z^2}
 \vartheta_{b+\gamma,-a}(\tau),
 \label{eq:ZwegersS}
\end{equation}
because \(\det A_r=-12r\) and \(i(-i\tau)^{2/2}=\tau\).
A set of representatives of the discriminant group is
\begin{equation*}
 \gamma_{u,v}:=\binom{u/(4r)}{-v/3},\qquad
 u\pmod{4r},\quad v\in\{0,1,2\}.
\end{equation*}
For \(s,u\in\Z\), \(t\in\R\), and \(v\in\{0,1,2\}\), set
\begin{equation*}
 \mathfrak a_{u,t}:=\binom{u/(4r)}t,\qquad
 \mathfrak b_s:=\binom{-s/(4r)}{1/6},\qquad
 X^{(s)}_{u,v}:=\vartheta_{\mathfrak a_{u,1/6-v/3},\mathfrak b_s}.
\end{equation*}
Since \(B_r(\mathfrak a_s,b)=1/12\), equation
\eqref{eq:ZwegersS} becomes
\begin{equation}
 K_s(-1/\tau)=\frac{\tau\zeta_{12}}{\sqrt{12r}}
 \sum_{u\bmod4r}\sum_{v=0}^2X^{(s)}_{u,v}(\tau).
 \label{eq:Kdisc}
\end{equation}
The characteristic-shift law
\cite[Corollary~2.9(2)]{Zwegers2002}
\begin{equation}
 \vartheta_{a,b+\mu}
 =e^{2\pi iB_r(a,\mu)}
 \vartheta_{a,b}
 \qquad(\mu\in A_r^{-1}\Z^2)
 \label{eq:ellipticlaw}
\end{equation}
gives
\begin{equation}
 X^{(s)}_{u,v}=e^{-2\pi isu/(4r)}
 \vartheta_{\mathfrak a_{u,1/6-v/3},b}.
 \label{eq:Xphase}
\end{equation}
The reflection \eqref{eq:Creflection} and the shift law
\eqref{eq:ellipticlaw}, with \(\mu=(0,1/3)^T\), give
\begin{equation}
 \vartheta_{\mathfrak a_{u,1/6},b}=\zeta_6^{-1}K_u,
 \qquad
 \vartheta_{\mathfrak a_{u,-1/6},b}=K_u.
 \label{eq:twobranches}
\end{equation}
For the third branch, set
\(Y_u:=\vartheta_{\mathfrak a_{u,-1/2},(0,-1/6)^T}\).
Reflection by \(\mathscr C\), followed by the integral translation of the first
characteristic \cite[Corollary~2.9(1)]{Zwegers2002}, gives
\(\vartheta_{\mathfrak a_{u,-1/2},b}=Y_u\), whereas
\eqref{eq:ellipticlaw} gives
\(\vartheta_{\mathfrak a_{u,-1/2},b}=-Y_u\).
Consequently \(\vartheta_{\mathfrak a_{u,-1/2},b}=0\), and hence
\(X^{(s)}_{u,2}=0\) by \eqref{eq:Xphase}.  Since
\(1+\zeta_6^{-1}=\sqrt3\,\zeta_{12}^{-1}\), equations
\eqref{eq:Kdisc}--\eqref{eq:twobranches} reduce to
\begin{equation*}
 K_s(-1/\tau)=\frac{\tau}{2\sqrt r}
 \sum_{u\bmod4r}e^{-2\pi isu/(4r)}K_u(\tau).
\end{equation*}
Pairing \(u=k\) with \(u=4r-k\) by \eqref{eq:Ksym} yields
\begin{equation*}
 K_s(-1/\tau)=-\frac{i\tau}{\sqrt r}
 \sum_{k=1}^{2r-1}
 \sin\!\left(\frac{\pi sk}{2r}\right)K_k(\tau).
\end{equation*}
Finally, \(s_\ell=\ve_\ell\ell\) and
\(K_k=\ve_kH_k^{(r)}\) imply
\begin{equation*}
 \sin\!\left(\frac{\pi s_\ell k}{2r}\right)K_k
 =(-1)^{\ell+k}\sin\!\left(\frac{\pi\ell k}{2r}\right)H_k^{(r)},
\end{equation*}
which proves \eqref{eq:HS}.  For \(1\leq\ell,m\leq2r-1\), the finite sine
orthogonality identity
\begin{equation*}
 \sum_{k=1}^{2r-1}
 \sin\!\left(\frac{\pi\ell k}{2r}\right)
 \sin\!\left(\frac{\pi m k}{2r}\right)
 =r\,\1_{\{\ell=m\}}
\end{equation*}
proves \(M_r^2=rI_{2r-1}\); symmetry is immediate.  Now
substituting \eqref{eq:fHG} into \eqref{eq:HS} gives \eqref{eq:fS}.
The formula for \(\boldsymbol{\mathcal D}_{S,r}\) follows by
rearranging \eqref{eq:fS}. It is holomorphic because both
\(\widehat{\boldsymbol f}^{(r)}(-1/\tau)\) and
\(C_r(\tau)\widehat{\boldsymbol f}^{(r)}(\tau)\) are holomorphic;
hence its antiholomorphic derivative vanishes. No step requires a
parity restriction on \(r\).
\end{proof}

\subsection{The finite correction and its antiholomorphic derivative}
\label{subsec:finite-correction-proof}

\begin{definition}
For \(\alpha,\delta\in\R\) and \(\tau\in\Hh\), define, following
\cite[Definition~4.1]{Zwegers2002},
\[
 R_{\alpha,\delta}(\tau)
 :=\sum_{\nu\in\alpha+\Z}
 \sgnstd(\nu)\beta\bigl(2\nu^2\operatorname{Im}(\tau)\bigr)
 e^{-\pi i\nu^2\tau-2\pi i\nu\delta},
\]
and
\[
 \theta_{\alpha,\delta}(\tau)
 :=\sum_{\nu\in\alpha+\Z}
 e^{\pi i\nu^2\tau+2\pi i\nu\delta}.
\]
For \(z\in\Hh\), we also use the unary theta series
\cite[Definition~1.14]{Zwegers2002}
\[
 \mathcal U_{\alpha,\delta}(z)
 :=\sum_{\nu\in\alpha+\Z}
 \nu e^{\pi i\nu^2z+2\pi i\nu\delta}.
\]
\end{definition}

The term with \(\nu=0\) in \(R_{\alpha,\delta}\) is zero by our
sign convention. For \(m\in\Z\), the phase factor is constant on
\(\alpha+\Z\), and the substitution \(\nu\mapsto-\nu\) gives the
vanishing identities below:
\[
 R_{\alpha,m}(\tau)=e^{-2\pi i\alpha m}R_{\alpha,0}(\tau),
 \qquad R_{0,0}=R_{1/2,0}=0.
\]
Our notation \(\mathcal U_{\alpha,\delta}\) corresponds to Zwegers'
\(g_{\alpha,\delta}\); we use a different symbol to distinguish it
from \(g_j\).

For \(t\in\R\), Jacobi's triple product identity gives
\begin{equation}
 \theta_{t,\pm1/2}((4r-3)\tau)
 =e^{\pm\pi it}q^{(4r-3)t^2/2}
 J_{(4r-3)(t+1/2),\,4r-3}.
 \label{eq:thetaJ}
\end{equation}
The decomposition of \(G_\ell^{(r)}\) is indexed by the following
positive integers:
\begin{equation*}
 d_r:=\frac{4r-3}{\gamma_r},\qquad
 \mathbf N_r:=\frac{12r(4r-3)}{\gamma_r^2},\qquad
 b_r:=\frac{2r}{\gamma_r}.
\end{equation*}
For \(s\in\Z\) and \(0\leq k<d_r\), define
\begin{align}
 u_{s,k}^{+}&:=-\frac{3s+2r+12rk}{\gamma_r},
 &t_{s,k}^{+}&:=\frac{s+4rk+\frac12}{4r-3},
 \label{eq:utplus}\\
 u_{s,k}^{-}&:=\frac{3s-2r+12rk}{\gamma_r},
 &t_{s,k}^{-}&:=\frac{s+4rk-\frac12}{4r-3}.
 \label{eq:utminus}
\end{align}

\begin{lemma}
\label{lem:Gfinite}
For \(1\leq\ell\leq2r-1\), let \(s=s_\ell\) as in
\eqref{eq:asl}.  Then
\begin{align}
 G_\ell^{(r)}(\tau)
 ={}&\sum_{k=0}^{d_r-1}\Biggl[
 R_{u_{s,k}^{-}/\mathbf N_r,\,b_r}(\mathbf N_r\tau)
 \theta_{t_{s,k}^{-},\,1/2}((4r-3)\tau)
 \notag\\[-2pt]
 &\hspace{27mm}-
 R_{u_{s,k}^{+}/\mathbf N_r,\,b_r}(\mathbf N_r\tau)
 \theta_{t_{s,k}^{+},\,-1/2}((4r-3)\tau)
 \Biggr].
 \label{eq:Gfinite}
\end{align}
This is the specialization of \cite[Proposition~4.3]{Zwegers2002}
to the binary lattice defined by \(A_r\). No additional boundary
term is needed, because both \(G_\ell^{(r)}\) and
\(R_{\alpha,\delta}\) use the convention \(\sgnstd(0)=0\).
\end{lemma}

\begin{proof}
Let \(s=s_\ell\), where \(1\leq\ell\leq2r-1\).
The negative vectors satisfy
\begin{equation*}
 Q_r(c_+)=Q_r(c_-)=-\frac{6r(4r-3)}{\gamma_r^2},\qquad
 B_r(c_+,b)=B_r(c_-,b)=-\frac{2r}{\gamma_r}.
\end{equation*}
For \(c\in\R^2\), define its integral \(B_r\)-orthogonal lattice by
\[
 \langle c\rangle_{\Z}^{\perp}
 :=\{v\in\Z^2:B_r(c,v)=0\}.
\]
For the vectors under consideration, these lattices are
\begin{equation*}
 \langle c_+\rangle_{\Z}^{\perp}=\Z(1,1)^T,\qquad
 \langle c_-\rangle_{\Z}^{\perp}=\Z(1,-1)^T.
\end{equation*}
The quotients
\[
 \Z^2/\bigl(\Z c_++\Z(1,1)^T\bigr),\qquad
 \Z^2/\bigl(\Z c_-+\Z(1,-1)^T\bigr)
\]
both have order \((4r-3)/\gamma_r=d_r\), by their determinants.
Since the second coordinate of each orthogonal generator is
\(\pm1\), every coset has a horizontal representative. Thus the
corresponding affine cosets are represented by
\(v_{s,k}:=(s/(4r)+k,-1/6)^T\), with \(0\leq k<d_r\).
Resolving these representatives into parallel and orthogonal
components gives
\begin{equation*}
 v_{s,k}=\frac{u^+_{s,k}}{\mathbf N_r}c_+
          +t^+_{s,k}(1,1)^T
        =\frac{u^-_{s,k}}{\mathbf N_r}c_-
          +t^-_{s,k}(1,-1)^T.
\end{equation*}
These decompositions yield the four parameters in
\eqref{eq:utplus}--\eqref{eq:utminus}.  Moreover, \(-2Q_r(c_\pm)=\mathbf N_r\).

Define \(e_+:=(1,1)^T\) and \(e_-:=(1,-1)^T\).  For either choice of
sign, the coset decomposition is
\begin{equation*}
 \binom{s/(4r)}{-1/6}+\Z^2
 =\bigsqcup_{k=0}^{d_r-1}
 \bigl(v_{s,k}+\Z c_\pm+\Z e_\pm\bigr).
\end{equation*}
Every point in the \(k\)-th coset has a unique expression
\begin{equation*}
 \nu=\lambda c_\pm+\mu e_\pm,\qquad
 (\lambda,\mu)=
 \left(\frac{u_{s,k}^\pm}{\mathbf N_r}+n,
       t_{s,k}^\pm+m\right),\qquad n,m\in\Z.
\end{equation*}
Orthogonality and direct calculation give
\begin{align}
 Q_r(\lambda c_\pm+\mu e_\pm)
 &=-\frac{\mathbf N_r}{2}\lambda^2+\frac{4r-3}{2}\mu^2,
 &B_r(c_\pm,\lambda c_\pm+\mu e_\pm)&=-\mathbf N_r\lambda,
 \label{eq:factorQ}\\
 B_r(\lambda c_++\mu e_+,b)
 &=-b_r\lambda-\frac\mu2,
 &B_r(\lambda c_-+\mu e_-,b)&=-b_r\lambda+\frac\mu2.
 \label{eq:factorB}
\end{align}
For every \(\lambda\in\R\), including \(\lambda=0\), we have
\(\sgnstd(-\mathbf N_r\lambda)=-\sgnstd(\lambda)\).
Absolute convergence allows separate summation in the parallel and
orthogonal coordinates. Substituting
\eqref{eq:factorQ}--\eqref{eq:factorB} into \eqref{eq:Gdef}, the
\(c_+\)-contribution from the \(k\)-th coset is therefore
\[
 -R_{u_{s,k}^+/\mathbf N_r,b_r}(\mathbf N_r\tau)
  \theta_{t_{s,k}^+,-1/2}((4r-3)\tau),
\]
whereas the minus sign preceding the \(c_-\)-term gives the contribution
\[
 R_{u_{s,k}^-/\mathbf N_r,b_r}(\mathbf N_r\tau)
  \theta_{t_{s,k}^-,1/2}((4r-3)\tau).
\]
Summing over \(0\leq k<d_r\) proves \eqref{eq:Gfinite}.
This is the explicit binary-lattice factorization in
\cite[Proposition~4.3]{Zwegers2002}.
\end{proof}

\begin{proposition}
\label{prop:dbar}
For \(\alpha,\delta\in\R\), \(L>0\), and \(\tau=x+iy\in\Hh\),
\[
 \partial_{\bar\tau}R_{\alpha,\delta}(L\tau)
 =-i\sqrt{\frac{L}{2y}}\,
 \overline{\mathcal U_{\alpha,\delta}(L\tau)}.
\]
Consequently, for \(1\leq\ell\leq2r-1\) and \(s=s_\ell\),
\begin{align*}
 \partial_{\bar\tau}H_\ell^{(r)}(\tau)
 =i\sqrt{\frac{\mathbf N_r}{2y}}
 \sum_{k=0}^{d_r-1}\Bigl[&
 \overline{\mathcal U_{u^-_{s,k}/\mathbf N_r,\,b_r}(\mathbf N_r\tau)}
 \theta_{t^-_{s,k},\,1/2}((4r-3)\tau)\\[-2pt]
 &-
 \overline{\mathcal U_{u^+_{s,k}/\mathbf N_r,\,b_r}(\mathbf N_r\tau)}
 \theta_{t^+_{s,k},\,-1/2}((4r-3)\tau)
 \Bigr].
\end{align*}
Thus, apart from the common factor \(i\sqrt{\mathbf N_r/(2y)}\), the
\(\bar\partial\)-derivative of the completion is a finite sum of
products of conjugates of weight-\(3/2\) unary theta series and
weight-\(1/2\) unary theta series.
\end{proposition}

\begin{proof}
The first identity is the rescaled form of
\cite[Proposition~4.2(4)]{Zwegers2002}. It also follows directly by
differentiating the defining series and using
\(\beta'(v)=-v^{-1/2}e^{-\pi v}\). Termwise differentiation is
justified by the Gaussian decay of the summands and their derivatives
on compact subsets of \(\Hh\); the term with \(\nu=0\) is identically
zero. This argument includes integral \(\alpha\), while reduction
modulo \(\Z\) gives the other cases from the range \(0<\alpha<1\)
in the cited proposition. Applying the identity to
\eqref{eq:Gfinite}, together with the holomorphy of
\(\widehat f_\ell^{(r)}\) and the identity
\(H_\ell^{(r)}=\widehat f_\ell^{(r)}-G_\ell^{(r)}\), gives
\(\partial_{\bar\tau}H_\ell^{(r)}
=-\partial_{\bar\tau}G_\ell^{(r)}\), which proves the formula.
\end{proof}

For \(r=2\), the parameters are \(4r-3=5\), \(\gamma_2=1\), \(d_2=5\), \(\mathbf N_2=120\),
and \(b_2=4\).  Substitution into \eqref{eq:Gfinite}, followed by
\eqref{eq:thetaJ}, gives the \(G\)-formulas in
Corollary~\ref{cor:r2}.

For \(r=3\), the parameters are \(4r-3=9\), \(\gamma_3=3\), \(d_3=3\),
\(\mathbf N_3=36\), and \(b_3=2\).  Substitution into \eqref{eq:Gfinite},
followed by \eqref{eq:thetaJ}, gives the \(G\)-formulas in
Corollary~\ref{cor:r3}. The terms with characteristics \(0\) and
\(1/2\) vanish because \(R_{0,2}=R_{1/2,2}=0\).

\subsection{The normalized multiple-sum vector}

Define the boundary vector by
\[
 \Theta_{2j-1}^{(r)}:=\Theta^{\mathrm o}_{r,j}
 \quad(1\leq j\leq r),\qquad
 \Theta_{2j}^{(r)}:=\Theta^{\mathrm e}_{r,j}
 \quad(1\leq j\leq r-1),
 \qquad
 \boldsymbol\Theta^{(r)}:=(\Theta_1^{(r)},\ldots,
                          \Theta_{2r-1}^{(r)})^T.
\]
The following consequence keeps the holomorphic boundary terms
separate from the nonholomorphic correction.

\begin{corollary}
\label{cor:nahm-euler-vector}
Let $\boldsymbol\Phi^{(r)}:=\widehat{\boldsymbol f}^{(r)}-
\boldsymbol\Theta^{(r)}$. Its components are
\begin{align*}
 \Phi_{2j-1}^{(r)}
 &=-2\zeta_{12}^{-5}q^{\kappa^{\mathrm o}_{r,j}}
       (q;q)_\infty\Eodd_{r,j}(q),&&1\leq j\leq r,\\
 \Phi_{2j}^{(r)}
 &=2\zeta_{12}^{-5}q^{\kappa^{\mathrm e}_{r,j}}
       J_{1,4}\mathcal N^+_{r,j}(q^2),&&1\leq j\leq r-1.
\end{align*}
Moreover,
\[
 \boldsymbol H^{(r)}
 =\boldsymbol\Phi^{(r)}+\boldsymbol\Theta^{(r)}-
  \boldsymbol G^{(r)},
\]
and
\begin{align*}
 \boldsymbol\Phi^{(r)}(\tau+1)
 &=T_r\boldsymbol\Phi^{(r)}(\tau),\\
 \boldsymbol\Phi^{(r)}(-1/\tau)
 &=C_r(\tau)\boldsymbol\Phi^{(r)}(\tau)
   +\boldsymbol{\mathcal D}_{S,r}(\tau)\\
 &\quad+C_r(\tau)\boldsymbol\Theta^{(r)}(\tau)
          -\boldsymbol\Theta^{(r)}(-1/\tau).
\end{align*}
\end{corollary}

\begin{proof}
The component formulas for \(\boldsymbol\Phi^{(r)}\) follow from
\eqref{eq:oddmain} and \eqref{eq:evenmain} upon subtraction of
\(\boldsymbol\Theta^{(r)}\).  Equation~\eqref{eq:fHG} then gives the
displayed identity for \(\boldsymbol H^{(r)}\).  Each Jacobi product occurring in
\(\boldsymbol\Theta^{(r)}\) has integral exponents; only its explicit
monomial prefactor contributes a phase under \(\tau\mapsto\tau+1\).  Moreover,
\[
 e^{2\pi i\lambda^{\mathrm o}_{r,j}}
 =\zeta_{24}^{-1}\zeta_{8r}^{(2j-1)^2},
 \qquad
 e^{2\pi i\lambda^{\mathrm e}_{r,j}}
 =\zeta_{24}^{-1}\zeta_{8r}^{(2j)^2}.
\]
Hence
\(\boldsymbol\Theta^{(r)}(\tau+1)
=T_r\boldsymbol\Theta^{(r)}(\tau)\), and \eqref{eq:fT} gives the
stated \(T\)-transformation for \(\boldsymbol\Phi^{(r)}\).  The remaining
identity follows by substituting
\(\widehat{\boldsymbol f}^{(r)}
=\boldsymbol\Phi^{(r)}+\boldsymbol\Theta^{(r)}\) into
\eqref{eq:fS} and using the definition of \(C_r(\tau)\).
\end{proof}

 For \(a\in\R\) and \(m>0\), define
\[
 \mathcal J_{a,m}(\tau)
 :=q^{(m-2a)^2/(8m)}J_{a,m}.
\]
Jacobi's triple product identity gives
\[
 \mathcal J_{a,m}(\tau)
 =\sum_{n\in\Z}(-1)^n
 q^{\frac m2\left(n+\frac{2a-m}{2m}\right)^2}.
\]
Thus \(\mathcal J_{a,m}\) is a unary theta series with rational
characteristic whenever \(a,m\in\Q\).  The boundary products
in the first two theorems have the exact normalizations
\begin{align*}
 q^{\lambda^{\mathrm e}_{r,j}}J_{2r-2j-1,\,4r-3}
 &=q^{-\frac{(r-3j)^2}{6r(4r-3)}}
   \mathcal J_{2r-2j-1,\,4r-3}(\tau),\\
 q^{\lambda^{\mathrm o}_{r,j}}J_{2r-2j,\,4r-3}
 &=q^{-\frac{(2r-6j+3)^2}{24r(4r-3)}}
   \mathcal J_{2r-2j,\,4r-3}(\tau).
\end{align*}
Thus the boundary products are generally monomial shifts of
normalized unary theta series. Keeping these factors explicit gives the transformation
laws for the normalized multiple-sum vector.

\subsection{An obstruction to a single untwisted Nahm sum}

The multiple-sum formulas above involve product factors and, in the
even case, signs. For the following normalized, boundary-subtracted
factors, these formulas cannot be replaced by a single untwisted
Nahm sum. Define, for \(1\leq j\leq r\) and
\(1\leq j\leq r-1\), respectively,
\begin{align}
 P^{\mathrm o}_{r,j}(q)
 &:=-\frac{\zeta_{12}^{5}}{2}q^{-\kappa^{\mathrm o}_{r,j}}
 \bigl(\widehat f_{2j-1}^{(r)}(\tau)-\Theta^{\mathrm o}_{r,j}(\tau)\bigr)
 =(q;q)_\infty\Eodd_{r,j}(q),
 \label{eq:Podd}\\
 P^{\mathrm e}_{r,j}(q)
 &:=\frac{\zeta_{12}^{5}}{2}q^{-\kappa^{\mathrm e}_{r,j}}
 \bigl(\widehat f_{2j}^{(r)}(\tau)-\Theta^{\mathrm e}_{r,j}(\tau)\bigr)
 =\frac{(q^2;q^2)_\infty}{1+q}\Eeven_{r,j}(q^2).
 \label{eq:Peven}
\end{align}

\begin{proposition}
\label{prop:strictobstruction}
Let \(r\geq2\).  For \(1\leq j\leq r\), the series
\(P^{\mathrm o}_{r,j}\), and for \(1\leq j\leq r-1\), the series
\(P^{\mathrm e}_{r,j}\), admit neither of the forms
\begin{equation}
 Cq^\alpha f_{A,\mathbf b,c}(q^h),
 \qquad
 Cq^\alpha\widetilde f_{A,\mathbf b,c,\mathbf d}(q^h),
 \label{eq:strictcandidate}
\end{equation}
where \(m\in\Z_{>0}\), \(C\in\C\setminus\{0\}\),
\(\alpha\in\R\), and \(h>0\). The data on the right may have any
rank \(m\), but must satisfy the defining rationality and
positive-definiteness conditions for ordinary or generalized Nahm
sums. Equality is understood as equality of generalized series at
\(q=0\). The conclusion excludes a single untwisted sum of the
displayed form; it does not exclude root-of-unity twists, finite
combinations, or scalar modularity on a subgroup.
\end{proposition}

\begin{proof}
In this proof, \(O(q^L)\) denotes a generalized series all of whose
exponents are at least \(L\).
For every nonzero vector in \(\Eodd_{r,j}\), the quadratic part of the
exponent is at least \(1\).  The linear part also contributes at least \(1\),
because all its coefficients are positive integers.  All reciprocal
Pochhammer factors in \eqref{eq:Eodd} contain only nonnegative powers, so
\begin{equation*}
 \Eodd_{r,j}(q)=1+O(q^2).
\end{equation*}
Similarly, for every nonzero vector in the defining sum for
\(\Eeven_{r,j}\), the quadratic and linear
parts together contribute at least \(2\) before the substitution
\(q\mapsto q^2\).  After this substitution, all reciprocal denominator
factors expand in nonnegative powers of \(q\).  Only
\((-q^3;q^2)_{N_{r-1}}^{-1}\) can have alternating coefficients, and it
cannot lower the exponent of a nonzero vector.  Consequently,
\(\Eeven_{r,j}(q^2)=1+O(q^4)\).  Using
\[
 (q;q)_\infty=1-q+O(q^2),\qquad
 \frac{(q^2;q^2)_\infty}{1+q}=1-q+O(q^2),
\]
together with \eqref{eq:Eeven} and
\eqref{eq:Podd}--\eqref{eq:Peven}, we obtain
\begin{equation}
 P^{\mathrm o}_{r,j}(q)=1-q+O(q^2),\qquad
 P^{\mathrm e}_{r,j}(q)=1-q+O(q^2).
 \label{eq:Pinitial}
\end{equation}

Let \(\Phi(q)\) be either untwisted Nahm sum on the right of
\eqref{eq:strictcandidate}. Let \(d_a=1\) in the ordinary case,
and let \(d_a\) denote the \(a\)th entry of \(\mathbf d\) in the
symmetrizable case. Define \(D:=\operatorname{diag}(d_1,\ldots,d_m)\).
By positive definiteness,
\[
 \frac12\mathbf n^TAD\mathbf n
 +\mathbf n^T\mathbf b+c
\]
tends to \(+\infty\) as \(\|\mathbf n\|\to\infty\) in
\(\N^m\). Moreover,
\[
 \frac{1}{(q^{d_1};q^{d_1})_{n_1}(q^{d_2};q^{d_2})_{n_2}\cdots
 (q^{d_m};q^{d_m})_{n_m}}
 =\sum_{k\geq0}c_{\mathbf n,k}q^k,\qquad
 c_{\mathbf n,k}\in\Z_{\geq0}.
\]
Consequently \(\Phi\) has a locally finite generalized expansion
\begin{equation*}
 \Phi(q)=\sum_{\lambda\in\Lambda}a_\lambda q^\lambda,
 \qquad a_\lambda\in\Z_{\geq0},
\end{equation*}
whose support \(\Lambda\subset\R\) is bounded below.  If
\(\lambda_0:=\min\Lambda\), then \(a_{\lambda_0}>0\).

Suppose that one of the series in \eqref{eq:Pinitial} were equal to
\(Cq^\alpha\Phi(q^h)\).  Comparison of the lowest exponents and their
coefficients gives
\[
 \alpha+h\lambda_0=0,\qquad Ca_{\lambda_0}=1.
\]
Thus \(C\) is positive real, and every coefficient of
\(Cq^\alpha\Phi(q^h)\) is nonnegative.  This contradicts the
coefficient \(-1\) of \(q\) in \eqref{eq:Pinitial}, proving the
proposition.
\end{proof}

\begin{remark}
Proposition~\ref{prop:strictobstruction} concerns only the normalized
factors \(P^{\mathrm o}_{r,j}\) and \(P^{\mathrm e}_{r,j}\) in
\eqref{eq:Podd}--\eqref{eq:Peven}; it makes no assertion about
\(\Eodd_{r,j}\), \(\Eeven_{r,j}\), or
\(\mathcal N^+_{r,j}\) separately.
\end{remark}
\subsection{The case \texorpdfstring{\(r=2\)}{r=2}}
\label{subsec:r2}

For \(r=2\), the lattice data specialize to
\begin{equation*}
 A_2=\operatorname{diag}(8,-3),\qquad
 c_+=\binom{3}{8},\qquad c_-=\binom{-3}{8},\qquad
 b=\binom{0}{1/6}.
\end{equation*}
The three characteristics are
\begin{equation*}
 a_1=\binom{1/8}{-1/6},\qquad
 a_2=\binom{-1/4}{-1/6},\qquad
 a_3=\binom{3/8}{-1/6}.
\end{equation*}
For the correction terms, we use the notation
\begin{equation*}
 \omega_2:=\zeta_{20},\qquad
 U_1^{(2)}(\tau):=q^{9/40}J_{1,5},\qquad
 U_2^{(2)}(\tau):=q^{1/40}J_{2,5},\qquad
 \mathscr R_a^{(2)}(\tau):=R_{a,4}(120\tau).
\end{equation*}
For \(\lambda_1,\lambda_2\in\Z\), define the rank-two Euler sum
\begin{equation*}
 \mathcal E^{[2]}_{\lambda_1,\lambda_2}(q):=
 \sum_{n_1,n_2\geq0}
 \frac{q^{\sum_{a,b=1}^{2}\min(a,b)n_an_b+\sum_{b=1}^{2}(\lambda_1+\cdots+\lambda_b)n_b}}
 {(q;q)_{n_1}}.
\end{equation*}
The substitution \(n_1=N_1-N_2\), \(n_2=N_2\) gives
\begin{equation*}
 \Eodd_{2,1}(q)=\mathcal E^{[2]}_{2,2}(q),\qquad
 \Eodd_{2,2}(q)=\mathcal E^{[2]}_{1,1}(q).
\end{equation*}

\begin{corollary}
\label{cor:r2}
The completed vector and its \(T\)-transformation matrix are
\begin{equation*}
 \boldsymbol H^{(2)}(\tau)=
 \begin{pmatrix}
 \vartheta_{(1/8,-1/6)^T,b}^{c_+,c_-}(\tau)\\
 \vartheta_{(-1/4,-1/6)^T,b}^{c_+,c_-}(\tau)\\
 \vartheta_{(3/8,-1/6)^T,b}^{c_+,c_-}(\tau)
 \end{pmatrix},
 \qquad
 T_2=\operatorname{diag}
 \bigl(\zeta_{48},\zeta_{24}^{5},\zeta_{48}^{25}\bigr).
\end{equation*}
The sine matrix in the \(S\)-transformation is
\begin{equation*}
 M_2=
 \begin{pmatrix}
 \frac1{\sqrt2}&-1&\frac1{\sqrt2}\\
 -1&0&1\\
 \frac1{\sqrt2}&1&\frac1{\sqrt2}
 \end{pmatrix},
 \qquad M_2^2=2I_3.
\end{equation*}
Consequently,
\begin{align*}
 \boldsymbol H^{(2)}(\tau+1)
 &=T_2\boldsymbol H^{(2)}(\tau),
 &
 \boldsymbol H^{(2)}(-1/\tau)
 &=-\frac{i\tau}{\sqrt2}M_2\boldsymbol H^{(2)}(\tau),
 \\
 \boldsymbol G^{(2)}(\tau+1)
 &=T_2\boldsymbol G^{(2)}(\tau),
 &
 \widehat{\boldsymbol f}^{(2)}(\tau+1)
 &=T_2\widehat{\boldsymbol f}^{(2)}(\tau).
\end{align*}
The holomorphic vector satisfies
\begin{align*}
 \widehat{\boldsymbol f}^{(2)}
 &=\boldsymbol H^{(2)}+\boldsymbol G^{(2)},
 \\
 \widehat{\boldsymbol f}^{(2)}(-1/\tau)
 &=-\frac{i\tau}{\sqrt2}M_2\widehat{\boldsymbol f}^{(2)}(\tau)
 +\boldsymbol G^{(2)}(-1/\tau)
 +\frac{i\tau}{\sqrt2}M_2\boldsymbol G^{(2)}(\tau).
\end{align*}

The components of the nonholomorphic correction are
\begin{align*}
 G_1^{(2)}={}&
 U_1^{(2)}\!\left\{
 \omega_2^3(\mathscr R_{47/120}^{(2)}-\mathscr R_{17/120}^{(2)})
 +\omega_2^7(\mathscr R_{113/120}^{(2)}-\mathscr R_{23/120}^{(2)})
 \right\}
 \notag\\[-2pt]
 &+U_2^{(2)}\!\left\{
 \omega_2(\mathscr R_{119/120}^{(2)}-\mathscr R_{89/120}^{(2)})
 +\omega_2^9(\mathscr R_{41/120}^{(2)}-\mathscr R_{71/120}^{(2)})
 \right\},
 \\
 G_2^{(2)}={}&
 U_1^{(2)}\!\left\{
 \omega_2^3(\mathscr R_{31/60}^{(2)}-\mathscr R_{1/60}^{(2)})
 +\omega_2^7(\mathscr R_{49/60}^{(2)}-\mathscr R_{19/60}^{(2)})
 \right\}
 \notag\\[-2pt]
 &+U_2^{(2)}\!\left\{
 \omega_2(\mathscr R_{7/60}^{(2)}-\mathscr R_{37/60}^{(2)})
 +\omega_2^9(\mathscr R_{13/60}^{(2)}-\mathscr R_{43/60}^{(2)})
 \right\},
 \\
 G_3^{(2)}={}&
 U_1^{(2)}\!\left\{
 \omega_2^3(\mathscr R_{77/120}^{(2)}-\mathscr R_{107/120}^{(2)})
 +\omega_2^7(\mathscr R_{83/120}^{(2)}-\mathscr R_{53/120}^{(2)})
 \right\}
 \notag\\[-2pt]
 &+U_2^{(2)}\!\left\{
 \omega_2(\mathscr R_{29/120}^{(2)}-\mathscr R_{59/120}^{(2)})
 +\omega_2^9(\mathscr R_{11/120}^{(2)}-\mathscr R_{101/120}^{(2)})
 \right\}.
\end{align*}
Here the argument \(\tau\) is suppressed in \(U_k^{(2)}\) and
\(\mathscr R_a^{(2)}\).  In each component, two terms from
\eqref{eq:Gfinite} vanish because \(J_{0,5}=j(1;q^5)=0\).

The boundary corrections are
\begin{equation*}
 \Theta^{\mathrm o}_{2,1}(\tau)=0,\qquad
 \Theta^{\mathrm e}_{2,1}(\tau)=2\zeta_{12}^{-5}q^{5/24}J_{1,5},
 \qquad
 \Theta^{\mathrm o}_{2,2}(\tau)
 =-2\zeta_{12}^{-5}q^{25/48}J_{0,5}=0.
\end{equation*}
The multiple-sum representations are
\begin{align*}
 \widehat f_1^{(2)}(\tau)
 &=-2\zeta_{12}^{-5}q^{97/48}(q;q)_\infty
 \mathcal E^{[2]}_{2,2}(q),
 \\
 \widehat f_2^{(2)}(\tau)
 &=2\zeta_{12}^{-5}q^{29/24}
 J_{1,4}\mathcal N^+_{2,1}(q^2)
 +2\zeta_{12}^{-5}q^{5/24}J_{1,5},
 \\
 \widehat f_3^{(2)}(\tau)
 &=-2\zeta_{12}^{-5}q^{25/48}(q;q)_\infty
 \mathcal E^{[2]}_{1,1}(q).
\end{align*}
The two identities in Theorem~\ref{thm:nahm-fbar-hat} also give
\begin{align*}
 \mathcal N^-_{2,1}(q^2)
 &=\frac{\zeta_{12}^{5}}{2J_{1,4}}
   q^{-5/24}\widehat f_2^{(2)}(\tau),\\
 \mathcal N^-_{2,1}(q^2)-q\mathcal N^+_{2,1}(q^2)
 &=\frac{J_{1,5}}{J_{1,4}}.
\end{align*}
Thus the boundary correction is nonzero only for
\(\widehat f_2^{(2)}\).  For \(\widehat f_1^{(2)}\), it vanishes
by the first case of \eqref{eq:oddtheta}; for
\(\widehat f_3^{(2)}\), it vanishes because \(J_{0,5}=0\).
\end{corollary}

\subsection{The case \texorpdfstring{\(r=3\)}{r=3}}
\label{subsec:r3}

For \(r=3\), the lattice data specialize to
\begin{equation*}
 A_3=\operatorname{diag}(12,-3),\qquad
 c_+=\binom14,\qquad c_-=\binom{-1}4,\qquad
 b=\binom0{1/6}.
\end{equation*}

For \(\lambda_1,\ldots,\lambda_4\in\Z\), define the fourfold Euler sum
\begin{equation*}
 \mathcal E_{\lambda_1,\ldots,\lambda_4}(q):=
 \sum_{n_1,\ldots,n_4\geq0}
 \frac{q^{\sum_{a,b=1}^{4}\min(a,b)n_an_b+\sum_{b=1}^{4}(\lambda_1+\cdots+\lambda_b)n_b}}
 {(q;q)_{n_1}(q;q)_{n_2}(q;q)_{n_3}}.
\end{equation*}
The change of variables \(n_a=N_a-N_{a+1}\), with \(N_5=0\),
identifies the three odd-component sums with
\(\mathcal E_{2,2,2,2}\), \(\mathcal E_{1,1,2,2}\), and
\(\mathcal E_{1,1,1,1}\). For the nonholomorphic correction, we use
the notation
\begin{equation*}
 \xi:=\zeta_{36},\qquad
 U_k(\tau):=q^{(9-2k)^2/72}J_{k,9}\quad(1\leq k\leq4),
 \qquad
 \mathscr R_a(\tau):=R_{a,2}(36\tau)\quad(a\in\R).
\end{equation*}
With this notation, \(\mathscr R_0=\mathscr R_{1/2}=0\).

\begin{corollary}
\label{cor:r3}
The completed vector and its \(T\)-transformation matrix are
\begin{equation*}
 \boldsymbol H^{(3)}=
 \begin{pmatrix}
 \vartheta_{(1/12,-1/6)^T,b}^{c_+,c_-}\\
 \vartheta_{(-1/6,-1/6)^T,b}^{c_+,c_-}\\
 \vartheta_{(1/4,-1/6)^T,b}^{c_+,c_-}\\
 \vartheta_{(-1/3,-1/6)^T,b}^{c_+,c_-}\\
 \vartheta_{(5/12,-1/6)^T,b}^{c_+,c_-}
 \end{pmatrix},
 \qquad
 T_3=\operatorname{diag}(1,\zeta_8,\zeta_3,\zeta_8^5,1),
\end{equation*}
where the argument \(\tau\) is suppressed in the theta functions.
The sine matrix in the \(S\)-transformation is
\begin{equation*}
 M_3=
 \begin{pmatrix}
 \frac12&-\frac{\sqrt3}{2}&1&-\frac{\sqrt3}{2}&\frac12\\
 -\frac{\sqrt3}{2}&\frac{\sqrt3}{2}&0&-\frac{\sqrt3}{2}&\frac{\sqrt3}{2}\\
 1&0&-1&0&1\\
 -\frac{\sqrt3}{2}&-\frac{\sqrt3}{2}&0&\frac{\sqrt3}{2}&\frac{\sqrt3}{2}\\
 \frac12&\frac{\sqrt3}{2}&1&\frac{\sqrt3}{2}&\frac12
 \end{pmatrix}.
\end{equation*}
Thus \(\boldsymbol H^{(3)}(\tau+1)=T_3\boldsymbol H^{(3)}(\tau)\) and
\(\boldsymbol H^{(3)}(-1/\tau)=-i\tau M_3\boldsymbol H^{(3)}(\tau)/\sqrt3\).

The components of the nonholomorphic correction are
\begin{align*}
 G_1^{(3)}={}&
 \xi^{21}U_3\mathscr R_{1/4}
 +\xi^{15}U_3\mathscr R_{11/12}
 +\xi^7U_1\mathscr R_{11/36}
 -\xi^{13}U_2\mathscr R_{23/36}
 -\xi^{19}U_4\mathscr R_{35/36},
 \\
 G_2^{(3)}={}&
 \xi^{15}U_3\mathscr R_{2/3}
 +\xi^7U_1\mathscr R_{5/9}
 -\xi^{13}U_2\mathscr R_{8/9}
 -\xi^{19}U_4\mathscr R_{2/9},
 \\
 G_3^{(3)}={}&
 \xi^{11}U_1(\mathscr R_{31/36}-\mathscr R_{13/36})
 +\xi^5U_2(\mathscr R_{1/36}-\mathscr R_{19/36})
 +\xi^{17}U_4(\mathscr R_{7/36}-\mathscr R_{25/36}),
 \\
 G_4^{(3)}={}&
 \xi^{33}U_3\mathscr R_{1/6}
 -\xi^7U_1\mathscr R_{1/18}
 +\xi^{13}U_2\mathscr R_{7/18}
 -\xi U_4\mathscr R_{13/18},
 \\
 G_5^{(3)}={}&
 \xi^{33}U_3\mathscr R_{5/12}
 +\xi^3U_3\mathscr R_{3/4}
 -\xi^7U_1\mathscr R_{29/36}
 +\xi^{13}U_2\mathscr R_{5/36}
 -\xi U_4\mathscr R_{17/36}.
\end{align*}
Here the argument \(\tau\) is suppressed in \(U_k\) and
\(\mathscr R_a\).

The multiple-sum representations are
\begin{align*}
 \widehat f_1^{(3)}
 &=-2\zeta_{12}^{-5}q^4(q;q)_\infty
 \mathcal E_{2,2,2,2}(q),
 \\
 \widehat f_2^{(3)}
 &=2\zeta_{12}^{-5}q^{25/8}
 J_{1,4}\mathcal N^+_{3,1}(q^2)
 +\zeta_{12}^{-5}q^{1/8}J_{3,9},
 \\
 \widehat f_3^{(3)}
 &=-2\zeta_{12}^{-5}q^{7/3}(q;q)_\infty
 \mathcal E_{1,1,2,2}(q)
 -2\zeta_{12}^{-5}q^{1/3}J_{2,9},
 \\
 \widehat f_4^{(3)}
 &=2\zeta_{12}^{-5}q^{13/8}
 J_{1,4}\mathcal N^+_{3,2}(q^2)
 +2\zeta_{12}^{-5}q^{5/8}J_{1,9},
 \\
 \widehat f_5^{(3)}
 &=-2\zeta_{12}^{-5}q(q;q)_\infty
 \mathcal E_{1,1,1,1}(q).
\end{align*}
For the two even components, the second identity in
Theorem~\ref{thm:nahm-fbar-hat} gives
\begin{align*}
 \mathcal N^-_{3,1}(q^2)
 &=\frac{\zeta_{12}^{5}}{2J_{1,4}}
   q^{-17/8}\widehat f_2^{(3)}(\tau)
   +\frac{q^{-2}(\frac12J_{3,9}-J_{2,9})}{J_{1,4}},\\
 \mathcal N^-_{3,2}(q^2)
 &=\frac{\zeta_{12}^{5}}{2J_{1,4}}
   q^{-5/8}\widehat f_4^{(3)}(\tau),
\end{align*}
and hence
\begin{align*}
 \mathcal N^-_{3,1}(q^2)-q\mathcal N^+_{3,1}(q^2)
 &=\frac{q^{-2}(J_{3,9}-J_{2,9})}{J_{1,4}},\\
 \mathcal N^-_{3,2}(q^2)-q\mathcal N^+_{3,2}(q^2)
 &=\frac{J_{1,9}}{J_{1,4}}.
\end{align*}
For \(\widehat f_2^{(3)}\), one has \(3j=r\); the coefficient of
the boundary correction is therefore half that in the case \(3j>r\).
\end{corollary}

\subsection{Parity and exact coefficient checks}

The rank of an Euler sum is the number of summation variables
in its defining expression, and \(\lceil x\rceil\) denotes the least
integer not less than \(x\).

\begin{corollary}
\label{cor:parity}
Let \(R\in\Z_{>0}\).  If \(r=2R+1\), the vector has dimension
\(4R+1\).  Its odd components have \(1\leq j\leq2R+1\), Euler
rank \(4R\), and a nonzero theta correction precisely when
\(\lceil(2R+3)/3\rceil\leq j\leq2R\).  Its even components have
\(1\leq j\leq2R\), Nahm rank \(2R+1\), reduced Euler rank \(2R\),
and a nonzero theta correction precisely when
\(\lceil(2R+1)/3\rceil\leq j\leq2R\).

If \(r=2R\), the vector has dimension \(4R-1\).  Its odd
components have \(1\leq j\leq2R\), Euler rank \(4R-2\), and a
nonzero theta correction precisely when
\(\lceil(2R+2)/3\rceil\leq j\leq2R-1\).  Its even components have
\(1\leq j\leq2R-1\), Nahm rank \(2R\), reduced Euler rank
\(2R-1\), and a nonzero theta correction precisely when
\(\lceil2R/3\rceil\leq j\leq2R-1\).

In both cases, \(T_r\) and \(M_r\) are given by
\eqref{eq:Tr}--\eqref{eq:Mr}, and \eqref{eq:fHG} and
\eqref{eq:Gfinite} hold.  When \(3j=r\), the coefficient of the
even theta correction is half that in the case \(3j>r\).
\end{corollary}

Substituting \(r=2R+1\) and \(r=2R\) into the dimension and rank
formulas gives Corollary~\ref{cor:parity}. The ranges of nonzero
boundary corrections follow from \eqref{eq:oddtheta} and
\eqref{eq:eventheta}, together with \(J_{0,4r-3}=0\), which removes
the endpoint \(j=r\) in the odd case. All remaining Jacobi factors
in these ranges have first index strictly between zero and their
modulus, and hence are nonzero on \(\Hh\).
The formulas for \(T_r\), \(M_r\), and the finite nonholomorphic
correction have the same form for both parities.  The value of
\(\gamma_r\) depends only on whether \(3\mid r\); explicitly,
\[
 (d_r,\mathbf N_r,b_r)=
 \begin{cases}
  \bigl(4r-3,\,12r(4r-3),\,2r\bigr),&3\nmid r,\\[2pt]
  \bigl((4r-3)/3,\,4r(4r-3)/3,\,2r/3\bigr),&3\mid r.
 \end{cases}
\]

Corollaries~\ref{cor:r2} and~\ref{cor:r3} follow by specializing
\eqref{eq:Tr}, \eqref{eq:Mr}, \eqref{eq:Gfinite}, and
Theorems~\ref{thm:nahm-fbar-hat} and~\ref{thm:euler}.

As an independent finite check, we compared all $99$ component
identities in Theorems~\ref{thm:nahm-fbar-hat} and~\ref{thm:euler}
for $2\leq r\leq10$. After dividing the $\ell$th component by
$\zeta_{12}^{-5}q^{\ell^2/(8r)-1/24}$, both sides have integral
exponents and integer coefficients. The coefficients through
$q^{30}$ agree in all cases. The calculation includes every
boundary regime, the vanishing odd endpoint, and the half-weight
even cases $(r,j)=(3,1),(6,2),(9,3)$. These finite comparisons are
consistency checks, not steps in the proofs.

\section{Product evaluations and component representations}
\label{sec:proof-four-families}

The representations in Theorem~\ref{thm:four-families} are finite
linear combinations of
\begin{equation}
 \sum_{\mathbf n\in\N^r}
 \frac{\bigl(\prod_{a=1}^r\varepsilon_a^{n_a}\bigr)
       \mathfrak q^{\frac12\mathbf n^T\mathsf A
                    \operatorname{diag}(\boldsymbol d)\mathbf n
                    +\mathbf n^T\mathbf u+c}}
 {\prod_{a=1}^r(\mathfrak q^{d_a};
                         \mathfrak q^{d_a})_{n_a}},
 \qquad \mathbf u\in\Q^r,\quad c\in\Q,
 \label{eq:component-sum-class}
\end{equation}
where the matrix, index vector, and nome are those of the relevant
family. Here $\varepsilon_a=1$ for $A$ and $B$, while
$\varepsilon_a\in\{1,-1\}$ for $A^\vee$ and $B^\vee$.

The linear terms for the four matrix families are obtained by
raising and lowering the relative parameter in Bailey transformations.
All multiple sums in this section converge absolutely and locally
uniformly on \(\Hh\): the reciprocal Pochhammer factors are uniformly
bounded on compact subsets, independently of the summation indices,
and the positive definite quadratic terms dominate the linear terms.
These bounds justify the limits and rearrangements after the Bailey
transformations have been applied at finite indices.

The matrices $A,A^\vee$ and $B,B^\vee$ form two transpose pairs
with the prescribed symmetrizers. Their symmetric products satisfy
\[
 AD=BD-\frac12E_{rr},\qquad
 A^\vee D^\vee=B^\vee D^\vee-E_{rr},\qquad
 D^\vee=2D^{-1}.
\]
The corresponding quadratic form for $AD$ is
\[
 \frac12\mathbf x^TAD\mathbf x
 =\left(x_1+\frac{x_r}{2}\right)^2
 +\frac12\sum_{a=1}^{r-2}
       \left(2\sum_{b=a+1}^{r-1}x_b+x_r\right)^2
 +\frac14x_r^2.
\]
If the right-hand side vanishes, its last term gives $x_r=0$.
Working backwards through the sum gives $x_a=0$ for
$2\leq a\leq r-1$, and the first square then gives $x_1=0$.
Thus $AD$ is positive definite. Adding $x_r^2/4$ proves positive
definiteness for $BD$. The transpose families have the same
property by the congruences
\[
 A^\vee D^\vee=2D^{-1}(AD)D^{-1},\qquad
 B^\vee D^\vee=2D^{-1}(BD)D^{-1}.
\]

For the evaluations, define the linear coefficient vectors by
\begin{align*}
 (\mathbf b^A_{r,j})_a
 &=\begin{cases}
 0,&1\leq a\leq\min\{j,r-1\},\\
 2(a-j),&j+1\leq a\leq r-1,\\
 r-j,&a=r,
 \end{cases}
 &&1\leq j\leq r,\\
  \mathbf b^{A^\vee}_{r,j}
 &:={}(1,1,2,\ldots,j-1,j+1,j+3,\ldots,2r-j-1)^T,
 &&1\leq j\leq r,
 \\
 (\mathbf b^B_{r,j})_a
 &:={}
 \begin{cases}
  0,&1\leq a\leq j,\\
  2(a-j),&j+1\leq a\leq r-1,\\
  \dfrac{2r-2j-3}{2},&a=r,
 \end{cases}
 &&1\leq j\leq r-1,
 \\
 \mathbf b^{B^\vee}_{r,j}
 &:={}(1,1,2,\ldots,j-1,j+1,j+3,\ldots,
 2r-j-3,2r-j-2)^T,
 &&1\leq j\leq r-1.
\end{align*}
Thus the dual linear coefficients are integers. The auxiliary
$A$-identity uses one further vector:
\begin{align*}
 (\mathbf b^{A,0}_{r,j})_a
 &:={}
 \begin{cases}
  1,&a=1,\\
  0,&2\leq a\leq \min\{j,r-1\},\\
  2(a-j),&j+1\leq a\leq r-1,\\
  r-j+1,&a=r,
 \end{cases}
 &&(1\leq j\leq r).
\end{align*}

\subsection{Multi-sum Rogers--Ramanujan-type identities}

The identities in the next two theorems evaluate the nested sums
that arise after Euler summation and separation of the parity of the
last Nahm variable.

\begin{theorem}
\label{thm:bailey-euler-first}
For \(r\geq2\), the following identities hold in the
range $1\leq j\leq r$:
\begin{subequations}
\begin{align}
&\sum_{N_1\geq\cdots\geq N_{r-1}\geq0}
 \frac{q^{N_1^2+\cdots+N_{r-1}^2+N_j+\cdots+N_{r-1}}}
 {\displaystyle
 \prod_{a=1}^{r-2}(q;q)_{N_a-N_{a+1}}
  (q;q)_{N_{r-1}}(q;q^2)_{N_{r-1}}}
 \notag\\*
&\hspace{20mm}
 =\frac{J_{2r+2j,8r-2}-q^jJ_{2r-2j,8r-2}}{J_1},
 \label{eq:bailey-euler-a}
 \\
&\sum_{N_1\geq\cdots\geq N_{r-1}\geq0}
 \frac{q^{\displaystyle
  2N_1^2+\cdots+2N_{r-1}^2
  +2N_1+\cdots+2N_{j-1}
  +4N_j+\cdots+4N_{r-1}}}
 {\displaystyle
 \prod_{a=1}^{r-2}(q^2;q^2)_{N_a-N_{a+1}}
  (q^3;q^2)_{N_{r-1}}(q^4;q^4)_{N_{r-1}}}
 \notag\\*
&\hspace{20mm}
 =q^{2j-2r+1}(1-q)
 \frac{\widetilde J_{2r+2j,4r-1}
       -\widetilde J_{2r-2j,4r-1}}{J_2},
 \label{eq:bailey-euler-b}
 \\
&\sum_{N_1\geq\cdots\geq N_{r-1}\geq0}
 \frac{q^{\displaystyle
  N_1^2+\cdots+N_{r-1}^2
  +N_1+\cdots+N_{j-1}
  +2N_j+\cdots+2N_{r-1}+N_{r-1}}}
 {\displaystyle
 \prod_{a=1}^{r-2}(q;q)_{N_a-N_{a+1}}
  (q;q)_{N_{r-1}}(q^3;q^2)_{N_{r-1}}}
 \notag\\*
&\hspace{20mm}
 =(1-q)
 \frac{q^{j-r}J_{4r+2j,8r-2}+J_{2j,8r-2}
       -q^{j-r}J_{4r-2j,8r-2}}{J_1},
 \label{eq:bailey-euler-c}
 \\
&\sum_{N_1\geq\cdots\geq N_{r-1}\geq0}
 \frac{q^{\displaystyle
  2N_1^2+\cdots+2N_{r-1}^2
  +2N_j+\cdots+2N_{r-1}+2N_{r-1}}}
 {\displaystyle
 \prod_{a=1}^{r-2}(q^2;q^2)_{N_a-N_{a+1}}
  (q;q^2)_{N_{r-1}}(q^4;q^4)_{N_{r-1}}}
 \notag\\*
&\hspace{20mm}
 =\frac{\widetilde J_{4r-2-2j,4r-1}
       -q\widetilde J_{4r-2j,4r-1}
       +q\widetilde J_{2j,4r-1}}{J_2},
 \label{eq:bailey-euler-d}
 \\
&\sum_{N_1\geq\cdots\geq N_{r-1}\geq0}
 \frac{q^{\displaystyle
  N_1^2+\cdots+N_{r-1}^2
  +N_1+\cdots+N_{j-1}
  +2N_j+\cdots+2N_{r-1}}}
 {\displaystyle
  \prod_{a=1}^{r-2}(q;q)_{N_a-N_{a+1}}
  (q^2;q^2)_{N_{r-1}}(-q^{3/2};q)_{N_{r-1}}}
 \notag\\*
&\hspace{20mm}
 =(1+q^{1/2})q^{j-r+1/2}
 \frac{J_{r-j,2r-1/2}-J_{r+j,2r-1/2}}{J_1}.
 \label{eq:bailey-euler-e}
\end{align}
\end{subequations}
\end{theorem}

\begin{proof}
We first prove \eqref{eq:bailey-euler-c} using \eqref{eq:pair-C4}.
The Bailey-pair notation refers to the pair under consideration.
For \(j=r\), applying \eqref{eq:bailey-chain} \(r-1\) times gives
a pair relative to \((q,q)\) with
\begin{align*}
 \alpha_{2n}^{(r-1)}(q;q)
  &=(-1)^nq^{(4r-1)n^2+(2r+1)n},\\
 \alpha_{2n+1}^{(r-1)}(q;q)
  &=(-1)^{n+1}q^{(4r-1)n^2+(6r-3)n+2r-2},\\
 \beta_n^{(r-1)}(q;q)
  &=\sum_{n\geq N_1\geq\cdots\geq N_{r-1}\geq0}
    \frac{
      q^{\sum_{a=1}^{r-1}N_a^2+\sum_{a=1}^{r-2}N_a+2N_{r-1}}}
    {\displaystyle
     (q;q)_{n-N_1}\prod_{a=1}^{r-2}(q;q)_{N_a-N_{a+1}}
     (q;q)_{N_{r-1}}(q^3;q^2)_{N_{r-1}}}.
\end{align*}
The bounds stated at the beginning of this section allow the limit
to pass through the nested sum. Using
\((q;q)_{n-N_1}\to(q;q)_\infty\) for fixed indices and
\eqref{eq:bailey-limit}, we obtain
\begin{align*}
 &\sum_{N_1\geq\cdots\geq N_{r-1}\geq0}
 \frac{q^{N_1^2+\cdots+N_{r-1}^2+
              N_1+\cdots+N_{r-2}+2N_{r-1}}}
 {\displaystyle
  \prod_{a=1}^{r-2}(q;q)_{N_a-N_{a+1}}
  (q;q)_{N_{r-1}}(q^3;q^2)_{N_{r-1}}}
 \\
 &\quad=\frac{1}{(q^2;q)_\infty}
 \left(
  \sum_{n=0}^{\infty}(-1)^n
   q^{(4r-1)n^2+(2r+1)n}
  +\sum_{n=0}^{\infty}(-1)^{n+1}
   q^{(4r-1)n^2+(6r-3)n+2r-2}
 \right)
 \\
 &\quad=\frac{1-q}{(q;q)_\infty}
  \sum_{n\in\Z}(-1)^n
   q^{(4r-1)n^2+(2r+1)n}\\
 &\quad
 =\frac{(1-q)J_{6r,8r-2}}{J_1}.
\end{align*}
At \(j=r\), the two \(J_{2r,8r-2}\)-terms on the right-hand side
of \eqref{eq:bailey-euler-c} cancel, leaving precisely this product.
This proves the endpoint case.

For \(1\leq j\leq r-1\), first raise the relative parameter of
\eqref{eq:pair-C4} from \(q\) to \(q^2\) by
\eqref{eq:bailey-raise}. The resulting pair is
\begin{equation}
\begin{aligned}
 \alpha_{2n}(q^2;q)
  &=\frac{1-q^{4n+2}}{1-q^2}(-1)^nq^{3n^2+n},\\
 \alpha_{2n+1}(q^2;q)
  &=\frac{1-q^{4n+4}}{1-q^2}(-1)^nq^{3n^2+3n}
    (q^{2n+2}-1),\\
 \beta_n(q^2;q)
  &=\frac{q^n}{(q;q)_n(q^3;q^2)_n}.
\end{aligned}
\label{eq:C4-raised-pair}
\end{equation}
Applying \eqref{eq:bailey-chain} \(r-j-1\) times relative to
\((q^2,q)\), starting from \eqref{eq:C4-raised-pair}, yields
\begin{equation}
\begin{aligned}
 \alpha_{2n}^{(r-j-1)}(q^2;q)
  &=\frac{1-q^{4n+2}}{1-q^2}(-1)^n
    q^{(4r-4j-1)n^2+(4r-4j-3)n},\\
 \alpha_{2n+1}^{(r-j-1)}(q^2;q)
  &=\frac{1-q^{4n+4}}{1-q^2}(-1)^n
    q^{(4r-4j-1)n^2+(8r-8j-5)n+3(r-j-1)}
    (q^{2n+2}-1),\\
 \beta_n^{(r-j-1)}(q^2;q)
  &=\sum_{n\geq N_{j+1}\geq\cdots\geq N_{r-1}\geq0}
    \frac{q^{N_{j+1}^2+\cdots+N_{r-1}^2+
                   2N_{j+1}+\cdots+2N_{r-1}+N_{r-1}}}
    {\displaystyle
     (q;q)_{n-N_{j+1}}
     \prod_{a=j+1}^{r-2}(q;q)_{N_a-N_{a+1}}
     (q;q)_{N_{r-1}}(q^3;q^2)_{N_{r-1}}}.
\end{aligned}
\label{eq:C4-chained-pair}
\end{equation}
When \(j=r-1\), there are no chain steps. In that case, the
nested \(\beta\)-sum in \eqref{eq:C4-chained-pair} means the seed
\(\beta_n(q^2;q)\) in \eqref{eq:C4-raised-pair}; it has no free
summation indices.

The parameter-lowering transformation \eqref{eq:bailey-lower}, with
\(a=q^2\), now gives a Bailey pair relative to \((q,q)\):
\begin{equation}
\begin{aligned}
 \alpha_{2n}^{(r-j)}(q;q)
  &=(-1)^nq^{(4r-4j-1)n^2}
    \left(q^{(4r-4j-1)n}+q^{n-r+j}-q^{-n-r+j}\right),\\
 \alpha_{2n+1}^{(r-j)}(q;q)
  &=(-1)^nq^{(4r-4j-1)n^2}
    \left(q^{(8r-8j-1)n+3r-3j}
          -q^{(8r-8j-3)n+3r-3j-2}
          -q^{(4r-4j-1)n}\right),\\
 \beta_n^{(r-j)}(q;q)
  &=q^n\sum_{n\geq N_{j+1}\geq\cdots\geq N_{r-1}\geq0}
    \frac{q^{N_{j+1}^2+\cdots+N_{r-1}^2+
                   2N_{j+1}+\cdots+2N_{r-1}+N_{r-1}}}
    {\displaystyle
     (q;q)_{n-N_{j+1}}
     \prod_{a=j+1}^{r-2}(q;q)_{N_a-N_{a+1}}
     (q;q)_{N_{r-1}}(q^3;q^2)_{N_{r-1}}}.
\end{aligned}
\label{eq:C4-lowered-pair}
\end{equation}
The endpoint convention following \eqref{eq:C4-chained-pair}
also applies to the \(\beta\)-sequence in \eqref{eq:C4-lowered-pair}.

Finally, applying \eqref{eq:bailey-chain} \(j\) times relative to
\((q,q)\), starting from \eqref{eq:C4-lowered-pair}, gives the \(\alpha\)-sequence
\begin{equation*}
\begin{aligned}
 \alpha_{2n}^{(r)}(q;q)
  &=(-1)^nq^{(4r-1)n^2+2jn}
    \left(q^{(4r-4j-1)n}+q^{n-r+j}-q^{-n-r+j}\right),\\
 \alpha_{2n+1}^{(r)}(q;q)
  &=(-1)^nq^{(4r-1)n^2+6jn+2j}
    \left(q^{(8r-8j-1)n+3r-3j}
          -q^{(8r-8j-3)n+3r-3j-2}
          -q^{(4r-4j-1)n}\right).
\end{aligned}
\end{equation*}
The corresponding \(\beta\)-sequence is
\begin{equation*}
\begin{aligned}
 \beta_n^{(r)}(q;q)
  &=\sum_{n\geq N_1\geq\cdots\geq N_{r-1}\geq0}
    \frac{q^{\sum_{a=1}^{r-1}N_a^2+\sum_{a=1}^{j-1}N_a+2\sum_{a=j}^{r-1}N_a+N_{r-1}}}
    {\displaystyle
     (q;q)_{n-N_1}\prod_{a=1}^{r-2}(q;q)_{N_a-N_{a+1}}
     (q;q)_{N_{r-1}}(q^3;q^2)_{N_{r-1}}}.
\end{aligned}
\end{equation*}
Applying \eqref{eq:bailey-limit}, and pairing the even- and
odd-indexed \(\alpha\)-terms into three bilateral series, gives
\begin{align*}
 &\sum_{N_1\geq\cdots\geq N_{r-1}\geq0}
 \frac{q^{N_1^2+\cdots+N_{r-1}^2+
              N_1+\cdots+N_{j-1}+
              2N_j+\cdots+2N_{r-1}+N_{r-1}}}
 {\displaystyle
  \prod_{a=1}^{r-2}(q;q)_{N_a-N_{a+1}}
  (q;q)_{N_{r-1}}(q^3;q^2)_{N_{r-1}}}
 \\
 &\quad=\frac{1}{(q^2;q)_\infty}
 \Biggl(
   \sum_{n\in\Z}(-1)^n
    q^{(4r-1)n^2+(4r-2j-1)n}
  +\sum_{n\in\Z}(-1)^n
    q^{(4r-1)n^2+(2j+1)n-r+j}
 \\
 &\hspace{52mm}
  +\sum_{n\in\Z}(-1)^n
    q^{(4r-1)n^2+(8r-2j-1)n+3r-j}
 \Biggr)
 \\
 &\quad=\frac{1-q}{(q;q)_\infty}
 \Bigl\{
 q^{j-r}(q^{4r+2j},q^{4r-2j-2},q^{8r-2};q^{8r-2})_\infty
 \\
 &\hspace{33mm}
 +(q^{2j},q^{8r-2j-2},q^{8r-2};q^{8r-2})_\infty
-q^{j-r}(q^{4r-2j},q^{4r+2j-2},q^{8r-2};q^{8r-2})_\infty
 \Bigr\}
 \\
 &\quad=(1-q)
 \frac{q^{j-r}J_{4r+2j,8r-2}+J_{2j,8r-2}
       -q^{j-r}J_{4r-2j,8r-2}}{J_1}.
\end{align*}
Together with the endpoint calculation, this proves
\eqref{eq:bailey-euler-c} throughout its stated range.

For \eqref{eq:bailey-euler-a}, start with \eqref{eq:pair-C1}.
At \(j=r\), applying
\eqref{eq:bailey-chain} \(r-1\) times relative to \((1,q)\) gives
\[
 \alpha_0^{(r,r)}=1,\qquad
 \alpha_{2n}^{(r,r)}
 =(-1)^nq^{(4r-1)n^2}(q^n+q^{-n})\quad(n\geq1),
 \qquad
 \alpha_{2n+1}^{(r,r)}=0,
\]
and
\[
 \beta_n^{(r,r)}
 =
 \sum_{n\geq N_1\geq\cdots\geq N_{r-1}\geq0}
 \frac{q^{N_1^2+\cdots+N_{r-1}^2}}
 {(q;q)_{n-N_1}
  \displaystyle\prod_{a=1}^{r-2}(q;q)_{N_a-N_{a+1}}
  (q;q)_{N_{r-1}}(q;q^2)_{N_{r-1}}}.
\]
Consequently,
\[
 \sum_{n=0}^{\infty}\alpha_n^{(r,r)}
 =
 \sum_{n\in\Z}(-1)^nq^{(4r-1)n^2+n}
 =J_{4r,8r-2}.
\]
Since \(J_{0,8r-2}=j(1;q^{8r-2})=0\), the case \(j=r\)
follows from \eqref{eq:bailey-limit}.

For \(1\leq j\leq r-1\), define \(t:=r-j\).  The finite sum
needed in the parameter-raising transformation is
\[
 \sum_{k=0}^{2n}q^{-k^2}\alpha_k(1;q)
 =
 \sum_{k=0}^{2n+1}q^{-k^2}\alpha_k(1;q)
 =(-1)^nq^{-n^2-n}.
\]
The odd-indexed terms vanish, and comparison of successive even
partial sums proves the identity by induction. Hence \eqref{eq:bailey-raise} with
\(a=1\) gives a pair relative to \((q,q)\), with the same
\(\beta\)-sequence and \(\alpha\)-sequence
\[
 \alpha_{2n}^{[0]}
 =\frac{1-q^{4n+1}}{1-q}(-1)^nq^{3n^2-n},
 \qquad
 \alpha_{2n+1}^{[0]}
 =\frac{1-q^{4n+3}}{1-q}(-1)^nq^{3n^2+3n+1}.
\]
After \(t-1=r-j-1\) applications of \eqref{eq:bailey-chain}
relative to \((q,q)\), we obtain
\[
 \alpha_{2n}^{[t-1]}
 =\frac{1-q^{4n+1}}{1-q}(-1)^n
 q^{(4t-1)n^2+(2t-3)n},
\]
\[
 \alpha_{2n+1}^{[t-1]}
 =\frac{1-q^{4n+3}}{1-q}(-1)^n
 q^{(4t-1)n^2+(6t-3)n+2t-1},
\]
and, for \(j\leq r-2\), the corresponding \(\beta\)-sequence is
\[
 \beta_n^{[t-1]}
 =
 \sum_{n\geq N_{j+1}\geq\cdots\geq N_{r-1}\geq0}
 \frac{q^{N_{j+1}^2+\cdots+N_{r-1}^2
              +N_{j+1}+\cdots+N_{r-1}}}
 {(q;q)_{n-N_{j+1}}
  \displaystyle\prod_{a=j+1}^{r-2}(q;q)_{N_a-N_{a+1}}
  (q;q)_{N_{r-1}}(q;q^2)_{N_{r-1}}}.
\]
When \(j=r-1\), no Bailey-chain step is applied, and the
\(\beta\)-sequence is \(\beta_n(1;q)\) from \eqref{eq:pair-C1}.

The parameter-lowering transformation \eqref{eq:bailey-lower}
with \(a=q\) then gives a pair relative to \((1,q)\) with
\[
 \alpha_0^{[t]}=1,
\]
\[
 \alpha_{2n}^{[t]}
 =(-1)^nq^{(4t-1)n^2}
 \left(q^{(2t-1)n}+q^{-(2t-1)n}\right)
 \qquad(n\geq1),
\]
\[
 \alpha_{2n+1}^{[t]}
 =(-1)^n
 \left(
  q^{(4t-1)n^2+(6t-1)n+2t}
  -q^{(4t-1)n^2+(2t-1)n}
 \right),
\]
while \(\beta_n^{[t]}=q^n\beta_n^{[t-1]}\).

Finally, \(j\) applications of \eqref{eq:bailey-chain} relative to
\((1,q)\) give an \(\alpha\)-sequence with
\(\alpha_0^{(r,j)}=1\) and remaining terms
\[
 \alpha_{2n}^{(r,j)}
 =(-1)^nq^{(4r-1)n^2}
 \left(
  q^{(2r-2j-1)n}+q^{-(2r-2j-1)n}
 \right)
 \qquad(n\geq1),
\]
\[
 \alpha_{2n+1}^{(r,j)}
 =(-1)^n
 \left(
 q^{(4r-1)n^2+(6r-2j-1)n+2r-j}
 -q^{(4r-1)n^2+(2r+2j-1)n+j}
 \right).
\]
The \(\beta\)-sequence is
\[
 \beta_n^{(r,j)}
 =
 \sum_{n\geq N_1\geq\cdots\geq N_{r-1}\geq0}
 \frac{
 q^{N_1^2+\cdots+N_{r-1}^2+N_j+\cdots+N_{r-1}}}
 {(q;q)_{n-N_1}
  \displaystyle\prod_{a=1}^{r-2}(q;q)_{N_a-N_{a+1}}
  (q;q)_{N_{r-1}}(q;q^2)_{N_{r-1}}}.
\]

Combining the even- and odd-indexed terms gives two bilateral
series, which we evaluate by Jacobi's triple product identity:
\begin{align*}
 \sum_{n=0}^{\infty}\alpha_n^{(r,j)}
 &=
 \sum_{n\in\Z}(-1)^n
 q^{(4r-1)n^2-(2r-2j-1)n}
 -q^j\sum_{n\in\Z}(-1)^n
 q^{(4r-1)n^2-(2r+2j-1)n}
 \\
 &=J_{2r+2j,8r-2}-q^jJ_{2r-2j,8r-2}.
\end{align*}
The even-indexed terms give the first bilateral series. For the
odd-indexed terms, shifting the index in the first summand by one
produces the positive-index part of the second series; the other
summand supplies its nonpositive-index part. Passing to the limit in
the \(\beta\)-sum gives
\[
 \lim_{n\to\infty}\beta_n^{(r,j)}
 =
 \frac{1}{J_1}
 \sum_{N_1\geq\cdots\geq N_{r-1}\geq0}
 \frac{q^{N_1^2+\cdots+N_{r-1}^2+N_j+\cdots+N_{r-1}}}
 {\displaystyle
  \prod_{a=1}^{r-2}(q;q)_{N_a-N_{a+1}}
  (q;q)_{N_{r-1}}(q;q^2)_{N_{r-1}}}.
\]
Applying \eqref{eq:bailey-limit} with relative parameter \(1\)
proves \eqref{eq:bailey-euler-a}. In the notation fixed in Section~1,
the two products are
\[
 J_{2r+2j,8r-2}
 =(q^{2r+2j},q^{6r-2j-2},q^{8r-2};q^{8r-2})_\infty
\]
and
\[
 J_{2r-2j,8r-2}
 =(q^{2r-2j},q^{6r+2j-2},q^{8r-2};q^{8r-2})_\infty.
\]

For \eqref{eq:bailey-euler-e}, use \eqref{eq:pair-G2}.
When \(j=r\), applying \eqref{eq:bailey-chain}
\(r-1\) times relative to \((q,q)\) gives
\[
 \alpha_n^{(r,r)}
 =(-1)^nq^{(r-1/4)n^2+(r-3/4)n}
 \frac{1-q^{n+1/2}}{1-q^{1/2}},
\]
and
\[
 \beta_n^{(r,r)}
 =
 \sum_{n\geq N_1\geq\cdots\geq N_{r-1}\geq0}
 \frac{q^{N_1^2+\cdots+N_{r-1}^2
              +N_1+\cdots+N_{r-1}}}
 {\displaystyle
  (q;q)_{n-N_1}\prod_{a=1}^{r-2}(q;q)_{N_a-N_{a+1}}
  (-q^{3/2};q)_{N_{r-1}}(q^2;q^2)_{N_{r-1}}}.
\]
Expanding \(1-q^{n+1/2}\), reindexing the first unilateral
series, and using \(J_{0,2r-1/2}=0\) gives
\[
 \sum_{n=0}^{\infty}\alpha_n^{(r,r)}
 =\frac{q^{1/2}}{1-q^{1/2}}
 \left(J_{0,2r-1/2}-J_{2r,2r-1/2}\right).
\]
The case \(j=r\) now follows from \eqref{eq:bailey-limit}.

For \(1\leq j\leq r-1\), let \(t=r-j\). The finite sum in the
raising transformation telescopes to
\[
 \sum_{k=0}^{n}q^{-k^2-k}\alpha_k(q;q)
 =(-1)^nq^{-n(n+3)/4}
 \frac{1-q^{(n+1)/2}}{1-q^{1/2}}.
\]
Thus \eqref{eq:bailey-raise} with \(a=q\) gives a pair relative
to \((q^2,q)\), with \(\beta\)-sequence \(\beta_n(q;q)\) from
\eqref{eq:pair-G2} and
\[
 \alpha_n^{[0]}
 =(-1)^n
 \frac{1-q^{2n+2}}{1-q^2}
 \frac{1-q^{(n+1)/2}}{1-q^{1/2}}
 q^{(3n^2+n)/4}.
\]
After \(t-1=r-j-1\) applications of \eqref{eq:bailey-chain}
relative to \((q^2,q)\), we obtain
\[
 \alpha_n^{[t-1]}
 =(-1)^n
 \frac{1-q^{2n+2}}{1-q^2}
 \frac{1-q^{(n+1)/2}}{1-q^{1/2}}
 q^{(t-1/4)n^2+(2t-7/4)n},
\]
and, for \(j\leq r-2\), the corresponding \(\beta\)-sequence is
\[
 \beta_n^{[t-1]}
 =
 \sum_{n\geq N_{j+1}\geq\cdots\geq N_{r-1}\geq0}
 \frac{q^{N_{j+1}^2+\cdots+N_{r-1}^2
              +2N_{j+1}+\cdots+2N_{r-1}}}
 {(q;q)_{n-N_{j+1}}
  \displaystyle\prod_{a=j+1}^{r-2}(q;q)_{N_a-N_{a+1}}
  (-q^{3/2};q)_{N_{r-1}}(q^2;q^2)_{N_{r-1}}}.
\]
When \(j=r-1\), no Bailey-chain step is applied, so the
\(\beta\)-sequence remains \(\beta_n(q;q)\) from \eqref{eq:pair-G2}.

Applying \eqref{eq:bailey-lower} with \(a=q^2\) then gives a
pair relative to \((q,q)\), with \(\beta\)-sequence
\(q^n\beta_n^{[t-1]}\) and \(\alpha\)-sequence
\[
 \alpha_n^{[t]}
 =\frac{(-1)^n}{1-q^{1/2}}
 \left\{
 q^{(t-1/4)n^2+(2t-3/4)n}(1-q^{(n+1)/2})
 +q^{(t-1/4)n^2-n/4-t+1/2}(1-q^{n/2})
 \right\}.
\]
This formula includes \(n=0\): its second summand vanishes and
\(\alpha_0^{[t]}=1\). Applying
\eqref{eq:bailey-chain} \(j\) times relative to \((q,q)\) yields
\[
 \alpha_n^{(r,j)}
 =\frac{(-1)^n}{1-q^{1/2}}
 \left\{
 q^{(r-1/4)n^2+(2r-j-3/4)n}(1-q^{(n+1)/2})
 +q^{(r-1/4)n^2+(j-1/4)n+j-r+1/2}(1-q^{n/2})
 \right\},
\]
and
\[
 \beta_n^{(r,j)}
 =
 \sum_{n\geq N_1\geq\cdots\geq N_{r-1}\geq0}
 \frac{q^{\displaystyle
  N_1^2+\cdots+N_{r-1}^2
  +N_1+\cdots+N_{j-1}
  +2N_j+\cdots+2N_{r-1}}}
 {\displaystyle
  (q;q)_{n-N_1}\prod_{a=1}^{r-2}(q;q)_{N_a-N_{a+1}}
  (-q^{3/2};q)_{N_{r-1}}(q^2;q^2)_{N_{r-1}}}.
\]
Expanding the factors \(1-q^{(n+1)/2}\) and \(1-q^{n/2}\) gives
\begin{align*}
 &(1-q^{1/2})\sum_{n=0}^{\infty}\alpha_n^{(r,j)}
 \\
 &=\sum_{n=0}^{\infty}(-1)^n
 \bigl\{
 q^{(r-1/4)n^2+(2r-j-3/4)n}
 -q^{(r-1/4)n^2+(2r-j-1/4)n+1/2}
 \\
 &\hspace{34mm}
 +q^{(r-1/4)n^2+(j-1/4)n+j-r+1/2}
 -q^{(r-1/4)n^2+(j+1/4)n+j-r+1/2}
 \bigr\}
 \\
 &=q^{j-r+1/2}
 \left(J_{r-j,2r-1/2}-J_{r+j,2r-1/2}\right).
\end{align*}
In the displayed order, the third and second unilateral series
supply the nonpositive and positive indices of
\(q^{j-r+1/2}J_{r-j,2r-1/2}\), respectively. The fourth and first
supply the nonnegative and negative indices of
\(-q^{j-r+1/2}J_{r+j,2r-1/2}\), respectively.

Passing to the limit as \(n\to\infty\) and applying
\eqref{eq:bailey-limit}, together with
\((q^2;q)_\infty=J_1/(1-q)\), we obtain
\begin{align*}
&\sum_{N_1\geq\cdots\geq N_{r-1}\geq0}
 \frac{q^{\displaystyle
  N_1^2+\cdots+N_{r-1}^2
  +N_1+\cdots+N_{j-1}
  +2N_j+\cdots+2N_{r-1}}}
 {\displaystyle
  \prod_{a=1}^{r-2}(q;q)_{N_a-N_{a+1}}
  (-q^{3/2};q)_{N_{r-1}}(q^2;q^2)_{N_{r-1}}}
 \\
&\hspace{15mm}=
 (1+q^{1/2})q^{j-r+1/2}
 \frac{J_{r-j,2r-1/2}-J_{r+j,2r-1/2}}{J_1}.
\end{align*}
This establishes \eqref{eq:bailey-euler-e} in the remaining range
and completes its proof.

To deduce \eqref{eq:bailey-euler-b}, replace \(q\) by \(q^2\) in
\eqref{eq:bailey-euler-e}.  This gives
\begin{align*}
&\sum_{N_1\geq\cdots\geq N_{r-1}\geq0}
 \frac{q^{\displaystyle
  2N_1^2+\cdots+2N_{r-1}^2
  +2N_1+\cdots+2N_{j-1}
  +4N_j+\cdots+4N_{r-1}}}
 {\displaystyle
  \prod_{a=1}^{r-2}(q^2;q^2)_{N_a-N_{a+1}}
  (-q^3;q^2)_{N_{r-1}}(q^4;q^4)_{N_{r-1}}}
 \\
&\hspace{15mm}
 =(1+q)q^{2j-2r+1}
 \frac{J_{2r-2j,4r-1}-J_{2r+2j,4r-1}}{J_2}.
\end{align*}
All exponents are now integral. Replacing \(q\) by \(-q\)
leaves the numerator on the left unchanged, since its exponent is
even, while \((-q)^2=q^2\) and
\[
 (-(-q)^3;(-q)^2)_n=(q^3;q^2)_n.
\]
Since \(2j-2r+1\) is odd, the monomial factor changes sign.
The two Jacobi products have even first indices and odd modulus
\(4r-1\), and therefore
\[
 j((-q)^a;(-q)^{4r-1})=\widetilde J_{a,4r-1}
 \qquad(a\ \text{even}).
\]
We therefore obtain
\begin{align*}
&\sum_{N_1\geq\cdots\geq N_{r-1}\geq0}
 \frac{q^{\displaystyle
  2N_1^2+\cdots+2N_{r-1}^2
  +2N_1+\cdots+2N_{j-1}
  +4N_j+\cdots+4N_{r-1}}}
 {\displaystyle
  \prod_{a=1}^{r-2}(q^2;q^2)_{N_a-N_{a+1}}
  (q^3;q^2)_{N_{r-1}}(q^4;q^4)_{N_{r-1}}}
 \\
&\hspace{15mm}
 =q^{2j-2r+1}(1-q)
 \frac{\widetilde J_{2r+2j,4r-1}
       -\widetilde J_{2r-2j,4r-1}}{J_2},
\end{align*}
which is \eqref{eq:bailey-euler-b}.

It remains to prove \eqref{eq:bailey-euler-d}. Applying
Lemma~\ref{lem:bailey-lower} with \(a=q\) to the Bailey pair
\eqref{eq:pair-G1-star} gives a pair relative to \((1,q)\) with
\(\alpha_0^{(0)}(1;q)=1\) and
\begin{align*}
 \alpha_n^{(0)}(1;q)
 &=(-1)^n\left(q^{3n(n+1)/4}+q^{3n(n-1)/4}\right)
 \qquad(n\geq1),\\
 \beta_n^{(0)}(1;q)
 &=\frac{q^n}{(-q^{1/2};q)_n(q^2;q^2)_n}
 \qquad(n\geq0).
\end{align*}
The finite identity
\[
 \sum_{k=0}^{n}q^{-k^2}\alpha_k^{(0)}(1;q)
 =(-1)^nq^{-n(n+3)/4}
  \left(1-q^{(n+1)/2}+q^{n+1/2}\right)
\]
follows from the value at \(n=0\) and equality of successive
differences.

For \(j=r\), apply \eqref{eq:bailey-chain} \(r-1\) times
relative to \((1,q)\). For \(1\leq j\leq r-1\), apply
Lemma~\ref{lem:bailey-raise} with \(a=1\) and pass to the limit
\(b\to0\), using the finite identity above to evaluate the sum.
Then apply \eqref{eq:bailey-chain} \(r-j-1\) times relative to
\((q,q)\), followed by Lemma~\ref{lem:bailey-lower} with \(a=q\)
and, finally, \eqref{eq:bailey-chain} \(j\) times relative to \((1,q)\).
In both cases, the resulting pair has \(\alpha_0^{(r,j)}(1;q)=1\)
and, for \(n\geq1\),
\begin{align*}
 \alpha_n^{(r,j)}(1;q)
 ={}&(-1)^nq^{(r-1/4)n^2}
 \Bigl\{q^{(r-j-3/4)n}+q^{-(r-j-3/4)n}\\
 &\quad-q^{1/2}\bigl(q^{(r-j-1/4)n}+q^{-(r-j-1/4)n}\bigr)\\
 &\quad+q^{1/2}\bigl(q^{(r-j+1/4)n}+q^{-(r-j+1/4)n}\bigr)\Bigr\}.
\end{align*}
When \(j=r\), the last two pairs of terms cancel.  Combining
opposite indices, applying Jacobi's triple product identity, and
using \(J_{a,m}=J_{m-a,m}\) gives
\[
 \sum_{n=0}^{\infty}\alpha_n^{(r,j)}(1;q)
 =J_{2r-j-1,2r-1/2}-q^{1/2}J_{j,2r-1/2}
  +q^{1/2}J_{2r-j,2r-1/2}.
\]
The \(\beta\)-sequence obtained from these transformations,
together with \eqref{eq:bailey-limit}, therefore yields
\begin{align*}
 &\sum_{N_1\geq\cdots\geq N_{r-1}\geq0}
 \frac{q^{N_1^2+\cdots+N_{r-1}^2+N_j+\cdots+N_{r-1}+N_{r-1}}}
 {\displaystyle
  \prod_{a=1}^{r-2}(q;q)_{N_a-N_{a+1}}
  (-q^{1/2};q)_{N_{r-1}}(q^2;q^2)_{N_{r-1}}}\\
 &\qquad=
 \frac{J_{2r-j-1,2r-1/2}-q^{1/2}J_{j,2r-1/2}
       +q^{1/2}J_{2r-j,2r-1/2}}{J_1}.
\end{align*}
Replace \(q\) by \(q^2\), and then replace \(q\) in the resulting
identity by \(-q\). After the first substitution, the numerator
exponent is even and the Jacobi products have even first indices and
odd modulus \(4r-1\). The second substitution gives the corresponding
\(\widetilde J\)-products and changes the original factors
\(q^{1/2}\) to \(-q\). This proves \eqref{eq:bailey-euler-d} and
completes the theorem.
\end{proof}

The next two identities are the counterparts needed for the
\(B^\vee\)- and \(B\)-families. The first follows from the
cancellation between the two signed branches; the second uses the
pair \eqref{eq:pair-C5}.

\begin{theorem}
\label{thm:bailey-euler-second}
For \(r\geq2\) and \(1\leq j\leq r-1\),
\begin{subequations}
\begin{align}
&\sum_{N_1\geq\cdots\geq N_{r-1}\geq0}\!
 \frac{(-1)^{N_{r-1}}q^{\displaystyle
  N_1^2+\cdots+N_{r-1}^2
  +N_1+\cdots+N_{j-1}
 +2N_j+\cdots+2N_{r-2}+N_{r-1}
  +\frac12N_{r-1}^2}}
 {\displaystyle
  \prod_{a=1}^{r-2}(q;q)_{N_a-N_{a+1}}
  (q^2;q^2)_{N_{r-1}}(-q^{3/2};q)_{N_{r-1}}}
 \notag\\*
&\hspace{15mm}=(1+q^{1/2})q^{j-r+1}
 \frac{J_{r-j-1/2,2r-3/2}-J_{r-j-1,2r-3/2}}{J_1},
 \label{eq:bailey-euler-f}
 \\
&\sum_{N_1\geq\cdots\geq N_{r-1}\geq0}\!
 \frac{q^{\displaystyle
  N_1^2+\cdots+N_{r-1}^2+N_j+\cdots+N_{r-2}
  +\binom{N_{r-1}}2}}
 {\displaystyle
  \prod_{a=1}^{r-2}(q;q)_{N_a-N_{a+1}}
  (q;q^2)_{N_{r-1}}(q;q)_{N_{r-1}}}
 \notag\\*
&\hspace{14mm}=
 \frac{J_{2r+2j-2,8r-6}-q^jJ_{2r-2j-2,8r-6}}{J_1}.
 \label{eq:bailey-euler-g}
\end{align}
\end{subequations}
\end{theorem}

\begin{proof}
Let \(\mathscr S^\vee_{r,j}(q)\) denote the multiple sum in
\eqref{eq:bailey-euler-f}.  In the definitions of
\(\mathcal N^-_{r,j}\) and \(\mathcal N^+_{r,j}\), introduce the partial sums
\(N_a:=n_{a+1}+\cdots+n_r\) for \(1\leq a\leq r-1\).
The corresponding linear coefficient vectors satisfy
\[
 \mathbf b^+_{r,j}-\mathbf b^-_{r,j}
 =(1,0,\ldots,0,1)^T.
\]
With \(N:=N_{r-1}=n_r\), Euler's identity evaluates the
\(n_1\)-sums as follows:
\begin{align*}
 &\sum_{n_1\geq0}
  \frac{q^{n_1^2/2+n_1N}}{(q;q)_{n_1}}
 -q^{N+1/2}\sum_{n_1\geq0}
  \frac{q^{n_1^2/2+n_1N+n_1}}{(q;q)_{n_1}}
 \\
 &\quad=(-q^{N+1/2};q)_\infty
 -q^{N+1/2}(-q^{N+3/2};q)_\infty
 =(-q^{N+3/2};q)_\infty.
\end{align*}
Summing over the remaining variables gives
\begin{equation*}
 \mathcal N^-_{r,j}(q)-q^{1/2}\mathcal N^+_{r,j}(q)
 =(-q^{3/2};q)_\infty \mathscr S^\vee_{r,j}(q).
\end{equation*}
By Theorem~\ref{thm:nahm-fbar-hat}, the left-hand side equals
\[
 \frac{q^{j-r+1}}{J_{1/2,2}}
 \left(J_{r-j-1/2,2r-3/2}-J_{r-j-1,2r-3/2}\right).
\]
The product definitions give
\[
 \frac{1}{J_{1/2,2}(-q^{3/2};q)_\infty}
 =\frac{1+q^{1/2}}{J_1}.
\]
Combining these formulas proves \eqref{eq:bailey-euler-f}.

To prove \eqref{eq:bailey-euler-g}, apply
\eqref{eq:bailey-raise-infty} with \(a=1\) to
\eqref{eq:pair-C5}.  Since
\[
 \sum_{m=0}^{2n}\alpha_m(1;q)
 =\sum_{m=0}^{2n+1}\alpha_m(1;q)
 =(-1)^nq^{n^2+n},
\]
the resulting Bailey pair relative to \((q,q)\) is
\begin{equation}
\begin{aligned}
 \alpha^{[0]}_{2n}(q;q)
  &=\frac{1-q^{4n+1}}{1-q}(-1)^nq^{n^2-n},
 &&n\geq0,\\
 \alpha^{[0]}_{2n+1}(q;q)
  &=\frac{1-q^{4n+3}}{1-q}(-1)^nq^{n^2-n-1},
 &&n\geq0,\\
 \beta^{[0]}_n(q;q)
  &=\frac{q^{\binom n2-n}}{(q;q)_n(q;q^2)_n}.
\end{aligned}
\label{eq:C5-raised-pair}
\end{equation}

Apply \(k=r-j-1\geq0\) chain steps \eqref{eq:bailey-chain}
to \eqref{eq:C5-raised-pair}, relative to \((q,q)\). The
\(\alpha\)-sequence becomes
\begin{align*}
 \alpha^{[1]}_{2n}(q;q)
 &=\frac{1-q^{4n+1}}{1-q}(-1)^n
   q^{(4r-4j-3)n^2+(2r-2j-3)n},
 &&n\geq0,\\
 \alpha^{[1]}_{2n+1}(q;q)
 &=\frac{1-q^{4n+3}}{1-q}(-1)^n
   q^{(4r-4j-3)n^2+(6r-6j-7)n+2r-2j-3},
 &&n\geq0.
\end{align*}
For the corresponding \(\beta\)-sequence, we use the notation
\(\beta^{[1]}_n(q;q)=\mathscr B^{(r,j)}_n(q)\). If
\(1\leq j\leq r-2\), then
\[
\mathscr B^{(r,j)}_n(q):=
 \sum_{n\geq N_{j+1}\geq\cdots\geq N_{r-1}\geq0}
 \frac{q^{\displaystyle
   N_{j+1}^2+\cdots+N_{r-1}^2+
   N_{j+1}+\cdots+N_{r-2}+
   \binom{N_{r-1}}2}}
 {(q;q)_{n-N_{j+1}}
  \displaystyle\prod_{a=j+1}^{r-2}(q;q)_{N_a-N_{a+1}}
  (q;q)_{N_{r-1}}(q;q^2)_{N_{r-1}}},
\]
whereas for \(j=r-1\) we set
\[
\mathscr B^{(r,r-1)}_n(q):=
 \frac{q^{\binom n2-n}}{(q;q)_n(q;q^2)_n},
\]
corresponding to the case \(k=0\).

The parameter-lowering transformation \eqref{eq:bailey-lower}
with \(a=q\) then gives a pair relative to \((1,q)\), with
\(\beta\)-sequence
\[
 \beta^{[2]}_n(1;q)=q^n\mathscr B^{(r,j)}_n(q).
\]
At \(n=0\), the convention in \eqref{eq:bailey-lower} gives
\(\alpha^{[2]}_0(1;q)=1\).  For the remaining indices, we obtain
\begin{align*}
 \alpha^{[2]}_{2n}(1;q)
 &=(-1)^nq^{(4r-4j-3)n^2}
   \left(q^{(2r-2j-1)n}+q^{-(2r-2j-1)n}\right),
 &&n\geq1,\\
 \alpha^{[2]}_{2n+1}(1;q)
 &=(-1)^nq^{(4r-4j-3)n^2+(6r-6j-5)n+2r-2j-2}
 \\
 &\quad-(-1)^nq^{(4r-4j-3)n^2+(2r-2j-1)n},
 &&n\geq0.
\end{align*}
Finally, applying \eqref{eq:bailey-chain} \(j\) times relative to
\((1,q)\) gives a pair with \(\alpha^{[3]}_0(1;q)=1\) and
\begin{equation*}
\begin{aligned}
 \alpha^{[3]}_{2n}(1;q)
 &=(-1)^nq^{(4r-3)n^2}
   \left(q^{(2r-2j-1)n}+q^{-(2r-2j-1)n}\right),
 &&n\geq1,\\
 \alpha^{[3]}_{2n+1}(1;q)
 &=(-1)^nq^{(4r-3)n^2+(6r-2j-5)n+2r-j-2}-(-1)^nq^{(4r-3)n^2+(2r+2j-1)n+j},
 &&n\geq0,
\end{aligned}
\end{equation*}
with \(\beta\)-sequence
\begin{align*}
 \beta^{[3]}_n(1;q)
 &=\sum_{n\geq N_1\geq\cdots\geq N_{r-1}\geq0}
 \frac{q^{\displaystyle
   N_1^2+\cdots+N_{r-1}^2+
   N_j+\cdots+N_{r-2}+
   \binom{N_{r-1}}2}}
 {\displaystyle
  (q;q)_{n-N_1}\prod_{a=1}^{r-2}(q;q)_{N_a-N_{a+1}}
  (q;q)_{N_{r-1}}(q;q^2)_{N_{r-1}}}.
\end{align*}

Passing to the limit in the nested \(\beta\)-sum and applying
\eqref{eq:bailey-limit} with relative parameter \(1\) shows that
the multiple sum in \eqref{eq:bailey-euler-g} is
\[
 \frac{1}{(q;q)_\infty}
 \sum_{m=0}^{\infty}\alpha^{[3]}_m(1;q).
\]
The formulas for \(\alpha^{[3]}_m\), together with
\(\alpha^{[3]}_0=1\), give
\begin{align*}
 \sum_{m=0}^{\infty}\alpha^{[3]}_m(1;q)
 &=\sum_{n\in\Z}(-1)^n
   q^{(4r-3)n^2+(2r-2j-1)n}
  -q^j\sum_{n\in\Z}(-1)^n
   q^{(4r-3)n^2+(2r+2j-1)n}\\
 &=J_{6r-2j-4,8r-6}
   -q^jJ_{6r+2j-4,8r-6}\\
 &=J_{2r+2j-2,8r-6}
   -q^jJ_{2r-2j-2,8r-6},
\end{align*}
where the last equality follows from \(J_{a,m}=J_{m-a,m}\).
This proves \eqref{eq:bailey-euler-g}.
\end{proof}

\subsection{Evaluations for the four matrix families}

For the \(A\)-family, Euler summation and separation of the parity
of \(n_r\) reduce the Nahm sums to the identities in
Theorem~\ref{thm:bailey-euler-first}. For \(A^\vee\), Euler summation
alone gives the required nested sum. This yields the following
three evaluations.

\begin{theorem}
\label{thm:nahm-A-families}
For \(r\geq2\) and \(1\leq j\leq r\),
\begin{align}
&\sum_{\mathbf n\in\Z_{\geq0}^r}
 \frac{q^{\frac12\mathbf n^TAD\mathbf n+
                 \mathbf n^T\mathbf b^A_{r,j}}}
 {\displaystyle
  (q^2;q^2)_{n_1}(q^2;q^2)_{n_2}\cdots(q^2;q^2)_{n_{r-1}}(q;q)_{n_r}}
 \notag\\*[-2pt]
&\hspace{9mm}=
 \frac{J_{4r+4j,16r-4}-q^{2j}J_{4r-4j,16r-4}}{J_{1,4}}
 +q^{j-r/2+1/2}
 \frac{\widetilde J_{2r+2j,4r-1}
       -\widetilde J_{2r-2j,4r-1}}{J_{2,4}},
 \label{eq:A-primary}
 \\
&\sum_{\mathbf n\in\Z_{\geq0}^r}
 \frac{q^{\frac12\mathbf n^TAD\mathbf n+
                 \mathbf n^T\mathbf b^{A,0}_{r,j}}}
 {\displaystyle
  (q^2;q^2)_{n_1}(q^2;q^2)_{n_2}\cdots(q^2;q^2)_{n_{r-1}}(q;q)_{n_r}}
 \notag\\*[-2pt]
&\hspace{9mm}=
 \frac{q^{(2j-r+1)/2}J_{8r+4j,16r-4}
       +q^{(3r-2j+1)/2}J_{4j,16r-4}
       -q^{(2j-r+1)/2}J_{8r-4j,16r-4}}{J_{1,4}}
 \notag\\*[-2pt]
&\hspace{22mm}+
 \frac{\widetilde J_{4r-2-2j,4r-1}
       -q\widetilde J_{4r-2j,4r-1}
       +q\widetilde J_{2j,4r-1}}{J_{2,4}},
 \label{eq:A-secondary}
 \\
&\sum_{\mathbf n\in\Z_{\geq0}^r}
 \frac{q^{\frac12\mathbf n^TA^\vee D^\vee\mathbf n+
                 \mathbf n^T\mathbf b^{A^\vee}_{r,j}}}
 {\displaystyle
  (q;q)_{n_1}(q;q)_{n_2}\cdots(q;q)_{n_{r-1}}(q^2;q^2)_{n_r}}
 =q^{j-r+1/2}
 \frac{J_{r-j,2r-1/2}-J_{r+j,2r-1/2}}{J_{1/2,2}}.
 \label{eq:A-dual}
\end{align}
\end{theorem}

\begin{proof}
Let \(K\) denote the left-hand side of \eqref{eq:A-primary}.
For \(2\leq a\leq r-1\), consider the auxiliary sums
\[
 L_a:=n_a+\cdots+n_{r-1}+\frac{n_r}{2},
 \qquad L_r:=\frac{n_r}{2}.
\]
Direct expansion gives
\[
 \frac12\mathbf n^TAD\mathbf n+\mathbf n^T\mathbf b^A_{r,j}
 =n_1^2+n_1n_r+2\sum_{a=2}^{r-1}L_a^2
  +\frac12n_r^2+2\sum_{a=j+1}^{r}L_a.
\]
Replacing \(\mathbf b^A_{r,j}\) by \(\mathbf b^{A,0}_{r,j}\)
adds \(n_1+n_r\) to the exponent. The empty-sum convention
applies throughout, including \(j=r\) and \(r=2\).

Separating the two cases \(n_r=2m+\varepsilon\), where \(m\geq0\)
and \(\varepsilon\in\{0,1\}\), leads to the variables
\[
 N_a:=n_{a+1}+\cdots+n_{r-1}+m\quad(1\leq a\leq r-2),
 \qquad N_{r-1}:=m.
\]
Thus \(N_1\geq\cdots\geq N_{r-1}\geq0\) and
\(n_{a+1}=N_a-N_{a+1}\) for \(1\leq a\leq r-2\).  We use
\(\sum_{\mathbf N}\) for summation over this weakly decreasing
chain.  For the sum in \eqref{eq:A-primary}, the exponents
in the two parity classes are
\begin{align*}
 E_0={}&n_1^2+2mn_1+2\sum_{a=1}^{r-1}N_a^2
              +2\sum_{a=j}^{r-1}N_a,\\
 E_1={}&n_1^2+(2m+1)n_1+2\sum_{a=1}^{r-1}N_a^2
              +2\sum_{a=1}^{j-1}N_a+4\sum_{a=j}^{r-1}N_a
              +\frac{3r-2j-1}{2}.
\end{align*}
Euler's identity, in the form
\[
 \sum_{n\geq0}\frac{q^{n^2+(2m+d)n}}{(q^2;q^2)_n}
 =(-q^{2m+d+1};q^2)_\infty\qquad(d=0,1,2),
\]
evaluates the sum over \(n_1\).  Let \(K_0\) and
\(K_1\) be the contributions with even and odd \(n_r\),
respectively.  Then
\begin{equation}
 K=K_0+K_1.
 \label{eq:A-parity-split}
\end{equation}
We also use the product decompositions
\begin{align*}
 \frac{(-q^{2m+1};q^2)_\infty}{(q;q)_{2m}}
 &=\frac{(-q;q^2)_\infty}
 {(q^2;q^2)_m(q^2;q^4)_m},\\
 \frac{(-q^{2m+2};q^2)_\infty}{(q;q)_{2m+1}}
 &=\frac{(-q^2;q^2)_\infty}
 {(1-q)(q^3;q^2)_m(q^4;q^4)_m}.
\end{align*}
The even contribution to \eqref{eq:A-primary} is therefore
\begin{equation}
\begin{aligned}
 K_0={}&(-q;q^2)_\infty\sum_{\mathbf N}
 \frac{q^{2N_1^2+\cdots+2N_{r-1}^2
             +2N_j+\cdots+2N_{r-1}}}
 {\prod_{a=1}^{r-2}(q^2;q^2)_{N_a-N_{a+1}}(q^2;q^2)_m(q^2;q^4)_m}.
\end{aligned}
\label{eq:A-even-sum}
\end{equation}
The odd part is
\begingroup\small
\begin{equation}
\begin{aligned}
 K_1={}&
 \frac{q^{(3r-2j-1)/2}(-q^2;q^2)_\infty}{1-q}
 \sum_{\mathbf N}
 \frac{q^{2N_1^2+\cdots+2N_{r-1}^2
 +2N_1+\cdots+2N_{j-1}+4N_j+\cdots+4N_{r-1}}}
 {\prod_{a=1}^{r-2}(q^2;q^2)_{N_a-N_{a+1}}(q^3;q^2)_m(q^4;q^4)_m}.
\end{aligned}
\label{eq:A-odd-sum}
\end{equation}
\endgroup
We evaluate \(K_0\) using \eqref{eq:bailey-euler-a}
with \(q\) replaced by \(q^2\), and \(K_1\) using
\eqref{eq:bailey-euler-b}.  The product identities
\[
 \frac{(-q;q^2)_\infty}{J_2}=\frac1{J_{1,4}},
 \qquad
 \frac{(-q^2;q^2)_\infty}{J_2}=\frac1{J_{2,4}}
\]
then give
\begin{align}
 K_0&=
 \frac{J_{4r+4j,16r-4}-q^{2j}J_{4r-4j,16r-4}}{J_{1,4}},
 \label{eq:A-even-eval}\\
 K_1&=q^{j-r/2+1/2}
 \frac{\widetilde J_{2r+2j,4r-1}
       -\widetilde J_{2r-2j,4r-1}}{J_{2,4}}.
 \label{eq:A-odd-eval}
\end{align}
Substituting \eqref{eq:A-even-eval} and \eqref{eq:A-odd-eval}
into \eqref{eq:A-parity-split} proves \eqref{eq:A-primary}.

For \eqref{eq:A-secondary}, retain the same variables \(m\),
\(\varepsilon\), and \(\mathbf N\). Denote the even and odd
contributions by
\(K_0^\circ\) and \(K_1^\circ\), respectively.
The exponent differs from that in \eqref{eq:A-primary} by
\(n_1+n_r\), so the sum over \(n_1\) becomes
\[
 \sum_{n_1\geq0}
 \frac{q^{n_1^2+(2m+\varepsilon+1)n_1}}{(q^2;q^2)_{n_1}}
 =(-q^{2m+\varepsilon+2};q^2)_\infty.
\]
The corresponding product decompositions are
\begin{align*}
 \frac{(-q^{2m+2};q^2)_\infty}{(q;q)_{2m}}
 &=\frac{(-q^2;q^2)_\infty}
 {(q;q^2)_m(q^4;q^4)_m},\\
 \frac{(-q^{2m+3};q^2)_\infty}{(q;q)_{2m+1}}
 &=\frac{(-q^3;q^2)_\infty}
 {(1-q)(q^2;q^2)_m(q^6;q^4)_m}.
\end{align*}
It follows that
\begin{align*}
 K_0^\circ={}&(-q^2;q^2)_\infty\sum_{\mathbf N}
 \frac{q^{2N_1^2+\cdots+2N_{r-1}^2
 +2N_j+\cdots+2N_{r-1}+2m}}
 {\displaystyle
  \prod_{a=1}^{r-2}(q^2;q^2)_{N_a-N_{a+1}}
  (q;q^2)_m(q^4;q^4)_m},\\
 K_1^\circ={}&
 \frac{q^{(3r-2j+1)/2}(-q^3;q^2)_\infty}{1-q}
 \sum_{\mathbf N}
 \frac{q^{2N_1^2+\cdots+2N_{r-1}^2
 +2N_1+\cdots+2N_{j-1}
 +4N_j+\cdots+4N_{r-1}+2m}}
 {\displaystyle
  \prod_{a=1}^{r-2}(q^2;q^2)_{N_a-N_{a+1}}
  (q^2;q^2)_m(q^6;q^4)_m}.
\end{align*}
Equation~\eqref{eq:bailey-euler-d} evaluates
\(K_0^\circ\), and \eqref{eq:bailey-euler-c} with \(q\)
replaced by \(q^2\) evaluates \(K_1^\circ\).  Using
\[
 \frac{(1-q^2)(-q^3;q^2)_\infty}{(1-q)J_2}
 =\frac1{J_{1,4}},
\]
we obtain
\begin{align*}
 K_0^\circ={}&
 \frac{\widetilde J_{4r-2-2j,4r-1}
       -q\widetilde J_{4r-2j,4r-1}
       +q\widetilde J_{2j,4r-1}}{J_{2,4}},\\
 K_1^\circ={}&
 \frac{q^{(2j-r+1)/2}J_{8r+4j,16r-4}
       +q^{(3r-2j+1)/2}J_{4j,16r-4}
       -q^{(2j-r+1)/2}J_{8r-4j,16r-4}}{J_{1,4}}.
\end{align*}
Adding these expressions proves \eqref{eq:A-secondary}.

For the remaining identity \eqref{eq:A-dual}, introduce the tail sums
\(N_a:=n_{a+1}+\cdots+n_r\) for \(1\leq a\leq r-1\).  Direct
expansion gives
\begin{align*}
 \frac12\mathbf n^TA^\vee D^\vee\mathbf n
 +\mathbf n^T\mathbf b^{A^\vee}_{r,j}
={}&\frac12n_1^2+n_1N_{r-1}+n_1
 +\sum_{a=1}^{r-1}N_a^2
 +\sum_{a=1}^{j-1}N_a+2\sum_{a=j}^{r-1}N_a.
\end{align*}
Euler's identity gives
\[
 \sum_{n_1\geq0}
 \frac{q^{n_1^2/2+n_1N_{r-1}+n_1}}{(q;q)_{n_1}}
 =(-q^{N_{r-1}+3/2};q)_\infty.
\]
Moreover,
\[
 \frac{(-q^{N_{r-1}+3/2};q)_\infty}
 {(q^2;q^2)_{N_{r-1}}}
 =\frac{(-q^{3/2};q)_\infty}
 {(q^2;q^2)_{N_{r-1}}(-q^{3/2};q)_{N_{r-1}}}.
\]
Factoring out \((-q^{3/2};q)_\infty\) and applying
\eqref{eq:bailey-euler-e}, we find that the left-hand side of
\eqref{eq:A-dual} equals
\[
 (-q^{3/2};q)_\infty(1+q^{1/2})q^{j-r+1/2}
 \frac{J_{r-j,2r-1/2}-J_{r+j,2r-1/2}}{J_1}.
\]
Finally,
\[
 \frac{(1+q^{1/2})(-q^{3/2};q)_\infty}{J_1}
 =\frac1{J_{1/2,2}},
\]
which proves \eqref{eq:A-dual}.
\end{proof}

The evaluation for $B$ follows from
Theorem~\ref{thm:bailey-euler-second}. For $B^\vee$, Euler
summation identifies the required series with the signed
cancellation in Theorem~\ref{thm:nahm-fbar-hat}. These two
arguments give the following evaluations.

\begin{theorem}
\label{thm:nahm-B-families}
For \(r\geq2\) and \(1\leq j\leq r-1\),
\begin{align}
&\sum_{\mathbf n\in\Z_{\geq0}^r}
 \frac{q^{\frac12\mathbf n^TBD\mathbf n+
                 \mathbf n^T\mathbf b^B_{r,j}}}
 {\displaystyle
  (q^2;q^2)_{n_1}(q^2;q^2)_{n_2}\cdots(q^2;q^2)_{n_{r-1}}(q;q)_{n_r}}
 \notag\\*[-2pt]
&\hspace{7mm}=
 \frac{J_{4(r+j-1),4(4r-3)}
       -q^{2j}J_{4(r-j-1),4(4r-3)}}{J_{1,4}}
 \notag\\*[-2pt]
&\hspace{20mm}+
 q^{(4j-2r+1)/4}
 \frac{\widetilde J_{2r-2j-1,4r-3}
       -\widetilde J_{2r-2j-2,4r-3}}{J_{2,4}},
 \label{eq:B-primary}
 \\
&\sum_{\mathbf n\in\Z_{\geq0}^r}
 \frac{(-1)^{n_r}
 q^{\frac12\mathbf n^TB^\vee D^\vee\mathbf n+
                 \mathbf n^T\mathbf b^{B^\vee}_{r,j}}}
 {\displaystyle
  (q;q)_{n_1}(q;q)_{n_2}\cdots(q;q)_{n_{r-1}}(q^2;q^2)_{n_r}}
 =q^{j-r+1}
 \frac{J_{r-j-1/2,2r-3/2}
       -J_{r-j-1,2r-3/2}}{J_{1/2,2}}.
 \label{eq:B-dual}
\end{align}
\end{theorem}

\begin{proof}
We first prove \eqref{eq:B-primary}.  For \(2\leq a\leq r-1\), define
\[
 L_a:=n_a+\cdots+n_{r-1}+\frac{n_r}{2}.
\]
A direct expansion gives
\[
 \frac12\mathbf n^TBD\mathbf n+\mathbf n^T\mathbf b^B_{r,j}
 =n_1^2+n_1n_r+2\sum_{a=2}^{r-1}L_a^2+\frac34n_r^2
  +2\sum_{a=j+1}^{r-1}L_a-\frac12n_r.
\]
For the two parity classes \(n_r=2m+\varepsilon\), where \(m\geq0\)
and \(\varepsilon\in\{0,1\}\), introduce
\[
 N_a:=n_{a+1}+\cdots+n_{r-1}+m\quad(1\leq a\leq r-2),
 \qquad N_{r-1}:=m.
\]
As in the \(A\)-case, \(\sum_{\mathbf N}\) denotes summation over
\(N_1\geq\cdots\geq N_{r-1}\geq0\). The exponents in the
even and odd parity classes are
\begin{align*}
 E_0^B={}&n_1^2+2mn_1+2\sum_{a=1}^{r-1}N_a^2
 +2\sum_{a=j}^{r-2}N_a+m^2-m,\\
 E_1^B={}&n_1^2+(2m+1)n_1+2\sum_{a=1}^{r-1}N_a^2+m^2
 +2\sum_{a=1}^{j-1}N_a+4\sum_{a=j}^{r-2}N_a+2m
 +\frac{6r-4j-7}{4}.
\end{align*}
Euler's identity evaluates the \(n_1\)-sum as
\[
 \sum_{n_1\geq0}
 \frac{q^{n_1^2+(2m+\varepsilon)n_1}}{(q^2;q^2)_{n_1}}
 =(-q^{2m+\varepsilon+1};q^2)_\infty.
\]
Using the product decompositions
\begin{align*}
 \frac{(-q^{2m+1};q^2)_\infty}{(q;q)_{2m}}
 &=\frac{(-q;q^2)_\infty}
 {(q^2;q^2)_m(q^2;q^4)_m},\\
 \frac{(-q^{2m+2};q^2)_\infty}{(q;q)_{2m+1}}
 &=\frac{(-q^2;q^2)_\infty}
 {(1-q)(q^3;q^2)_m(q^4;q^4)_m},
\end{align*}
we obtain the even and odd contributions
\begin{align*}
 \mathscr S^{(0)}_{r,j}={}&(-q;q^2)_\infty
 \sum_{\mathbf N}
 \frac{q^{2N_1^2+\cdots+2N_{r-1}^2
 +2N_j+\cdots+2N_{r-2}+m^2-m}}
 {\displaystyle
  \prod_{a=1}^{r-2}(q^2;q^2)_{N_a-N_{a+1}}
  (q^2;q^4)_m(q^2;q^2)_m},\\
 \mathscr S^{(1)}_{r,j}={}&
 \frac{q^{(6r-4j-7)/4}(-q^2;q^2)_\infty}{1-q}
 \sum_{\mathbf N}
 \frac{q^{2N_1^2+\cdots+2N_{r-1}^2
 +2N_1+\cdots+2N_{j-1}
 +4N_j+\cdots+4N_{r-2}+2m+m^2}}
 {\displaystyle
  \prod_{a=1}^{r-2}(q^2;q^2)_{N_a-N_{a+1}}
  (q^3;q^2)_m(q^4;q^4)_m}.
\end{align*}
Equation~\eqref{eq:bailey-euler-g}, with \(q\) replaced by \(q^2\),
gives
\[
 \mathscr S^{(0)}_{r,j}=
 \frac{(-q;q^2)_\infty}{J_2}
 \left(J_{4(r+j-1),4(4r-3)}
       -q^{2j}J_{4(r-j-1),4(4r-3)}\right).
\]
Since \((-q;q^2)_\infty/J_2=1/J_{1,4}\), this is the first term on
the right-hand side of \eqref{eq:B-primary}.

To evaluate the odd contribution, replace $q$ first by $q^2$
and then by $-q$ in \eqref{eq:bailey-euler-f}. The multiple sum in
$\mathscr S^{(1)}_{r,j}$ becomes
\[
 (1-q)q^{2j-2r+2}
 \frac{\widetilde J_{2r-2j-1,4r-3}
       -\widetilde J_{2r-2j-2,4r-3}}{J_2}.
\]
We simplify the prefactor using
\[
 \frac{(-q^2;q^2)_\infty}{J_2}=\frac1{J_{2,4}},
 \qquad
 \frac{6r-4j-7}{4}+2j-2r+2=\frac{4j-2r+1}{4}.
\]
This yields
\[
 \mathscr S^{(1)}_{r,j}=
 q^{(4j-2r+1)/4}
 \frac{\widetilde J_{2r-2j-1,4r-3}
       -\widetilde J_{2r-2j-2,4r-3}}{J_{2,4}}.
\]
Adding the two parity contributions proves \eqref{eq:B-primary}.

For \eqref{eq:B-dual}, we identify its left-hand side directly
with the signed cancellation in Theorem~\ref{thm:nahm-fbar-hat}.
Work first at the nome $q^2$ and fix $t=n_r$. Euler's identity gives
\[
 \sum_{n\geq0}\frac{q^{n^2+2tn}}{(q^2;q^2)_n}
 =(-q^{2t+1};q^2)_\infty,
 \qquad
 \sum_{n\geq0}\frac{q^{n^2+2tn+2n}}{(q^2;q^2)_n}
 =(-q^{2t+3};q^2)_\infty.
\]
The relation
\[
 (-q^{2t+1};q^2)_\infty
 -q^{2t+1}(-q^{2t+3};q^2)_\infty
 =(-q^{2t+3};q^2)_\infty
\]
therefore implies
\[
 \sum_{n\geq0}
 \frac{q^{n^2+2tn}(1-q^{2n+2t+1})}{(q^2;q^2)_n}
 =\sum_{n\geq0}\frac{q^{n^2+2tn+2n}}{(q^2;q^2)_n}.
\]
Multiplying by the common factors in the remaining variables and
using $\mathbf b^{B^\vee}_{r,j}=\mathbf b^-_{r,j}+e_1$, we obtain
\begin{align*}
 &\sum_{\mathbf n\in\N^r}
 \frac{(-1)^{n_r}
 q^{\mathbf n^TB^\vee D^\vee\mathbf n+
                 2\mathbf n^T\mathbf b^{B^\vee}_{r,j}}}
 {\bigl(\prod_{a=1}^{r-1}(q^2;q^2)_{n_a}\bigr)
                    (q^4;q^4)_{n_r}}\\
 &\hspace{20mm}=\mathcal N^-_{r,j}(q^2)
                    -q\mathcal N^+_{r,j}(q^2).
\end{align*}
Absolute convergence justifies the summations. The product
identity \eqref{eq:sub-eta-Bdual} now evaluates the right-hand
side. Replacing $\tau$ by $\tau/2$, with all fractional powers
interpreted by the upper-half-plane convention, gives
\eqref{eq:B-dual}. At \(j=r-1\), the identities
\[
 J_{0,4(4r-3)}=\widetilde J_{0,4r-3}
 =J_{0,2r-3/2}=0
\]
remove the terms with zero first index. Together with the
empty-sum convention, this includes the smallest case
\((r,j)=(2,1)\).
\end{proof}

\subsection{Half-period translation in the transpose families}

Translations of expressions with rational powers are interpreted
in the upper-half-plane variable. In particular, replacing $q$ by
$-q$ means $\tau\mapsto\tau+\tfrac12$, so that
$\sigma=2\tau$ is replaced by $\sigma+1$.
The following congruences show that no change of the quadratic
matrix or the index vector is needed.

\begin{lemma}
\label{lem:half-period-closure}
For every $\mathbf n\in\Z^r$,
\begin{align}
 \mathbf n^TA^\vee D^\vee\mathbf n&\equiv n_1\pmod2,
 &\mathbf n^TB^\vee D^\vee\mathbf n&\equiv n_1+n_r\pmod2.
 \label{eq:dual-parity}
\end{align}
For $\mathsf A\in\{A^\vee,B^\vee\}$,
$\mathbf u\in\Z^r$, $c\in\Q$, and
$\boldsymbol\varepsilon\in\{1,-1\}^r$, define
\[
 F_{\boldsymbol\varepsilon}(\sigma)
 :=\sum_{\mathbf n\in\N^r}
 \frac{\bigl(\prod_{a=1}^r\varepsilon_a^{n_a}\bigr)
 e^{2\pi i\sigma(\frac12\mathbf n^T\mathsf A D^\vee\mathbf n+
                           \mathbf n^T\mathbf u+c)}}
 {\prod_{a=1}^r
   (e^{2\pi i d_a^\vee\sigma};e^{2\pi i d_a^\vee\sigma})_{n_a}}.
\]
Then
\[
 F_{\boldsymbol\varepsilon}(\sigma+1)
 =e^{2\pi ic}F_{\boldsymbol\varepsilon'}(\sigma),
\]
where $\varepsilon'_1=-\varepsilon_1$ and
$\varepsilon'_a=\varepsilon_a$ for $2\leq a\leq r-1$, while
$\varepsilon'_r=\varepsilon_r$ for $\mathsf A=A^\vee$ and
$\varepsilon'_r=-\varepsilon_r$ for $\mathsf A=B^\vee$.
\end{lemma}

\begin{proof}
The quadratic-form identity \eqref{eq:BD-dual} gives
\[
 \mathbf n^TB^\vee D^\vee\mathbf n
 =n_1^2+2n_1n_r+
  2\sum_{a=1}^{r-1}(n_{a+1}+\cdots+n_r)^2+n_r^2.
\]
Since $n_a^2\equiv n_a\pmod2$, this proves the second
congruence. Subtracting $n_r^2$ proves the first. Each denominator
in $F_{\boldsymbol\varepsilon}$ is unchanged by
$\sigma\mapsto\sigma+1$, because $d_a^\vee$ is an integer.
The numerator acquires the factor
\[
 e^{2\pi ic}e^{2\pi i\mathbf n^T\mathbf u}
 e^{\pi i\mathbf n^T\mathsf A D^\vee\mathbf n}.
\]
The middle factor is $1$ because $\mathbf u$ is integral, and
the remaining sign is given by \eqref{eq:dual-parity}.
Absolute convergence permits this termwise calculation.
\end{proof}

\subsection{Recovery of the four product vectors}
\label{subsec:component-recovery}

The remaining task is to recover each prescribed component from the
available product evaluations. We combine the preceding identities
with the \(A\)-, \(B\)-, and \(A^\vee\)-evaluations in
\cite[Theorem~1.5(1)--(6)]{DuJiShenXu2026} and the signed
\(B^\vee\)-evaluation in
\cite[Proposition~1.11 and Theorem~1.12]{DuJiShenXu2026}.
For the \(A\)-families, the known product evaluations are denoted by
\begin{equation*}
\begin{aligned}
 f_j(\tau)&:=
 q^{(3r-2j-1)/2}\frac{J_{4j,16r-4}}{J_{1,4}}
 +\frac{\widetilde J_{2j,4r-1}}{J_{2,4}},
 \\
 f_j^\vee(2\tau)&:=\frac{J_{2j,4r-1}}{J_{1,4}}.
\end{aligned}
\end{equation*}
Here \(1\leq j\leq r\), and the rank \(r\) is suppressed in the
notation. The corresponding \(B\)-family evaluations are
\begin{equation*}
\begin{aligned}
 y_j(\tau)&:=
 q^{(6r-4j-5)/4}\frac{J_{4j,4(4r-3)}}{J_{1,4}}
 +\frac{\widetilde J_{2j,4r-3}}{J_{2,4}},
 \\
 y_j^\vee(2\tau)&:=\frac{J_{2j,4r-3}}{J_{1,4}}.
\end{aligned}
\end{equation*}
Here \(1\leq j\leq r-1\). The cited evaluations represent
\(f_j,f_j^\vee,y_j,y_j^\vee\) by generalized Nahm sums with
matrices \(A,A^\vee,B,B^\vee\), respectively; the last family is
signed. For the new evaluations \eqref{eq:A-primary},
\eqref{eq:A-dual}, \eqref{eq:B-primary}, and \eqref{eq:B-dual}, use
the notation
\begin{align*}
 F_j(\tau)&:=
 \frac{J_{4r+4j,16r-4}-q^{2j}J_{4r-4j,16r-4}}{J_{1,4}}
 +q^{j-r/2+1/2}
 \frac{\widetilde J_{2r+2j,4r-1}
       -\widetilde J_{2r-2j,4r-1}}{J_{2,4}},
 \\
 F_j^\vee(2\tau)&:=
 q^{2j-2r+1}
 \frac{J_{2r-2j,4r-1}-J_{2r+2j,4r-1}}{J_{1,4}},
 \\
 Y_j(\tau)&:=
 \frac{J_{4(r+j-1),4(4r-3)}
       -q^{2j}J_{4(r-j-1),4(4r-3)}}{J_{1,4}}
 \notag\\[-2pt]
 &\hspace{16mm}+
 q^{(4j-2r+1)/4}
 \frac{\widetilde J_{2r-2j-1,4r-3}
       -\widetilde J_{2r-2j-2,4r-3}}{J_{2,4}},
 \\
 Y_{r,j}^\vee(2\tau)&=
 q^{2j-2r+2}
 \frac{J_{2r-2j-1,4r-3}
       -J_{2r-2j-2,4r-3}}{J_{1,4}}.
\end{align*}

For \(F_j\) and \(F_j^\vee\), the range is \(1\leq j\leq r\);
for \(Y_j\) and \(Y_{r,j}^\vee\), it is \(1\leq j\leq r-1\).
The dual expressions are obtained by replacing \(q\) by \(q^2\)
in \eqref{eq:A-dual} and \eqref{eq:B-dual}. In this notation the cancellation identity of the first theorem is
\begin{equation}
 Y_{r,j}^\vee(2\tau)
 =\mathcal N^-_{r,j}(q^2)-q\mathcal N^+_{r,j}(q^2).
 \label{eq:Yvee-Npm}
\end{equation}

\begin{proof}[Proof of Theorem~\ref{thm:four-families}]
The bars and tildes on the auxiliary products below are labels,
not complex conjugation.

For the \(A\)-family, define
\[
 \widetilde g_i(\tau):=
 q^{(2i-r-1)/2}\frac{J_{8r-4i,16r-4}}{J_{1,4}}
 +\frac{\widetilde J_{2i-1,4r-1}}{J_{2,4}}.
\]
Substitution of the product formulas, using
\(J_{a,m}=J_{m-a,m}\) where necessary, gives
\begin{equation*}
 \widetilde g_i(\tau)=
 \begin{cases}
  q^{(1-r)/2}F_r(\tau),&i=1,\\[2pt]
  q^{(2i-r-1)/2}F_{r-i}(\tau)+f_i(\tau),
    &2\leq i\leq r-1,\\[2pt]
  f_{2r-i}(\tau),&r\leq i\leq2r-1.
 \end{cases}
\end{equation*}
Since
\[
 g_i(\tau)=q^{(4r-4i+1)^2/(32r-8)}\widetilde g_i(\tau),
\]
each \(g_i\) is a finite combination of the \(A\)-family Nahm
sums in the sense of Theorem~\ref{thm:four-families}.
In particular, the quadratic matrix and symmetrizer are unchanged.

For the \(A^\vee\)-family, consider the product quotient
\[
 \overline g_i^\vee(2\tau):=
 \frac{J_{2r-i,4r-1}}{J_{1,4}}.
\]
The symmetry \(J_{a,m}=J_{m-a,m}\) gives
\begin{equation}
 \overline g_i^\vee(2\tau)=
 \begin{cases}
  f_r^\vee(2\tau),&i=1,\\[2pt]
  f_{r-k}^\vee(2\tau),&i=2k,\quad1\leq k\leq r-1,\\[2pt]
 f_{r-k}^\vee(2\tau)-q^{2r-i}F_k^\vee(2\tau),
     &i=2k+1,\quad1\leq k\leq r-2,\\[2pt]
  f_1^\vee(2\tau)-qF_{r-1}^\vee(2\tau),&i=2r-1.
 \end{cases}
 \label{eq:combination-Adual}
\end{equation}
For a product quotient \(P(2\tau)\), the notation \(P(2\tau;Q)\)
denotes the same expression with \(q\) replaced by \(Q\). Fractional
powers retain the upper-half-plane convention. Then
\begin{align*}
 g_i^\vee(2\tau)
 =\frac12q^{(2i^2-2i-2r+1)/(16r-4)}
 \left(
  \overline g_i^\vee(2\tau;q)
  +(-1)^{\binom i2}\overline g_i^\vee(2\tau;-q)
 \right).
\end{align*}
With \(\sigma=2\tau\), the replacement \(q\mapsto-q\) means
\(\tau\mapsto\tau+\tfrac12\), or equivalently
\(\sigma\mapsto\sigma+1\). Lemma~\ref{lem:half-period-closure}
shows that this changes only the phases and signs of the dual
Nahm summands, not their matrix \(A^\vee\) or symmetrizer
\(D^\vee\). Apply \eqref{eq:combination-Adual} at both nomes,
including the change of every explicit monomial prefactor. This gives the required combination of $A^\vee$-Nahm sums.
All dual linear coefficients in these identities, including those
in the cited evaluations, are integral. The translation lemma
therefore absorbs every half-period translate into signs at the
same nome $q^2$.  This
includes the previously unresolved odd indices under the stated
half-period convention; the case
\(i=2r-1\) was obtained in \cite[Remark~1.8(2)]{DuJiShenXu2026}.

For the \(B\)-family, set
\[
 \overline h_i(\tau):=
 q^{(4i-2r-1)/4}
 \frac{J_{4(2r-i-1),4(4r-3)}}{J_{1,4}}
 +\frac{\widetilde J_{2(2r-i-1),4r-3}}{J_{2,4}}.
\]
Then
\begin{equation*}
 \overline h_i(\tau)=
 \begin{cases}
  q^{(3-2r)/4}Y_{r-1}(\tau),&i=1,\\[2pt]
  q^{(4i-2r-1)/4}Y_{r-i}(\tau)+y_{i-1}(\tau),
    &2\leq i\leq r-2,\\[2pt]
  y_{2r-i-1}(\tau),&\max\{2,r-1\}\leq i\leq2r-2.
 \end{cases}
\end{equation*}
Because
\[
 h_i(\tau)=q^{(4i-4r+1)^2/(32r-24)}\overline h_i(\tau),
\]
the displayed cases recover every component using the prescribed
\(B\)-family sums and their monomial multiples.

Finally, for the \(B^\vee\)-family, let
\[
 \overline h_i^\vee(2\tau):=
 \frac{J_{2r+i-2,4r-3}}{J_{1,4}}.
\]
Then
\begin{equation}
 \overline h_i^\vee(2\tau)=
 \begin{cases}
  y_{r-k-1}^\vee(2\tau),
     &i=2k+1,\quad0\leq k\leq r-2,\\[2pt]
  q^{2r-2k-2}Y_{r,k}^\vee(2\tau)+y_{r-k-1}^\vee(2\tau),
     &i=2k,\quad1\leq k\leq r-2,\\[2pt]
  Y_{r,r-1}^\vee(2\tau),&i=2r-2.
 \end{cases}
 \label{eq:combination-Bdual}
\end{equation}
With the same convention for the nome, we have
\[
 h_i^\vee(2\tau)
 =\frac12q^{(i^2-i-r+1)/(8r-6)}
 \left(
  \overline h_i^\vee(2\tau;q)
  +\overline h_i^\vee(2\tau;-q)
 \right).
\]
Apply \eqref{eq:combination-Bdual} at \(q\) and \(-q\), again
translating the complete expressions, and use
Lemma~\ref{lem:half-period-closure}. Every component is thereby
expressed as a finite combination of signed \(B^\vee\)-Nahm sums
at the nome $q^2$: the translation lemma again absorbs the
half-period translates into signs, with matrix $B^\vee$ and
symmetrizer $D^\vee$ retained. Equivalently, the signed sums can
be expressed as differences of parity-restricted sums. The even
indices \(2\leq i\leq2r-4\) supply the range requested in
\cite[Problem~1.13(d)]{DuJiShenXu2026}. Every monomial prefactor has rational exponent, so it is
absorbed into the constant parameter $c$ in the multiple sum;
any phase produced by a translation is part of the scalar
coefficient. This proves the precise representation class stated
in Theorem~\ref{thm:four-families} for all four vectors.
\end{proof}

\end{document}